%% file: main.tex
\documentclass[11pt,a4paper]{article}
\usepackage[margin=1in,a4paper]{geometry}

\usepackage[utf8]{inputenc}

\usepackage{amsmath}
\usepackage{amsthm}
\usepackage{amsfonts}
\usepackage{amssymb}
\usepackage{stmaryrd}
\usepackage{graphicx}
\usepackage{subfig}

\usepackage{enumitem}
\usepackage{multicol}

\usepackage[dvipsnames]{xcolor}
\newcommand\myshade{70}
\colorlet{mylinkcolor}{violet}
\colorlet{mycitecolor}{YellowOrange}
\colorlet{myurlcolor}{Aquamarine}

\usepackage{hyperref}
\hypersetup{
    linkcolor  = mylinkcolor!\myshade!black,
    citecolor  = mycitecolor!\myshade!black,
    urlcolor   = myurlcolor!\myshade!black,
    colorlinks = true,
}
\usepackage{cleveref}

\usepackage{tikz}

\input{macroses/enviroments}
\input{macroses/macroses}

\usepackage{bbm}
\usepackage{dsfont}

\usepackage{algorithm}
\usepackage{algorithmic} 
\usepackage[algo2e]{algorithm2e}

\usepackage[intoc,english]{nomencl}
\makenomenclature
\usepackage{etoolbox}

\renewcommand\nomgroup[1]{%
  \vspace{0em}\item[\bfseries
  \ifstrequal{#1}{A}{General}{%
  \ifstrequal{#1}{B}{Geometry}{%
  \ifstrequal{#1}{C}{Geometric and Statistical Model}{%
  \ifstrequal{#1}{D}{Statistical Queries}{%
  \ifstrequal{#1}{E}{Matrix Calculus}{%
  \ifstrequal{#1}{F}{Lower Bounds}{%
  \ifstrequal{#1}{G}{Other Symbols}{}}}}}}}%
]}

\title{Adversarial Manifold Estimation}

\date{}
\author{
    Eddie Aamari \\
    {\small \texttt{aamari@lpsm.paris}} \\
    {\small LPSM, CNRS} \\
    {\small Universit\'e Paris Cité, Sorbonne Université} \\
    {\small Paris, France}
    \and 
    Alexander Knop \\
    {\small \texttt{aknop@ucsd.edu}} \\
    {\small Department of Mathematics} \\
    {\small University of California, San Diego} \\
    {\small La Jolla, CA, USA}
}

\begin{document}

    \maketitle
    
    \begin{abstract}
        \input{parts/abstract}
    \end{abstract}

\setcounter{tocdepth}{3}
\tableofcontents

    \section{Introduction}
    \input{parts/introduction}

    \section{Preliminaries}
    \input{parts/preliminaries}
        
    \section{Manifold Propagation Algorithm}
    \label{sec:mpa}
    \input{parts/mpa}

    \section{Geometric Routines with Statistical Queries}
    \label{sec:SQ-routines}
    \input{parts/SQ-routines}

    \section{Manifold Estimation with Statistical Queries}
    \label{sec:SQ-manifold-estimation}
    \input{parts/SQ-manifold-estimation}    

    \section{Further Directions}
    \label{sec:conclusion}
    \input{parts/outro}

    \appendix

    \section{Proofs of the Properties of Manifold Propagation}
    \label{sec:mpa-tech}
    \input{parts/mpa-technical-lemmas}

    \section{Preliminary Geometric Results}
    \label{sec:misc}
    \input{parts/miscellaneous}

    \section{Projection Routine}
    \label{sec:SQ-projection-appendix}
    \input{parts/projection}

    \section{Tangent Space Estimation Routine}
    \label{sec:SQ-tangent-appendix}
    \input{parts/tangent}

    \section{Seed Point Detection}
    \label{sec:SQ-seed-appendix}
    \input{parts/seed}

    \section{Proof for the Main Statistical Query Manifold Estimators}
    \label{sec:SQ-upper-bounds}
    \input{parts/SQ-upper-bounds}

    \section{Statistical Query Lower Bounds in Metric Spaces}
    \label{sec:lower-bounds}
    \input{parts/lower-bounds}

    \section{Lower Bounds for Manifold Models}
    \label{sec:manifold-lower-bound-constructions}
    \input{parts/manifold-lower-bound-constructions}

    \newpage

\input{parts/notation}
    \printnomenclature

    \bibliographystyle{alpha}
    \bibliography{biblio}

\end{document}

%% file: macroses/enviroments.tex
\theoremstyle{plain}
\newtheorem{theorem}{Theorem}[section]
\newtheorem{lemma}{Lemma}[section]
\newtheorem{proposition}{Proposition}[section]

\theoremstyle{definition}
\newtheorem{definition}{Definition}[section]
\newtheorem{remark}{Remark}[section]

%% file: macroses/macroses.tex
\input{macroses/sets}

\newcommand{\proofof}{Proof of\xspace}

\newcommand{\ceil}[1]{\left\lceil#1\right\rceil}
\newcommand{\floor}[1]{\left\lfloor#1\right\rfloor}

\input{macroses/probability}
\input{macroses/norms}

\newcommand{\model}{\mathcal{D}}
\newcommand{\parameterspace}{\mathcal{M}}
\newcommand{\functional}{\theta}
\newcommand{\metric}{\Theta}
\newcommand{\dist}{\rho}
\newcommand{\Bmetric}{\B_{(\metric,\dist)}}

\newcommand{\rank}{\mathrm{rank}}
\newcommand{\tr}{\mathrm{tr}}
\newcommand{\transpose}[1]{{#1}^\top}

\newcommand{\sampling}{\mathcal{L}}
\newcommand{\samplingPauli}{\sampling_{\mathrm{Pauli}}}

\newcommand{\opt}{\text{opt}}

\newcommand{\Length}{\mathrm{Length}}
\newcommand{\Med}{\mathrm{Med}}
\newcommand{\rch}{\mathrm{rch}}

\newcommand{\oracle}{\mathsf{O}}
\newcommand{\STAT}{\mathrm{STAT}}
\newcommand{\answer}{\mathrm{a}}
\newcommand{\Algorithm}[1]{\mathtt{#1}}
\newcommand{\RandmizedAlgorithm}[1]{\mathcal{#1}}
\newcommand{\ttau}{\Tilde{\tau}}

\newcommand{\MPA}{\hyperref[alg:MPA]{\texttt{Manifold Propagation}}\xspace}
\newcommand{\ABS}{\hyperref[alg:routine_detection_raw_point]{\texttt{SQ Ambient Binary Search}}\xspace}


%% file: macroses/sets.tex
\newcommand{\set}[1]{\left\{{#1}\right\}}
\newcommand{\argmin}{\operatornamewithlimits{argmin}}
\newcommand{\argmax}{\operatornamewithlimits{argmax}}

\newcommand{\C}{\mathbb{C}}
\newcommand{\R}{\mathbb{R}}
\newcommand{\Z}{\mathbb{Z}}

\newcommand{\B}{\mathrm{B}}
\newcommand{\Sphere}{\mathcal{S}}

\newcommand{\diam}{\mathrm{diam}}

\newcommand{\Gr}[2]{\mathbb{G}^{#1,#2}}

\newcommand{\dd}{\mathrm{d}}
\newcommand{\dHaus}{\mathrm{d_H}}

\newcommand{\manifolds}[3]{\mathcal{M}^{{#1}, {#2}}_{#3}}
\newcommand{\manifoldspoint}[4]{\set{#4} \sqcup \mathcal{M}^{{#1}, {#2}}_{#3}}
\newcommand{\manifoldsball}[4]{\B(0,#4) \sqcap \mathcal{M}^{{#1}, {#2}}_{#3}}

\newcommand{\distributions}[6]{\mathcal{D}^{{#1}, {#2}}_{#3}({#4}, {#5}, {#6})}
\newcommand{\distributionspoint}[7]{\set{#7} \sqcup \mathcal{D}^{{#1}, {#2}}_{#3}({#4}, {#5}, {#6})}
\newcommand{\distributionsball}[7]{\B(0,#7) \sqcap \mathcal{D}^{{#1}, {#2}}_{#3}({#4}, {#5}, {#6})}

%% file: macroses/probability.tex
\newcommand{\supp}{\mathrm{Supp}}
\newcommand{\E}{\mathop{\mathbb{E}}}
\newcommand{\indicator}[1]{\mathbbm{1}_{#1}}

\newcommand{\TV}{\mathop{\mathrm{TV}}}

\newcommand{\Haus}{\mathcal{H}}

\newcommand{\CV}{\mathrm{cv}}
\newcommand{\PK}{\mathrm{pk}}

%% file: macroses/norms.tex
\newcommand{\inner}[2]{\left\langle {#1}, {#2} \right\rangle}
\newcommand{\norm}[1]{\left\Vert #1 \right\Vert }

\newcommand{\normop}[1]{\left\Vert #1 \right\Vert_{\mathrm{op}} }
\newcommand{\normF}[1]{\left\Vert #1 \right\Vert_{\mathrm{F}} }
\newcommand{\normNuc}[1]{\left\Vert #1 \right\Vert_{\ast} }

%% file: parts/abstract.tex
This paper studies the statistical query (SQ) complexity of estimating
$d$-dimensional submanifolds in $\mathbb{R}^n$.
We propose a purely geometric algorithm called Manifold Propagation,
that reduces the problem to three natural geometric
routines: projection, tangent space estimation, and point 
detection. We then provide constructions of these geometric routines in the SQ framework.
Given an adversarial $\mathrm{STAT}(\tau)$ oracle and a target Hausdorff distance precision $\varepsilon = \Omega(\tau^{2 / (d + 1)})$, the resulting SQ manifold reconstruction algorithm has query complexity $\tilde{O}(n \varepsilon^{-d / 2})$, which is proved to be nearly optimal.
In the process, we establish low-rank matrix completion results for SQ's and lower bounds for randomized SQ estimators in general metric spaces.

%% file: parts/introduction.tex
In the realm of massive data acquisition, the curse of dimensionality phenomena led to major developments of computationally efficient statistical and machine learning techniques. 
Central to them are topological data analysis and geometric methods, which have recently garnered a lot of attention and proved fruitful in both theoretical and 
applied areas~\cite{Wasserman18}. These realms refer to a  collection of statistical methods that find intrinsic structure in data. In short, this field is based upon the idea that data described with a huge number of features $n$ may be subject to redundancies and correlations, so that they include only $d \ll n$ intrinsic and informative degrees of freedom. 
This low-dimensional paradigm naturally leads to the problem of recovering this intrinsic structure, for data visualization or to mitigate the curse of dimensionality. This problem is usually referred to as \emph{support estimation}~\cite{Cuevas97} or \emph{dimensionality reduction}~\cite{Lee07}.

Linear dimensionality reduction techniques, such as Principal Component Analysis and LASSO-types methods~\cite{Hastie09}, assume that the data of interest lie on a low-dimensional linear subspace. 
This assumption appears to be often too strong in practice, so that one may use problem-specific featurization techniques, or other
non-linear methods. 
On the other hand, non-linear dimensionality reduction techniques such as Isomap~\cite{Ten97}, Local Linear Embedding~\cite{Row00} and Maximum Variance Unfolding~\cite{Pelletier13}, work under the relaxed assumption that the data of interest lie on an embedded $d$-dimensional manifold of $\R^n$ with $d \ll n$, hence allowing natively for non-linearities.

\subsection{Context}

\paragraph{Geometric Inference from Samples.}

The classical statistical framework, based on data points, is usually referred to as \emph{PAC-learning}~\cite{Val84} or \emph{sample framework}. In this setting, the learner is given a set $\set{x_1, \dots, x_s}$ of $s$ samples drawn, most commonly independently, from an unknown distribution $D$. From these samples, the learner then aims at 
estimating a parameter of interest $\functional(D)$, which in our context will naturally be taken as the support $\functional(D) = \supp(D) \subseteq \R^n$. As described above, for this problem to make sense, $D$ shall be assumed to be supported on (or near) a low-dimensional structure, i.e. a $d$-dimensional submanifold with $d \ll n$. Here, the precision is usually measured with the Hausdorff distance, a $L^\infty$-type distance between compact subsets of~$\R^n$.

The Hausdorff estimation of manifolds in the noiseless sample framework is now well understood. The first minimax manifold estimation results in the sample framework are due to~\cite{Genovese12,Genovese12b}. At the core of their work is the reach, a scale parameter that quantitatively encodes $\mathcal{C}^2$ regularity of the unknown manifold $M = \supp(D) \subseteq \R^n$, and that allows measuring the typical scale at which $M$ looks like $\R^d$~\cite{Aamari19} (see \Cref{def:reach}). Under reach assumptions, the estimator of~\cite{Genovese12b} achieves a worst-case average precision at most $O((\log s / s)^{2 / d})$, but with a computationally intractable method. 
This rate was later shown to be $\log$-tight optimal by~\cite{Kim2015}. Given a target precision $\varepsilon>0$, these results hence reformulate to yield sample complexity of order $s = O(\varepsilon^{-d / 2} \log(1 / \varepsilon))$. 
Later,~\cite{Aamari18} gave a constructive estimator combining local PCA and the computational geometry algorithm from~\cite{Boissonnat14}, which outputs a triangulation of the sample points in polynomial time, and linear time in the ambient dimension $n$~\cite{Boissonnat14}. 
More recently,~\cite{Divol20} proposed a computationally tractable minimax adaptive method that automatically selects the intrinsic dimension and reach.
Let us also mention that by using local polynomials, faster sample rates can also be achieved when the manifolds are smoother than $\mathcal{C}^2$~\cite{Aamari19b}.

Manifold estimation in the presence of noise is by far less understood.
The only known optimal statistical method able to handle samples corrupted with additive noise~\cite{Genovese12} is 
intractable.
Currently, the best algorithmically tractable estimators in this context require either the noise level to vanish fast enough as the sample size grows~\cite{Puchkin21}, or very specific distributional assumptions on the noise: either Gaussian~\cite{Fefferman19,Dunson21} or ambient uniform~\cite{Aizenbud21}.
To date, the only computationally tractable sample method that truly is robust to some type of noise is due to~\cite{Aamari18}, in which the authors consider the so-called \emph{clutter noise} model introduced by~\cite{Genovese12b}. 
In this model, the samples are generated by a mixture of a distribution $D$ on $M$ and a \emph{uniform}
distribution in the ambient space, with respective coefficients $\beta \in (0,1]$ and
$(1 - \beta)$. That is, the $s$-sample consists of unlabelled points in $\R^n$,
with approximately $\beta s$ informative points on $M$ and $(1 - \beta)s$ non-informative
ambient clutter points. In~\cite{Genovese12b,Kim2015}, the optimal sample complexity was
shown to be $s = O(\beta^{-1}\varepsilon^{-d/2}\log(1/\varepsilon))$, but this rate was
obtained with an intractable method. 
On the other hand,~\cite{Aamari18} proposed to label the non-informative data points, which allows to apply a clutter-free estimation procedure to the remaining decluttered points. 
This results in a computationally tractable minimax optimal method, with an additional computational cost due to the decluttering. However, the success of this extra step relies heavily on the assumption that the ambient clutter is uniform.

Overall, the existing reasonable manifold estimation methods are heavily attached to the \emph{individual} data points and do not tolerate much noise, as the change of a single point may have the method fail completely. 
Let us mention from now that in sharp contrast, the statistical query framework considered in this work is inherently robust to clutter noise without an artificial decluttering, no matter the clutter noise distribution (see \Cref{rk:clutter_noise}).

\paragraph{Private Learning.}
\label{par:private-learning}
Beyond the classical sample complexity, the modern practice of statistics raised concerns leading to quantitative and qualitative estimation constraints~\cite{Wainwright14}.
For instance, in many applications of learning methods, the studied data is contributed by
individuals, and features represent their (possibly) private characteristics such as
gender, race, or health history. Hence, it is essential not to reveal too much information about any particular individual.
The seminal paper~\cite{KLNRS11:what-can-we-learn-privately} on this topic introduces 
the notion of private learning, 
a learning framework inspired by differentially private algorithms~\cite{DMNS16}.
Given samples $\set{x_1,\dots,x_s}$, this framework imposes privacy to a learner by requiring it not to be significantly affected if a particular sample $x_i$ is replaced with an arbitrary $x'_i$. 

In contrast to precision, which is analyzed with respect 
to a model, the level of differential privacy is a worst-case notion.
Hence, when analyzing the privacy guarantees of a learner, no assumption should be made on the underlying generative model. Indeed, such a generative assumption could fall apart in the presence of background knowledge that the adversary might have: 
conditioned on this background knowledge, the model may change drastically.

There are two main types of differentially private algorithms.
\emph{Global differential privacy} assumes that there is a trusted entity (i.e. a
central data aggregator) that can give private answers to database queries~\cite{KLNRS11:what-can-we-learn-privately}. This
approach was used by LinkedIn to share advertisements data~\cite{RSPDKSA20}, 
by Uber’s system for internal analytics~\cite{JNS18}, and is about to be implemented 
by the U.S. Census Bureau for publishing data~\cite{CB20:census-is-private}.

In contrast, \emph{local differential privacy}, as
defined by~\cite{EGS03,KLNRS11:what-can-we-learn-privately}, 
even further restricts the learners.
It requires that even if an adversary
has access to the personal revealed data of individuals, this adversary will still
be unable to learn too much about the user's actual personal data.
The simplest example of a locally private learning protocol was originally introduced to encourage truthfulness in surveys~\cite{Warner:randomized-response}.
In local differential privacy, a trusted entity is not necessarily present, and
each individual protects their own privacy, for instance, by adding noise to their 
data separately.

\paragraph{Statistical Queries.}
Instead of sample complexity, this paper considers the notion of statistical query  complexity, which was proved to be equivalent to the locally private learning complexity up
to a polynomial factor~\cite{KLNRS11:what-can-we-learn-privately}, 
and that naturally enforces robustness to clutter noise (see \Cref{rk:clutter_noise}).

First introduced by Kearns~\cite{Kearns:learning-from-SQs}, the statistical query (SQ)
framework is a restriction of PAC-learning, where the learner is only allowed to 
obtain approximate averages of the unknown distribution $D$ via an adversarial 
oracle, but cannot see any sample. That is, given a tolerance parameter $\tau > 0$,
a $\STAT(\tau)$ oracle for the distribution $D$ accepts functions 
$r : \R^n \to [-1, 1]$ as queries from the learner, and can answer \emph{any} value $\answer \in \R$
such that $|\!\E_{x \sim D}[r(x)] - \answer| \le \tau$. 
Informally, the fact that the oracle is adversarial is the counterpart to the fact that differential privacy is a worst-case notion.
We emphasize that in
the statistical query framework, estimators (or learners) are only given access to such an oracle, and not to the data themselves.
Limiting the learner's accessible information to adversarially perturbed averages both restricts the range of the usable algorithms,
and effectively forces them to be robust and efficient.

Kearns showed that any statistical query learner can be transformed into a classical PAC-learning algorithm with robustness to random classification noise~\cite{Kearns:learning-from-SQs}. 
Conversely, many commonly used PAC-learning algorithms have statistical query implementations~\cite{%
    Kearns:learning-from-SQs,%
    Bylander:learning-linear-threshold,%
    DV:rescaling-for-linear-programs}.
Though, Kearns also showed that there are information-theoretic 
obstacles that are specific to the statistical query framework; e.g. parity functions 
require an exponential number of queries~\cite{Kearns:learning-from-SQs}.
In other words, PAC-learning is strictly stronger than SQ-learning.

We have already mentioned the connection between statistical queries and private learning. 
On top of this, the simplicity of the SQ framework allowed its application in several other fields, such as (theoretical and practical) learning algorithms for distributed data systems. Indeed, a problem 
has an efficient SQ algorithm if and only if it has an efficient distributed learner~\cite{BD98,SVW16}.
Another incentive to study statistical queries arises from quantum computations: general quantum PAC learners can perform complex entangling measurements which do not seem realizable for
near-term quantum computers. 
To overcome this issue, Arunachalam, Grilo, and Yuen~\cite{AGY20:quantum-SQ} introduced the notion of quantum statistical query learning, for which practical implementations would only require to measure a single quantum state at a time.

Overall, certainly the most interesting property of statistical query algorithms
is the possibility of proving unconditional lower bounds on the complexity
of statistical problems. Considering the number of learning algorithms that
are implementable in the statistical query framework, these lower bounds provide strong
evidence of hardness of these problems.
Moreover, for many learning problems, the known unconditional lower bounds for the statistical query framework closely match the known computational complexity upper bounds.
For instance,~\cite{BFJKMR94} proved that SQ algorithms require a quasi-polynomial 
number of queries to learn disjunctive normal forms (DNF), which matches the running time upper bound by 
Verbeurgt~\cite{Ver90}. Similar results were proved by~\cite{FGRVX17} for the planted
clique problem, and by~\cite{DKS17} for high-dimensional Gaussian mixtures learning. Finally, some problem-specific 
statistical query lower bounds directly imply 
lower bounds against general convex relaxations of Boolean constraint satisfaction 
problems~\cite{FPV18,Feldman17}, lower bounds on approximation of Boolean functions by 
polynomials~\cite{DFTWW15}, and lower bounds on dimension complexity of Boolean function
classes~\cite{She08,Feldman17}.

\subsection{Contribution}
This paper is the long and complete version of~\cite{Aamari21}. 
We establish nearly matching upper and lower bounds on the statistical query complexity
of manifold learning. As a corollary, it provides an efficient and natural 
noise-tolerant sample manifold estimation technique; as another side-product,
it also provides, to the best of our knowledge, the first private manifold estimation method.
In some regimes of the parameters, it also exhibits another example of a natural statistical problem with different sample and statistical query complexities.

\subsubsection{Main Results}

\paragraph{Upper Bounds.}
The main contribution of this paper is the construction of a low-complexity 
\emph{deterministic} SQ algorithm that uniformly estimates the compact connected $d$-dimensional 
$\mathcal{C}^2$-manifolds $M \subseteq \R^n$ with reach $\rch_M \geq \rch_{\min} > 0$
(i.e. curvature roughly bounded by  $1 / \rch_{\min}$), from distributions $D$ with
support $\supp(D) = M$ that have a Lipschitz density function bounded below by $f_{\min} > 0$ on $M$ (i.e. volume of $M$ roughly bounded by  $1 / f_{\min}$). See \Cref{def:manifold_model_distribution} for a formal definition.
The estimation error is measured in Hausdorff distance, 
which plays the role of a sup-norm between compact subsets of $\R^n$.

In \Cref{prop:SQ_lower_bound_infinity_unbounded}, we prove that without any prior information about the location of the manifolds, SQ algorithms cannot estimate them, even with an unlimited number of queries.
It is worth noting that this phenomenon is specific to the SQ framework and does not occur in the sample framework.
We consider two ways to ``localize'' the model. Namely, we either assume: 
that the manifold contains the origin (fixed point model), 
or that the manifold lies within the ball of radius $R > 0$ centered at the origin (bounding ball model).
\begin{description}
    \item[{[Fixed point model]}] 
        \Cref{thm:SQ_upper_bound_point} presents a \emph{deterministic} algorithm which, given the information that $0 \in M$, achieves
        precision $\varepsilon$ using 
        \[
            q = O\left(
                \frac{n \operatorname{polylog}(n)}{f_{\min}}
                \left(\frac{1}{\rch_{\min} \varepsilon}\right)^{d / 2}
            \right)
        \]
        queries to a $\STAT(\tau)$ oracle, provided that 
        $$
            \varepsilon 
            =
            \Omega\left(
            \rch_{\min}
                \biggl(
                    \frac{\tau}{f_{\min}\rch_{\min}^d}
                \biggr)^{2 / (d + 1)}
            \right)
        ,
        \text{ and }
        \tau = O(f_{\min} \rch_{\min}^d)
        .
        $$

    \item[{[Bounding ball model]}]
        \Cref{thm:SQ_upper_bound_ball} shows that the same estimation problem can still be solved using $O(n \log(R / \varepsilon))$ extra queries to $\STAT(\tau)$ if $M$ is only assumed to be contained in the ball $\B(0,R)$.
        That is, it shows that manifold estimation with precision $\varepsilon$ in the bounding ball model can be done deterministically with
        \[
            q
            =
            O\left(                
                n\log\left(\frac{R}{\varepsilon}\right) + 
                \frac{n\operatorname{polylog}(n)}{f_{\min}}
                \left(\frac{1}{\rch_{\min} \varepsilon}\right)^{d / 2}
            \right)
        \]
        queries to a $\STAT(\tau)$ oracle, under similar conditions as above.
\end{description}
Notice the limited quasi-linear dependency on the ambient dimension $n$ in these bounds.
Actually, in the fixed point model, the given query complexity corresponds to the
sample complexity up to the $n \operatorname{polylog}(n)$ factor~\cite{Kim2015,Divol20}.
This apparent discrepancy can be explained by the fact that a single sample of
$M \subseteq \R^n$ corresponds to $n$ coordinates, while statistical queries are
forced to be real-valued.
More interestingly, the extra cost $O(n \log(R / \varepsilon))$ in the bounding ball 
model is specific to the statistical query framework and does not appear in the sample 
framework~\cite{Kim2015}, although this term would dominate only in the regime where $R$ is exponentially bigger than $\rch_{\min}$.

The insights underlying these upper bounds are described in 
\Cref{section:intro-manifold-propagation,section:intro-SQ-implementations}, and the 
formal statements in \Cref{sec:mpa,sec:SQ-routines,sec:SQ-manifold-estimation}.

\paragraph{Differentially Private Manifold Estimation.}
As a direct corollary (see~\cite[Theorem~5.7]{KLNRS11:what-can-we-learn-privately}),  these SQ upper bounds transform into private learning upper bounds.
They yield, to the best of our knowledge, the first private learning algorithms for manifold estimation.
More precisely, we proved that for all $\varepsilon = O(\rch_{\min})$, there is 
a local $\delta$-differentially\footnote{As the present paper uses $\varepsilon$ for precision, we use $\delta$ as the privacy parameter, contrary to the standard notation.}
private algorithm estimating the $d$-dimensional $\mathcal{C}^2$-manifolds $M$ with precision $\varepsilon$ that requires no more than
\begin{align*}
    s_{\delta\text{-private}}(\varepsilon)
    &=
    \Tilde{O}\left(
        \frac{n}{(f_{\min}\rch_{\min}^d)^3  \delta^2} 
        \left(
            \frac{\rch_{\min}}{\varepsilon}
        \right)^{(3d+2)/2} 
    \right)
\end{align*}
samples in the fixed point model, where $\Tilde{O}$ hides the logarithmic terms of the complexity.
See~\cite{KLNRS11:what-can-we-learn-privately} for more formal and thorough developments on differential privacy.

\paragraph{Lower Bounds.}

Complementing these upper bounds on the statistical query complexity of manifold estimation, we prove a \emph{computational} and an \emph{informational} lower bound, that both nearly match.
To examine whether or not randomness may facilitate manifold learning, the below lower bounds apply to \emph{randomized} SQ algorithms, which are allowed to use randomness and to fail with probability at most $\alpha \in[0, 1)$.
Recall that the above upper bounds stand for \emph{deterministic} ($\alpha = 0$) SQ algorithms.

First, we prove the following computational lower bounds.

\begin{description}
    \item[{[Fixed point model]}]
        \Cref{thm:SQ_lower_bound_point_computational}  asserts that any \emph{randomized} SQ 
        algorithm estimating $M$ with 
        precision $\varepsilon$ and probability of error at most $\alpha$ in the fixed point model requires at least
        \[
            q
            =
            \Omega\left(
                \frac{
                    \frac{n}{f_{\min}}
                    \left(\frac{1}{\rch_{\min} \varepsilon}\right)^{d / 2} + \log(1 - \alpha)
                }{
                    \log\left(1 + \frac{1}{\tau}\right)
                }
            \right)
        \] 
        queries to a $\STAT(\tau)$ oracle.

    \item[{[Bounding ball model]}]
        \Cref{thm:SQ_lower_bound_ball_computational} states that any \emph{randomized} SQ 
        algorithm estimating $M$ with 
        precision $\varepsilon$ and probability of error at most $\alpha$ in the bounding ball model requires at least
        \[
            q
            =
            \Omega\left(
                \frac{
                    n \log\left( \frac{R}{\varepsilon} \right)
                    +
                    \frac{n}{f_{\min}}
                    \left(\frac{1}{\rch_{\min} \varepsilon}\right)^{d / 2} 
                    +
                    \log(1 - \alpha)
                }{
                    \log\left(1 + \frac{1}{\tau}\right)
                }
            \right)
        \] 
        queries to a $\STAT(\tau)$ oracle.
\end{description}
In words, this proves that for any fixed probability of error $\alpha < 1$, the above manifold estimators are optimal up to a $\operatorname{polylog}(n,1/\tau)$ factor.
Hence, randomized algorithms are not significantly more powerful than deterministic ones in these models.

Second, we establish informational lower bounds (\Cref{thm:SQ_lower_bound_point_informational,thm:SQ_lower_bound_ball_informational})
that advocate for the necessity of the requirements on $\varepsilon$ and $\tau$ made in the upper bounds.
More precisely, they assert that if we either have
$
    \varepsilon 
    =
    o
    \left(
    \rch_{\min}
    \left(
        \frac{\tau}{f_{\min}\rch_{\min}^d}
    \right)^{2 / d}
    \right)
$,
or
$\tau = \Omega(f_{\min} \rch_{\min}^d)$
and
$\varepsilon = o(\rch_{\min})$,
then no SQ algorithm (even randomized) can estimate manifolds in these models with precision $\varepsilon$, regardless of its number of queries.
Said otherwise, the adversarial tolerance parameter $\tau$ prevents the learner to have arbitrary precision $\varepsilon$ , with quantitative lower bound
$$
    \varepsilon 
    =
    \Omega
    \left(
    \rch_{\min}
    \min\set{
    1
    ,
    \left(
        \frac{\tau}{f_{\min}\rch_{\min}^d}
    \right)^{2 / d}
    }
    \right)
$$
no matter the computational power of the statistician.

The high level exposition of these lower bounds stands in \Cref{section:intro-lower-bound},
and all the necessary details and formal statements in \Cref{sec:SQ-manifold-estimation}.

\subsubsection{Manifold Propagation Algorithm}
\label{section:intro-manifold-propagation}
The core component of the upper bounds (\Cref{thm:SQ_upper_bound_point,thm:SQ_upper_bound_ball})
is a purely geometric algorithm, which we call \MPA, parametrized by an initialization method $\hat{x}_0$ and two routines $\hat{T}(\cdot)$ and $\hat{\pi}(\cdot)$ related to the manifold $M$:
\begin{description}[font=\it]
    \item[(Seed point)] This initialization method finds a point $\hat{x}_0$ that is $\eta$-close to $M$, for some $\eta \geq 0$.
    \item[(Tangent space)] Given a point $x_0$ that is $\eta$-close to $M$, this routine
        finds a linear subspace $\hat{T}(x_0)$ that is $(\sin \theta)$-close to the tangent
        space $T_{\pi_M(x_0)} M$ at its projection $\pi_M(x_0)$ (i.e. the closest point on $M$), for some $\theta \geq 0$.
    \item[(Projection)] Given a point $x_0$ that is $\Lambda$-close to $M$, this routine
        finds a point $\hat{\pi}(x_0)$ that is $\eta$-close to its projection $\pi_M(x_0)$, where $\Lambda \geq \eta$.
\end{description}

Then, given a tuning scale parameter $\Delta > 0$, \MPA iteratively explores the connected manifold $M$ starting from the seed point, and constructs a $\Omega(\Delta)$-sparse and $O(\Delta)$-dense point cloud $\mathcal{O}$ of points close to $M$ (see \Cref{thm:MPA_properties}).
This algorithm is reminiscent of breadth-first search and can be roughly described as follows:
\begin{enumerate}
    \item 
    Start with the seed point $\hat{x}_0$ in the vicinity of the manifold and initialize a queue of points to $\mathcal{Q} = \set{\hat{x}_0}$, and the final output point cloud to $\mathcal{O} = \emptyset$.
    \item 
    Pick a point $x_0 \in \mathcal{Q}$, remove $x_0$ from $\mathcal{Q}$ and add it to $\mathcal{O}$.
    Compute the approximate tangent space $\hat{T}(x_0)$ of $M$ at $x_0$.
    \item 
    Consider a covering $y_1$, \dots, $y_k$ of a sphere of radius $\Delta$ in $\hat{T}(x_0)$. To avoid backtracking,  remove all of these points that are too close to (the already-visited) points from $\mathcal{Q} \cup \mathcal{O}$.
    To account for the linear approximation made and the past estimation errors, ``project'' the remaining points $y_i$'s on $M$ with $\hat{\pi}(\cdot)$ and add them to $\mathcal{Q}$.
    \item If $\mathcal{Q}$ is empty, terminate and output $\mathcal{O}$. Otherwise go to Step $2$.
\end{enumerate}
Note the importance of the proximity check of Step 3, without which the algorithm would not terminate, even with infinitely precise routines.

Then, given such a point cloud $\mathcal{O}$ that forms a $O(\Delta)$-dense sample of $M$, existing algorithms from computational geometry (\Cref{thm:manifold_reconstruction_deterministic}) allow to reconstruct a manifold with precision $O(\Delta^2/\rch_{\min})$. 
This quadratic gain is made possible by the $\mathcal{C}^2$-smoothness of $M$~\cite{Boissonnat14,Aamari18}.
Hence, running \MPA with step size $\Delta = O(\sqrt{\rch_{M}\varepsilon})$ and applying \Cref{thm:manifold_reconstruction_deterministic} yields a dynamic method to estimate a manifold $M$ with reach $\rch_{\min} >0$.
Namely, to estimate $M$ with precision $\varepsilon$ in Hausdorff distance, it shows that it is sufficient to design routines for $M$ that have precision 
$\eta = O(\varepsilon)$ for the seed point,
$\sin \theta = O(\sqrt{\varepsilon/\rch_{\min}})$ for the tangent spaces,
and
$\eta = O(\varepsilon)$ for the projection.

To our knowledge, this provides the first computational geometric result involving the three routines above only.
It also does a single call to $\hat{x}_0$ for initialization, and provably no more than
$O(\Haus^d(M)/\Delta^{d}) = O(\Haus^d(M)/(\rch_{\min}\varepsilon)^{d/2})$ 
calls to the routines $\hat{\pi}(\cdot)$ and $\hat{T}(\cdot)$, where $\Haus^d(M)$ stands for the surface area of $M$. In particular, this number of calls is blind to the ambient dimension.
Overall, \MPA manages to have this possible ambient dependency supported by $\hat{x}_0$, $\hat{T}(\cdot)$ and $\hat{\pi}(\cdot)$ only.
See \Cref{sec:mpa} for the formal statements and further discussion.

\subsubsection{Statistical Query Algorithms for the Routines}
\label{section:intro-SQ-implementations}
In order to plug \MPA in the SQ framework, we then provide SQ implementations of its geometric routines.

\paragraph{Projection Routine.}
As mentioned above, the projection routine should allow to find a point $\hat{\pi}(x_0)$ that is $\eta$-close to $\pi_M(x_0)$, provided that $x_0$ is $\Lambda$-close to $M$.
To implement this using a $\STAT(\tau)$ oracle, we first note that the conditional expectation of $D$ in the neighborhood of $x_0$ has small bias for estimating $\pi_M(x_0)$.
That is, 
$\|\pi_M(x_0) - \E_{x \sim D}\left[ x \left| \B(x_0, h) \right. \right]\|$ 
is small for a properly tuned bandwidth $h$ (see \Cref{lem:local_conditional_mean_bias}). 
Hence, it is enough to estimate 
\[
    m_D(x_0, h) = \E_{x \sim D}\left[ x \left| \B(x_0, h) \right. \right] = 
    x_0
    +
    h \frac{
        \E_{x \sim D}\left[ \frac{(x-x_0)}{h} \cdot \indicator{\|x - x_0\| \le h} \right]
    }{
        D(\B(x_0, h))
    }
    ,
\]
where $D(\B(x_0, h))$ stands for the mass of the ball $\B(x_0, h)$ for distribution $D$.
Written as a ratio of two means, one easily sees how to estimate $m_D(x_0, h)$ in $\STAT(\tau)$, as we now explain.
The denominator only requires one query $r = \indicator{\B(x_0, h)}$ to the oracle. 
As about the numerator, which is a $n$-dimensional mean vector, the naive approach that would query each coordinate of its integrand separately
would end up with the dimension-dependent precision $\sqrt{n}\tau$ in Euclidean norm.
Instead, by using tight frames, an algorithm of Feldman, Guzm\'{a}n, and Vempala~\cite{Feldman17} allows to get precision $O(\tau)$ in only $2n$ queries.

\noindent
At the end of the day, we achieve precision 
$
\eta = 
O\left(
    \Lambda^2/\rch_{\min}
\right)
$
with $O(n)$ queries to $\STAT(\tau)$,
provided that: (see \Cref{thm:routine_projection})
$$
\Lambda 
=
\Omega\left(
\rch_{\min}
\left( \frac{\tau}{f_{\min}\rch_{\min}^d} \right)^{1/(d+1)} 
\right)
.
$$

\paragraph{Tangent Space Routine.}
Here, the tangent space routine should allow to estimate the tangent space $T_{\pi_M(x_0)} M$ of $M$ at $\pi_M(x_0)$, provided that $x_0$ is $\eta$-close to $M$.
Local Principal Component Analysis  proved fruitful in the sample framework~\cite{Aamari18}. Inspired by it, we notice that the local covariance matrix of $D$ at $x_0$
\[
\Sigma_D(x_0, h) = 
\E_{x \sim D}\left[
    \frac{(x - x_0)\transpose{(x - x_0)}}{h^2} \indicator{\|x - x_0\| \le h}
    \right]
\]
allows to approximate $T_{\pi_M(x_0)} M$.
That is, $\Sigma_D(x_0, h)$ is almost rank-$d$ and its first $d$ principal components span a $d$-plane close to $T_{\pi_M(x_0)} M$, for a properly tuned bandwidth $h$ (see \Cref{lem:covariance_decomposed}). 
Next, aiming at estimating $\Sigma_D(x_0, h) \in \R^{n\times n}$ in $\STAT(\tau)$, we note that seeing it as a mean vector of $\R^{n^2}$ and using tight frames~\cite{Feldman17} directly would cost $O(n^2)$ queries for precision $O(\tau)$, but would not exploit the low-rank (hence redundant) structure of $\Sigma_D(x_0,h)$.
Instead, we use matrix compression arguments~\cite{Fazel08} to present a new general SQ algorithm estimating low-rank mean matrices (\Cref{lem:mean_matrix_estimation}). 
This allows to mitigate the query complexity from $O(n^2)$ to $O(dn \log^6(n))$ while keeping precision $O(\tau)$ in Frobenius norm.

\noindent
At the end of the day, coming back to our initial problem of tangent space estimation in $\STAT(\tau)$, we achieve precision 
$
\sin \theta = 
O\left(
    \sqrt{\eta/\rch_{\min}}
\right)
$ 
with $O(dn \operatorname{polylog}(n))$ queries to $\STAT(\tau)$,
provided that: (see \Cref{thm:routine_tangent})
$$\eta 
=
\Omega\left(
\rch_{\min}
\left( \frac{\tau}{f_{\min}\rch_{\min}^d} \right)^{2/(d+1)} 
\right)
.
$$

\paragraph{Seed Point Detection.}

Finally, the seed point $\hat{x}_0$ should be $\eta$-close to $M$.
In the fixed point model, this method is trivial since $0 \in M$ by assumption.
In the bounding ball model, where it is only assumed that $M \subseteq \B(0, R)$, we proceed in two stages:
\begin{itemize}[leftmargin=*]
    \item 
    First, we use a divide and conquer strategy  over $\B(0,R)$ (\Cref{thm:routine_detection_raw_point}).
    The algorithm (\ABS) recurses down over a discretization of $\B(0,R)$ with unions of small balls, maintained to intersect $M = \supp(D)$ by querying their indicator functions, i.e. by checking that they have non-zero mass for $D$. It stops when there is only one ball left and outputs its center $\hat{x}_0^{raw}$.
    Unfortunately, the output point $\hat{x}_0^{raw}$ of this simple strategy might only be 
    $O( \rch_{\min} \left( \tau/(f_{\min}\rch_{\min}^d) \right)^{1/d}  )$-close to $M$, since this procedure does not use the $\mathcal{C}^2$-smoothness of $M$.

    \item 
    Starting from $\hat{x}_0^{raw}$, the algorithm applies iteratively the projection routine $\hat{\pi}(\cdot)$ described above.
    Since $\hat{x}_0^{raw}$ is already close to $M$, the procedure is guaranteed to improve precision quadratically at each step, and to output a point $\hat{x}_0$ that is $\eta$-close to $M$ after a logarithmic number of iterations.
\end{itemize}

\noindent
At the end of the day, we achieve precision $\eta$ with $O(n \log(R/\eta))$ queries to $\STAT(\tau)$,
provided that: (see \Cref{thm:routine_detection_seed_point})
$$
\eta 
=
\Omega\left(
\rch_{\min}
\left( \frac{\tau}{f_{\min}\rch_{\min}^d} \right)^{2/(d+1)} 
\right)
.
$$

\subsubsection{Lower Bound Techniques}
\label{section:intro-lower-bound}
The standard SQ lower bound techniques, such as those involving the so-called statistical dimension~\cite{Feldman17b}, appear to be ill-suited to our context.
In fact, the informational bounds on the statistical dimension naturally involve Kullback-Leibler or chi-squared divergences~\cite{Feldman17b,DKS17}, which are non-informative in non-dominated statistical models such as manifolds ones.
Indeed, two low-dimensional submanifolds $M_0,M_1 \subseteq \R^n$ that are not equal would differ in a non-trivial area, yielding distributions are not absolutely continuous with respect to one another. This results in infinite Kullback-Leibler and 
chi-squared divergences, and hence non-informative lower bounds.

To overcome this issue, we present a two-step technique --- involving a computational and an informational lower bound --- that does not involve these quantities.
The method applies in general metric spaces (see \Cref{sec:lower-bounds}), although we shall limit its exposition and application to the case of manifolds with Hausdorff distance in this introduction.

\paragraph{Computational Lower Bounds.}

We aim at deriving a lower bound the number $q(\varepsilon)$ of queries necessary to achieve precision $\varepsilon$.
For this, we observe that since a SQ algorithm should cope with \emph{any} adversarial oracle, it has to cope with the oracle that answers roundings $\answer = \tau \floor{\E_{x \sim D}[r(x)] / \tau}$ of the true queried mean to the nearest integer multiple of $\tau$ in $[-1,1]$.
As this oracle only generates $(1 + 1 / \tau)$ different answers, any SQ manifold
estimation algorithm that makes only $q$ queries to this discrete oracle produces at
most $\mathcal{N} \leq (1 + 1 / \tau)^q$ possible outputs $\hat{M}$.
Hence, if this estimator has precision $\varepsilon$, these outputs should form an 
$\varepsilon$-covering of the manifold class of interest $\mathcal{M}$.
Hence, deriving a lower bound on 
$q = q(\varepsilon) \ge \log \mathcal{N}(\varepsilon) / \log(1 + 1/\tau)$ 
boils down to deriving a lower bound on the $\varepsilon$-covering number of
$\mathcal{M}$ for the Hausdorff distance,
or equivalently, by duality, on its $\varepsilon$-packing number.
This argument also extends to randomized SQ algorithms (see \Cref{subsec:lower-bound-general-computational}).

We then explicitly build $\varepsilon$-packings of the manifold classes associated to the models, with a general combinatorial scheme (see \Cref{prop:packing_local_variations}) based on a single initial manifold $M_0$. The construction bumps $M_0$ at many different locations, with bumps of height $\Omega(\varepsilon)$ scattered in all the available $(n-d)$ codimensions of space, as shown in \Cref{figure:bump-combinatorial}.
\begin{figure}[!h]
    \centering
    \includegraphics[height=4em]{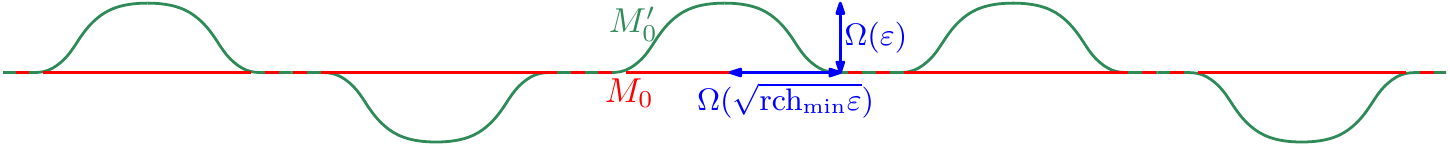}
    \caption{
        Construction of an $\varepsilon$-packing of the manifold class by local bumping. Here, in dimension $d = 1$ and codimension $n-d =1$, each bump has the two options ``upwards'' and ``downwards'' within each of the $N \gg 1$ locations, yielding $2^N$ $\varepsilon$-separated manifolds.
    }
    \label{figure:bump-combinatorial}
\end{figure}
Note that the $\mathcal{C}^2$-like assumption $\rch_M \geq \rch_{\min}$ forces to spread each of these bumps on a ball of radius $\Omega(\sqrt{\rch_{\min}\varepsilon})$.
Intuitively, in this construction, the larger the surface area $\Haus^d(M_0)$ of the base manifold $M_0$, the more room to include many bumps on it, and hence the stronger the lower bound. 
Hence, in the bounding ball model, we exhibit manifolds that can have large volume, while remaining
in $\B(0,R)$ and with reach $\rch_M \geq \rch_{\min}$. This is done by gluing next 
to each other linkable \emph{widgets} along a long path in the cubic grid in 
$\B(0,R)$ (see \Cref{subsec:building-large-volume-manifold}).
Overall, this construction allows to get the correct dependency in 
$1 / f_{\min}$ --- which plays the role of a maximal volume, see \Cref{subsubsec:implicit-bounds-on-parameters} --- in the bounds.

\begin{description}
    \item[{[Fixed point model]}] 
    If $0 \in M_0$, the above construction is possible while preserving the point $0 \in \R^n$ within all the bumped manifolds, yielding the lower bound (\Cref{thm:SQ_lower_bound_point_computational}).
    \item[{[Bounding ball model]}] 
    As in this model, no point is fixed, we build another type of $\varepsilon$-packing by translating a base manifold $M_0 \subseteq \B(0,R/2)$ in the ambient space by all the possible vectors of an $\varepsilon$-packing of the ambient ball $\B(0,R/2)$, which has size $\Omega((R/\varepsilon)^n)$. This yields the first term of the lower bound, while the second term follows as described above, by locally bumping a manifold $M_0 \subseteq \B(0,R)$ (\Cref{thm:SQ_lower_bound_ball_computational}).
\end{description}

\paragraph{Informational Lower Bounds.}
In addition, forgetting about the number of queries  SQ algorithms may do, they have a limited precision $\varepsilon$ given tolerance $\tau$.
Hence, aiming at lower bounding this best precision $\varepsilon(\tau)$ achievable in $\STAT(\tau)$, we notice that two distributions that are closer than $\tau/2$ in total variation distance allow an adversarial oracle to swap their respective answers, and hence make them --- and their supports --- indistinguishable using SQ’s.
This idea is at the core of standard lower bounds in the sample framework~\cite{Yu97}, and is formalized in the so-called \emph{Le Cam's lemma} for SQ's (\Cref{thm:lecam_SQ}).

\begin{figure}[!ht]
    \centering
    \includegraphics[height=4em]{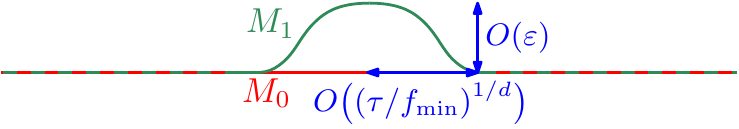}
    \caption{Indistinguishable manifolds for the informational lower bound.
    The measure on which they differ being of order $O(\tau)$, an adversarial $\STAT(\tau)$-oracle may fool the learner and force them to make an error proportional to the Hausdorff distance between them.
    }
    \label{figure:le-cams-lemma}
\end{figure}

To build such indistinguishable manifolds, we locally bump a base manifold $M_0$ at a single location.
As $M_0$ supports a $d$-dimensional measure with density lower bounded by $f_{\min}$, the largest possible width of such a bump is of order $\delta = \Omega((\tau/f_{\min})^{1/d})$, since the $d$-volume of this area multiplied by $f_{\min}$ (i.e. total variation) gets of order $\Omega(\tau)$.
Similarly as above, given the width $\delta$ of this bump, its largest possible height is $
\varepsilon 
= 
\Omega\left(
\delta^2/\rch_{\min} 
\right)
=
\Omega\left(
    \rch_{\min}
    \left(
        \frac{\tau}{f_{\min}\rch_{\min}^d}
    \right)^{2/d}
\right)
$, which provides the $\varepsilon$-separated manifolds indistinguishable in $\STAT(\tau)$.
This yields the announced informational lower bounds (\Cref{thm:SQ_lower_bound_point_informational,thm:SQ_lower_bound_ball_informational}), after picking manifolds $M_0$ in the fixed point and bounding model respectively that have volume of order $1/f_{\min}$, and the uniform distributions on them.

%% file: parts/preliminaries.tex
\subsection{Statistical Query Framework}
\label{subsec:SQ-general-setting}

To begin with the formal presentation of this work, let us define the statistical query (SQ) framework used throughout the paper.
In the SQ framework, the algorithm (or learner) is allowed to access to the unknown underlying distribution $D$ over $\R^n$ via an adversarial oracle $\oracle$ who knows it perfectly.
The learner also has access to some prior information on $D$ via the knowledge of a \emph{model} $\model$, i.e. a set of probability distributions over $\R^n$ assumed to contain $D$.
For a measurable function $r :\R^n \to [-1, 1]$, called \emph{query}, the oracle answers the mean value $\E_{x \sim D} [r(x)]$ of $r$ with respect to $D$, up to some adversarial error $\tau$ known to both parties.

Let $\mathfrak{F}$ denote the set of Borel-measurable functions from $\R^n$ to $[-1, 1]$.
An \emph{oracle} provides answers $\answer : \mathfrak{F} \to \R$. 
Given a query $r \in \mathfrak{F}$ and a tolerance parameter $\tau \geq 0$, we say that $\oracle$ is a \emph{valid} $\STAT(\tau)$ oracle for the distribution $D$ over $\R^n$ if its answers are such that 
$\left| \answer(r) - \E_{x \sim D} [r(x)]\right| \leq \tau$.
Let us insist on the fact that the oracle is adversarial, meaning that it can answer \emph{any} such values. Its adversarial strategy can also adapt to the previous queries made by the learner. See formal \Cref{def:deterministic-SQ}).

We now describe the estimation framework using SQ's.
Given a metric space $(\metric, \dist)$, a target precision $\varepsilon>0$ and a parameter of interest $\theta : \model \to \metric$, the learner aims at estimating $\theta(D)$ with precision $\varepsilon$ for the metric $\dist$ with a minimum number of queries $r:\R^n \to [-1,1]$, uniformly over the model $\model$.
The present framework is a particular case of the \emph{search problems} considered in~\cite{Feldman17b}, where a metric on $\Theta$ is not necessarily available. 

\begin{remark}
Manifold estimation will naturally bring us to consider the support $\theta(D) = \supp(D) \subseteq \R^n$ as the parameter of interest, and the Hausdorff distance $\dist = \dHaus$.
However, we present the broader setting of a general metric space $(\metric, \dist)$ of estimation, to also cover the intermediate results required by the SQ versions of the routines of \MPA (see \Cref{sec:SQ-routines}). Namely, it will involve  the Euclidean space $(\R^n,\norm{\cdot})$ for point estimation, and the matrix spaces $(\R^{n\times n},\normF{\cdot})$ and $(\R^{n\times n},\normop{\cdot})$.
\end{remark}

This paper considers interactive SQ algorithms, meaning that the learner is allowed to interact with the oracle dynamically and does not have to send all their queries at once before any answer.
That is, query functions are allowed to depend arbitrarily on the previous answers given by the oracle. More formally, we give the following \Cref{def:deterministic-SQ}.
\begin{definition}[Deterministic Statistical Query Estimation Framework]
\label{def:deterministic-SQ}
~

\begin{itemize}[leftmargin=*]
    \item
    A \emph{statistical query algorithm making $q$ queries} is a tuple
    $\Algorithm{A} = (r_1, \dots, r_q, \hat{\theta})$,
    where
    \[
    r_1 \in \mathfrak{F},
    ~
    r_2 : \R \to \mathfrak{F},
    ~
    \cdots,
    ~
    r_q : \R^{q - 1} \to \mathfrak{F},
    \text{ and } 
    \hat{\theta} : \R^q \to \metric
    . 
    \]    
    
    \item
    Let  $\answer_1 : \mathfrak{F} \to \R$, $\answer_2 : \mathfrak{F}^2 \to \R$, \dots,  $\answer_q : \mathfrak{F}^q \to \R$. 
    We say that $\oracle = (\answer_1, \dots, \answer_q)$ is a \emph{$\STAT(\tau)$ oracle for SQ algorithms making $q$ queries on the distribution $D$},
    if for all $r_1 \in \mathfrak{F}$, $r_2 : \R \to \mathfrak{F}$, \dots, $r_q : \R^{q - 1} \to \mathfrak{F}$,
    \begin{align*}
        \bigl|\answer_1(r_1) - \E_{x \sim D} \bigl[r_1(x) \bigr]\bigr| 
        &\le \tau,
        \\
        \bigl|
            \answer_2(r_1, r_2) 
            - 
            \E_{x \sim D} \bigl[r_2(\answer_1(r_1))(x)\bigr]
        \bigr| 
        &\le \tau,
        \\
        &\vdots 
        \\
        \bigl|
            \answer_q(r_1, \dots, r_q) 
            - 
            \E_{x \sim D} 
            \bigl[
                r_q(\answer_1(r_1), \dots, \answer_{q - 1}(r_1))(x)
            \bigr]
        \bigr| 
        &\le \tau
        .
    \end{align*}
    
    \item
    The \emph{output} of $\Algorithm{A} = (r_1, \dots, r_q, \hat{\theta})$ when it interacts with the oracle $\oracle = (\answer_1,\ldots,\answer_q)$ is defined by
    \[
        \hat{\theta}(r_1, \dots, r_q; \oracle) 
        = 
        \hat{\theta}(\answer_1(r_1), \dots, \answer_q(r_1, \dots, r_q))
        .
    \]
    
    \item
    Given a model $\model$ over $\R^n$ (i.e. a set of probability distributions),
    we say that a SQ algorithm $\Algorithm{A}$ is a \emph{$\STAT(\tau)$ estimator 
    with precision $\varepsilon$ for the statistical estimation problem 
    $\functional : \model \to \metric$} if for all $D \in \model$ and all valid 
    $\STAT(\tau)$ oracle $\oracle$ for $D$,
    \[
        \dist\big(\functional(D), \hat{\theta}(r_1, \dots, r_q; \oracle)\big) 
        \le \varepsilon.
    \]
    \end{itemize}

\end{definition}

Beyond deterministic algorithms, one may allow the learner to access randomness, and to fail at estimating the parameter of interest with some controlled probability $\alpha < 1$~\cite{Kearns:learning-from-SQs,Feldman17b}. This gives rise to the following \Cref{def:randomized-SQ}.

\begin{definition}[Randomized SQ Estimation Framework]
\label{def:randomized-SQ}
~

\begin{itemize}[leftmargin=*]
    \item
    A \emph{randomized SQ algorithm} $\RandmizedAlgorithm{A}$ is a distribution over SQ algorithms.
    
    \item
    Given a model $\model$ over $\R^n$, we say that a randomized SQ algorithm $\RandmizedAlgorithm{A}$ is a 
    \emph{
    $\STAT(\tau)$ algorithm with precision $\varepsilon$ and probability of failure (or error) $\alpha$ over $\model$},
    if for all distribution $D \in \model$ and all valid $\STAT(\tau)$ oracle $\oracle$ for $D$,
    \[ 
        \Pr_{\Algorithm{A} = (r_1, \dots, r_q, \hat{\theta}) \sim \RandmizedAlgorithm{A}} 
        \Big[
            \dist\big(\functional(D), 
                \hat{\theta}(r_1, \dots, r_q; \oracle)\big) \le \varepsilon
        \Big]
        \ge
        1 - \alpha
        .
    \]
    \end{itemize}
\end{definition}

Naturally, any SQ algorithm making at most $q$ queries can be emulated into a PAC algorithm by considering empirical averages. 
Indeed, given independent and identically distributed data $\{X_1,\ldots,X_s\}$ with common distribution $D$, Hoeffding's inequality yields that the oracle which answers
$\answer = \frac{1}{s} \sum_{i=1}^s r(X_i)
$ to any query $r : \R^n \to [-1,1]$
satisfies 
$
|\answer - \E_P[r]| \leq \tau
$
with probability 
$
\geq 
1 - 2 e^{-s\tau^2/2}
$
.
Therefore, if we can estimate $\functional : \model \to \metric$ ``efficiently'' in $\STAT(\tau)$, so do we in the PAC (sample) framework, with sample complexity $
s \leq 
\frac{q \sqrt{\log(q/\alpha)}}{\tau^2}
.
$
.

A priori, randomized algorithms may require significantly less queries than deterministic ones to achieve an estimation task~\cite{Feldman17b}.
However, we will show that this phenomenon does not occur for manifold estimation, as soon as the probability of error $\alpha$ is not considerably close to $1$. For this, we will exhibit upper bounds using \emph{deterministic} algorithms, and matching lower bounds on \emph{randomized} algorithms. See \Cref{sec:SQ-manifold-estimation} for the precise statements.

    \begin{remark}[About Noise]
        \label{rk:clutter_noise}
        The statistical models considered in this work (\Cref{def:manifold_model_distribution}) are noise-free, in the sense that the $\STAT(\tau)$ oracle --- although adversarial --- has access to the exact underlying distribution $D$.
        Beyond such an exact model, a noise model that is particularly popular in the manifold estimation literature is the so-called \emph{clutter noise} model~\cite{Aamari18,Genovese12b}.
        Given a nuisance parameter $\beta \in (0,1]$ and a fixed noise distribution $Q_0$ over $\R^n$ --- usually 
        the ambient uniform distribution over a compact set of $\R^n$ ---,  the associated clutter noise model is the set of mixtures
        \[
            \model^\mathrm{(clutter)}_{\beta, Q_0}
            =
            \set{
                \beta D + (1 - \beta) Q_0
                ,
                D \in \model
            }
        .
        \]
        In particular,
        $\model^\mathrm{(clutter)}_{\beta = 1, Q_0}$
        coincides with $\model$.
        For $\beta < 1$, in the independent and identically distributed (i.i.d.) sampling framework, it yields samples with a proportion of approximately $\beta$ informative points and $(1-\beta)$ of non-informative clutter points.
        
        As mentioned in the introduction, this type of noise model gave rise to subtle iterative decluttering procedures that rely heavily on the properties of $Q_0$ (i.e. being ambient uniform)~\cite{Aamari18}. This noise-specificity is also a limitation of the (intractable) estimator of~\cite{Genovese12b}, which would also fail with clutter distributions $Q_0$ other than uniform.
        In contrast, in the statistical query framework, if $\beta$ and $Q_0$ are known, then estimation techniques need not be much more elaborate for the case $\beta < 1$ than for $\beta =1$.
        Indeed, the statistical query complexity of an estimation problem in $\STAT(\tau)$ over $\model^\mathrm{(clutter)}_{\beta, Q_0}$ coincides with its counterpart in $\STAT(\tau/\beta)$ over $\model$.
        
        The correspondence is explicit: algorithms designed for $\beta = 1$ naturally generalize for $\beta < 1$ and vice-versa.
        To see this, let $r:\R^n \to [-1,1]$ be a query to a $\STAT(\tau)$ oracle with true distribution $D^\mathrm{(clutter)} = \beta D + (1 - \beta) Q_0$. Say that the learner gets answer $\answer \in \R$, then the function returning $\answer' = (\answer - (1 - \beta) \E_{Q_0}[r]) / \beta$, which can be computed by the learner who knows $Q_0$ and $\beta$, clearly simulates a valid $\STAT(\tau/\beta)$ oracle to the query for the distribution $D$.
        For the same reason, conversely, any $\STAT(\tau/\beta)$-algorithm over $\model$ yields a $\STAT(\tau)$-algorithm in $\model^\mathrm{(clutter)}_{\beta, Q_0}$.
        
        This shows that the statistical query complexity in $\STAT(\tau)$ over $\model$ coincides with its counterpart in $\STAT(\beta \tau)$ over $ \mathcal{D}^\mathrm{(clutter)}_{\beta, Q_0}$ for any fixed $0 < \beta \leq 1$ and clutter distribution $Q_0$.
        Conversely, any SQ algorithm in the clutter-free model can easily be made robust to clutter noise, as soon as the clutter distribution $Q_0$ and noise level $0 < \beta \leq 1$ are known to the learner.
    \end{remark}

As a first illustration of a non-trivial SQ estimation problem, let us describe that of the mean $\E_{x \sim D}[F(x)]$ of a bounded vector-valued function ${F:\R^n\to\R^k}$, where $\norm{F(x)}\leq 1$, see~\cite{Feldman17}. Here and below, $\norm{\cdot}$ stands for the Euclidean norm.
This example will be central in the construction of our SQ geometric routines (\Cref{thm:routine_projection,thm:routine_detection_seed_point,thm:routine_tangent}), and hence for the final SQ manifold estimation algorithms (\Cref{thm:SQ_upper_bound_ball,thm:SQ_upper_bound_point}).

One query to a $\STAT(\tau)$ oracle allows to compute the mean value of a function ranging in $[-1,1]$ with precision $\tau$.
Hence, the $k$ coordinate functions $r_i(x) = \inner{e_i}{F(x)} \in [-1,1]$ are valid queries, and allow to estimate each coordinate of $\E_{x \sim D}[F(x)]$ with precision $\tau$.
This naive strategy results in a deterministic SQ algorithm making $k$ queries to $\STAT(\tau)$ and precision $\tau$ for the sup-norm, but only $\sqrt{k}\tau$ for the Euclidean norm.
The following \Cref{lem:mean_vector_estimation} shows that the learner may ask $2k$ queries to a $\STAT(\tau)$ oracle, while still preserving a precision of order $\tau$ for the Euclidean norm.
The strategy consists in querying $F$ in a suitable frame of $\R^k$~\cite[Theorem~3.9]{Feldman17}, i.e. a redundant family of vectors of $\R^k$ which avoids the extra $\sqrt{k}$ factor of the non-redundant coordinate-wise queries.

\begin{lemma}
\label{lem:mean_vector_estimation}
    Let $D$ be a Borel probability distribution on $\R^n$, and ${F:\R^n\to\R^k}$ be such that $\norm{F(x)} \leq 1$ for all $x \in \R^n$.
    
    There exists a deterministic statistical query algorithm making $2k$ queries to a $\STAT(\tau)$ 
    oracle, and that estimates $\E_{x \sim D} \left[F(x)\right]$ with precision $C \tau$ for the
    Euclidean norm, where $C > 0$ is a universal constant.
\end{lemma}
        
\begin{proof}[\proofof \Cref{lem:mean_vector_estimation}]
    Let us denote by $D'$ the pushforward distribution of $D$ by $F$. As for all measurable function 
    $r : \R^k \to [-1, 1]$,
    \[
        \E_{x' \sim D'}[r(x')]
        =
        \E_{x \sim D}[r(F(x))],
    \]
    any valid $\STAT(\tau)$ oracle for $D$ simulates a valid $\STAT(\tau)$ oracle for $D'$. Hence, 
    applying~\cite[Theorem 3.9]{Feldman17} to $D'$, we get the desired 
    result.
\end{proof}

\subsection{Manifold Regularity and Distributional Assumptions}

\subsubsection{General Notation and Differential Geometry}
From now on, $n \ge 2$ is referred to as the ambient dimension and $\mathbb{R}^n$
is endowed with the Euclidean inner product $\left\langle \cdot, \cdot \right\rangle$
and the associated norm $\norm{\cdot}$.
The closed Euclidean ball of center $x$ and radius $r$ is denoted by $\B(x, r)$.
The volume of the $d$-dimensional unit ball $\B_d(0,1)$ is denoted by $\omega_d$, and that of the $d$-dimensional unit sphere $\Sphere^{d}(0,1) \subseteq \R^{d+1}$ by $\sigma_d$.

We will consider compact connected submanifolds $M$ of $\R^n$, without boundary,
and with dimension $d < n$~\cite{DoCarmo92}. 
Given a point $p \in M$, the tangent space of $M$ at $p$, denoted by $T_p M$,
is the $d$-dimensional linear subspace of $\R^n$ spanned by the velocity vectors
at $p$ of $\mathcal{C}^1$ curves of $M$.
The Grassmannian $\Gr{n}{d}$ is the set of all the $d$-dimensional linear subspaces of $\R^n$, so that $T_p M \in \Gr{n}{d}$ for all $p \in M$.
In addition to the Euclidean structure induced by $\R^n$ on $M \subseteq \R^n$,
we also endow $M$ with its intrinsic geodesic distance $\dd_M$, with $\B_M(p, s)$
denoting the closed geodesic ball of center $p\in M$ and of radius $s$. 
More precisely, given a $\mathcal{C}^1$ curve $c: [a,b] \rightarrow M$,
the length of $c$ is defined as $\Length(c)= \int_a^b \norm{c'(t)}dt$. Given 
$p, q \in M$, there always exists a path $\gamma_{p \rightarrow q}$ of minimal
length joining $p$ and $q$~\cite{DoCarmo92}. Such a curve 
$\gamma_{p \rightarrow q}$ is called geodesic, and the geodesic distance between 
$p$ and $q$ is given by 
$\dd_M(p,q) = \Length(\gamma_{p \rightarrow q})$~\cite[Chapter 2]{Burago01}. 
In particular, $(M, \dd_M)$ is a length space~\cite[Remark 5.1.7]{Burago01}. 
A geodesic $\gamma$ such that $\norm{\gamma'(t)} = 1$ for all $t$ is 
called arc-length parametrized. Unless stated otherwise, we always assume that
geodesics are parametrized by arc-length. For all $p \in M$ and all unit vectors 
$v \in T_p M$, we denote by $\gamma_{p, v}$ the unique arc-length parametrized geodesic
of $M$ such that $\gamma_{p, v}(0)=p$ and $\gamma'_{p, v}(0) = v$~\cite{DoCarmo92};
the exponential map is defined as $\exp_p^M(vt) = \gamma_{p, v}(t)$. Note that from
the compactness of $M$, $\exp_p^M: T_p M \rightarrow M$ is defined globally on 
$T_p M$~\cite[Theorem 2.5.28]{Burago01}.

\subsubsection{Geometric and Statistical Models}
\label{subsec:geom-stat-model}
Let us detail the geometric assumptions we will make throughout. Besides the differential structure given by low-dimensional submanifolds, the core regularity assumption of this work will be encoded by the \emph{reach}, a central quantity in the statistical analysis of geometric structures (see~\cite{Aamari19} and references therein), and that we now describe.

To this aim, let us define the \emph{medial axis} $\Med(K)$ of a closed subset $K \subseteq \R^n$ as the set of ambient points that have at least two nearest neighbors on $K$. 
Namely, if we let $\dd(z,K) = \inf_{p \in K} \norm{p-z}$ denote the distance
function to $K$,
\begin{align*}
    \Med(K) =
    \set{
        z \in \R^n
        ~|~
        \exists p\neq q \in K, \norm{p-z} = \norm{q-z} = \dd(z,K)
    }.
\end{align*}
By definition of the medial axis, the \emph{metric projection onto $K$}, given by
\begin{align*}
    \pi_K(z) = \argmin_{p \in K} \norm{p - z}
\end{align*}
is well defined exactly on $\R^n \setminus \Med(K)$. 
The reach of $K$ is then defined as the minimal distance from $K$ to $\Med(K)$.
\begin{definition}[{\cite[Theorem 4.18]{Federer59}}]
\label{def:reach}
    The \emph{reach} of a closed subset $K \subseteq \R^n$ is defined by
    \begin{align*}
        \rch_K
        &=
        \inf_{z \in \Med(K)} \dd(z,K )
        .
    \end{align*}
    Furthermore, if $K = M \subseteq \R^n$ is a $\mathcal{C}^2$-submanifold, then its reach can be written as
    \begin{align*}
        \rch_M
        &=
        \inf_{p \neq q \in M}
        \frac{\norm{q-p}^2}{\dd(q-p,T_p M)}
        .
    \end{align*}
\end{definition}
The second formulation of \Cref{def:reach} assesses how a large positive reach testifies of a quantitative uniform regularity of $M \subseteq \R^n$. 
Indeed, the submanifold $M$ being $\mathcal{C}^2$-smooth essentially means that locally, $M$ deviates at most quadratically from its tangent spaces. Adding the condition $\rch_M \geq \rch_{\min} > 0$ yields the quantitative bound $\dd(q-p,T_p M) \leq \norm{q-p}^2/(2\rch_{\min})$ for all $p,q \in M$. In particular, this condition bounds curvature and intrinsic metric properties (see \Cref{lem:geodesic_comparison}).
We shall refer the interested reader to~\cite{Aamari19} for further discussions on the reach.

\begin{definition}
\label{def:manifold_model}
    We let $\manifolds{n}{d}{\rch_{\min}}$ denote the class of compact connected $d$-dimensional
    $\mathcal{C}^2$-submanifolds $M$ of $\R^n$, without boundary, that have 
    reach bounded below by $\rch_M \ge \rch_{\min}$.
\end{definition}

Among the key properties shared by submanifolds $M$ with reach bounded below $\rch_M \geq \rch_{\min}$ are a quantitative equivalence between the Euclidean and geodesic distances, and the fact that their curvature is uniformly bounded by $1/\rch_{\min}$.

\begin{lemma}
\label{lem:geodesic_comparison}
        Let $M \in \manifolds{n}{d}{\rch_{\min}}$ and $p, q \in M$. 
        If $\norm{q - p} < 2 \rch_{\min}$, then
        \begin{align*}
            \norm{q-p}
            \le
            \dd_M(p,q)
            \le
            2 \rch_{\min}
            \arcsin{\left(\frac{\norm{q-p}}{2\rch_{\min}}\right)}
            .
        \end{align*}
        In particular, for all $r < 2 \rch_{\min}$,
        \begin{align*}
            \B\left( 
                p, 
                r\left(
                    1 - \left( r / \rch_{\min} \right)^2/24 
               \right)
            \right) \cap M
            &\subseteq
            \B_M(p, r) 
            \\
            &\subseteq 
            \B(p, r) \cap M 
            \\
            &\subseteq 
            \B_M
            \left(
                p,
                r\left(
                    1 + \left( r / \rch_{\min} \right)^2/4 
                \right)
            \right).
        \end{align*}
        Furthermore, if $\gamma : [a, b] \to M$ is an arc-length 
        parametrized geodesic, then for all $t \in [a, b]$, 
        $\norm{\gamma''(t)} \le 1 / \rch_{\min}$.
    \end{lemma}
    
    \begin{proof}[\proofof \Cref{lem:geodesic_comparison}]
        We clearly have $\norm{q - p} \le \dd_M(p,q)$, and the upper bound comes 
        from~\cite[Lemma 3]{Boissonnat19}. The ball inclusions then follow from 
        the elementary bounds $\sin s \ge s(1-s^2/6)$ for $s\ge 0$, and 
        $\arcsin u \le u(1+u^2)$ for $0 \le u \le 1$.
        The last claim is a rephrasing of~\cite[Proposition~6.1]{Niyogi08}.
    \end{proof}
    
    These estimates will be used to compare, in a quantitative way, the (curved) geometry of $M$ with that of the (flat) Euclidean $d$-dimensional space.
    Finally, we present the following uniform estimate on the massivity of submanifolds $M \in \manifolds{n}{d}{\rch_{\min}}$, which we will use below to show that \MPA terminates.    
    For $\delta > 0$, the $\delta$-packing number $\PK_M(\delta)$ of $M \subseteq \R^n$ is the maximal cardinal $k$ of a set of points $\set{p_i}_{1\leq i \leq k} \subseteq M$ such that $\B(p_i,\delta) \cap \B(p_j,\delta) = \emptyset$ for all $i \neq j$ (i.e. $\norm{p_i - p_i} > 2\delta$) (see \Cref{subsec:euclidean-packing-and-covering} for more details).
    \begin{lemma}
    \label{lem:packing_number}
        Let $M \in \manifolds{n}{d}{\rch_{\min}}$. 
        Then for all $\delta \leq \rch_{\min}/8$, 
        \[
            \PK_M(\delta) \le \dfrac{\Haus^d(M)}{\omega_d(\delta/4)^d}
        ,
        \]
        where $\Haus^d(M)$ denotes the surface area of $M$.
    \end{lemma}
    \begin{proof}[\proofof \Cref{lem:packing_number}]
        Follows from \Cref{prop:packing_covering_link} and \Cref{prop:packing_covering_manifold}.
    \end{proof}

Based on the geometric model above (\Cref{def:manifold_model}), we now describe the statistical model (i.e. set of probability distributions) of this work.
Every $M \in \manifolds{n}{d}{\rch_{\min}}$ inherits a non-trivial finite measure induced by the $d$-dimensional Hausdorff measure $\Haus^d$ on $\R^n\supseteq M$, defined by $\mathrm{vol}_M = \indicator{M} \Haus^d$, and called the
volume measure of $M$. 
Note that with this normalization, $\mathrm{vol}_M(M) = \Haus^d(M)$ corresponds to the $d$-dimensional surface area of $M$, and $\mathrm{vol}_M/\Haus^d(M)$ corresponds to the uniform probability distribution on $M$.
    
\begin{definition}
\label{def:manifold_model_distribution}
    We let $\distributions{n}{d}{\rch_{\min}}{f_{\min}}{f_{\max}}{L}$ denote the set of
    Borel probability distributions $D$ on $\R^n$ with 
    $M = \supp(D) \in \manifolds{n}{d}{\rch_{\min}}$ and a density $f$ with respect to $\mathrm{vol}_M$ such that:
    \begin{itemize}
        \item $f$ is bounded away from zero and infinity: 
            $0 < f_{\min} \le f(x) \le f_{\max} < \infty$ for all $x \in M$.
        \item $f$ is $L$-Lipschitz over $M$:
            $|f(x) - f(y)| \le L \norm{x - y}$ for all $x, y \in M$.
    \end{itemize}
\end{definition}

In this model, as will be clear below, the extra degree of freedom allowed by the density $f$ being non-constant will contribute in the final estimation rate and query complexity, especially through the lower bound $f_{\min}$.
On the geometric side, $f_{\min}^{-1}$ and $f_{\max}^{-1}$ impose quantitative restrictions on the volume $\Haus^d(M)$ of $M$ (see \Cref{subsubsec:implicit-bounds-on-parameters}).

Since $\distributions{n}{d}{\rch_{\min}}{f_{\min}}{f_{\max}}{L}$ is invariant by translations in $\R^n$, this model actually provides insufficient prior information to derive any uniform SQ complexity bound over it.
This contrasts sharply with the sample framework~\cite{Aamari18,Genovese12b}, where the sample points provide automatic location information and yields finite sample complexity over $\distributions{n}{d}{\rch_{\min}}{f_{\min}}{f_{\max}}{L}$.

\begin{proposition}
    \label{prop:SQ_lower_bound_infinity_unbounded}
    Assume that $\sigma_d f_{\min} \rch_{\min}^d \leq 1$. Then for all $\varepsilon > 0$, manifold estimation over $\distributions{n}{d}{\rch_{\min}}{f_{\min}}{f_{\max}}{L}$ with precision $\varepsilon$ has infinite randomized statistical query complexity.
\end{proposition}
The assumption that $\sigma_d f_{\min} \rch_{\min}^d \leq 1$ is made to preclude degeneracy of the model. It can be shown to be necessary (see \Cref{subsubsec:implicit-bounds-on-parameters} below for a more detailed discussion).
The proof of \Cref{prop:SQ_lower_bound_infinity_unbounded} relies on the fact that the supports $\supp(D)$ of distributions 
$D \in \distributions{n}{d}{\rch_{\min}}{f_{\min}}{f_{\max}}{L}$ form an unbounded class for the Hausdorff distance.
It is therefore natural to add an extra location constraint to the model. We study two different such constraints. 
The first one fixes membership of a distinguished point to $M$, which we take to be the origin $0\in \R^n$ without loss of generality.
The second one bounds the problem in an ambient ball of radius $R>0$, which we take to be centered at the origin $\B(0,R)$ without loss of generality.

\begin{definition}
    \label{def:models_bounded}
    Completing the framework of \Cref{def:manifold_model_distribution}, we consider the two following models.
     \begin{itemize}[leftmargin=*]
         \item Fixed point model:
         \begin{itemize}[leftmargin=*]
            \item 
            $\manifoldspoint{n}{d}{\rch_{\min}}{0}$ denotes the set of manifolds $M \in \manifolds{n}{d}{\rch_{\min}}$ such that $0 \in M$;
            \item
            The model $\distributionspoint{n}{d}{\rch_{\min}}{f_{\min}}{f_{\max}}{L}{0}$ stands for the set of distributions $D \in \distributions{n}{d}{\rch_{\min}}{f_{\min}}{f_{\max}}{L}$ with support such that $0 \in M = \supp(D).$
         \end{itemize}

         \item Bounding ball model: given $R > 0$,
         \begin{itemize}[leftmargin=*]
            \item 
            $\manifoldsball{n}{d}{\rch_{\min}}{R}$ denotes the set of manifolds $M \in \manifolds{n}{d}{\rch_{\min}}$ such that $M \subseteq \B(0,R)$;
            \item
            The model $\distributionsball{n}{d}{\rch_{\min}}{f_{\min}}{f_{\max}}{L}{R}$ stands for the set of distributions $D \in \distributions{n}{d}{\rch_{\min}}{f_{\min}}{f_{\max}}{L}$ with support such that $\supp(D) = M \subseteq \B(0,R).$
         \end{itemize}     
         \end{itemize}
\end{definition}

Let us now discuss some features imposed by the above models.

    \subsubsection{On Some Implicit Bounds on the Model Parameters}
    \label{subsubsec:implicit-bounds-on-parameters}
    
    Although not explicit in \Cref{def:models_bounded}, parameters of the models are not arbitrary.
    That is, $\distributions{n}{d}{\rch_{\min}}{f_{\min}}{f_{\max}}{L}$ might be degenerate or even empty in some regimes of parameters, making the manifold estimation problem vacuous. The reason for this resides in implicit volume bounds imposed by the reach.
    Indeed, if $D \in \distributions{n}{d}{\rch_{\min}}{f_{\min}}{f_{\max}}{L}$ has support $M$, then since $D$ is a probability distribution,
    \[
    f_{\min} \Haus^d(M)
    \leq
    1
    =
    \int_M f \dd \Haus^d
    \leq
    f_{\max} \Haus^d(M)
    .
    \]
    As a result, the volume estimates of \Cref{prop:volume_bounds_under_reach_constraint} yield
    \begin{align*}
        f_{\min} 
        \leq
        \frac{1}{\Haus^d(M)}
        \leq
        \frac{1}{\sigma_d \rch_{\min}^d}
        \leq
        \frac{1}{\omega_d \rch_{\min}^d}
        .
    \end{align*}
    If furthermore, $D \in \distributionsball{n}{d}{\rch_{\min}}{f_{\min}}{f_{\max}}{L}{R}$ (i.e. $M \subseteq \B(0,R)$), then
    \begin{align*}
        f_{\max}
        \geq
        \frac{1}{\Haus^d(M)}
        \geq
        \frac{
            1
        }{
            \left(\frac{18R}{\rch_{\min}} \right)^n
            \omega_d \left( \frac{\rch_{\min}}{2} \right)^d
        }
        .
    \end{align*}
    Note that \Cref{prop:volume_bounds_under_reach_constraint} also yields that $R \geq \rch_{\min}/\sqrt{2}$.
    Consequently, to ensure non-vacuity of the models, and without loss of generality, it is natural to take the following setup. Here, $C_\square$ stands for a constant depending only on~$\square$.
    \begin{itemize}[leftmargin=*]
        \item 
        When working over $\distributionspoint{n}{d}{\rch_{\min}}{f_{\min}}{f_{\max}}{L}{0}$, we will always assume that $f_{\min} \leq f_{\max}$, $R \geq C \rch_{\min}$, and
        \[
        \omega_d f_{\min} \rch_{\min}^d
        \leq
        C_d^{-1}
        ,
        \]
        for some large enough constant $C_d > 0$.    
        \item 
        When working over $\distributionsball{n}{d}{\rch_{\min}}{f_{\min}}{f_{\max}}{L}{R}$, we will always assume that $f_{\min} \leq f_{\max}$, $R \geq C \rch_{\min}$,
        \[
        \omega_d f_{\min} \rch_{\min}^d
        \leq
        C_d^{-1}
        \text{~~and~~}
        \omega_d f_{\max} \rch_{\min}^d
        \geq
        C_{n,d}
        \left( \frac{\rch_{\min}}{R} \right)^n
        ,
        \]
        for some large enough constants $C,C_d,C_{n,d} > 0$.    
    \end{itemize}
See \Cref{sec:misc} for a more thorough exposition of the technical properties of the models $\distributionspoint{n}{d}{\rch_{\min}}{f_{\min}}{f_{\max}}{L}{0}$ and $\distributionsball{n}{d}{\rch_{\min}}{f_{\min}}{f_{\max}}{L}{R}$.
    
    \subsection{Manifold Reconstruction from Point Clouds}
    
    Following the recent line of research on manifold estimation~\cite{Genovese12,Genovese12b,Aamari18,Aamari19b,Divol20}, we will measure the accuracy of estimators $\hat{M}$ of manifolds $M$ via the so-called \emph{Hausdorff distance}, which plays the role of an $L^\infty$-distance between compact subsets of $\R^n$.
    To this aim, we will need the following piece of notation. 
    For $K \subseteq \R^n$ and $r\geq 0$, we let $K^r$ denote the \emph{$r$-offset of $K$}:
\begin{align}
    \label{eq:offset}
    K^r
    :=
    \set{
        z \in \R^n,
        \dd(z,K) \leq r
    }
    ,
\end{align}
where we recall that $\dd(z,K) = \inf_{p \in K} \norm{p-z}$ is the function distance to $K$.
\begin{definition}[{Hausdorff Distance~\cite[Section 7.3.1]{Burago01}}]
    \label{def:hausdorff_distance}
        Given two compact subsets $K,K' \subseteq \R^n$, the \emph{Hausdorff distance} between them is
        \begin{align*}
            \dHaus(K,K')
            &=
            \sup_{x \in \R^n} |\dd(x,K) - \dd(x,K')|
            \\
            &=
            \inf\set{
                r > 0
                ,
                K \subseteq (K')^r
                \text{ and }
                K' \subseteq K^r
            }
            .
        \end{align*}
    \end{definition}
    
    Manifold reconstruction from point clouds has been extensively studied in the area of computational geometry~\cite{DeyBook,Boissonnat14}.
    In this field, the learner is given a sample of $M$, usually seen as deterministic, and the goal is to build efficiently a reliable triangulation $\hat{M}$ of $M$, either topologically, geometrically, or both.
    Such a construction actually is always possible, provided that the point cloud is sufficiently close and dense in $M$, and that the learner is provided with tangent space estimates at these points.
    This is formalized in the following \Cref{thm:manifold_reconstruction_deterministic}, where $\normop{\cdot}$ stands for the operator norm over the set of matrices.
    \begin{theorem}[{Adapted from~\cite[Theorem 4.4]{Aamari18}}]
    \label{thm:manifold_reconstruction_deterministic}
        There exists $\lambda_d > 0$ such that for all $\varepsilon \le \lambda_{d}\rch_{\min}$
        and all $M \in \manifolds{n}{d}{\rch_{\min}}$, the following holds.
        
        Let $\mathcal{X} \subseteq \R^n$ be a finite point cloud and $\mathbb{T}_\mathcal{X} = 
        \bigl\{T_x\bigr\}_{x\in \mathcal{X}} \subseteq \Gr{n}{d}$ be a family of $d$-dimensional linear subspaces of 
        $\R^n$ such that
        \begin{multicols}{2}
            \begin{itemize}[leftmargin=*]
                \item $\underset{x \in \mathcal{X}}{\max} \: \dd(x,M) \le \eta$,
                \item $\underset{p\in M}{\max} \: \dd(p, \mathcal{X}) \le  \Delta$,
                \item $\underset{x \in \mathcal{X}}{\max} \: 
                    \normop{\pi_{T_{\pi_M(x)} {M}} - \pi_{T_x} } \le \sin \theta$.                
            \end{itemize}
        \end{multicols}
        \noindent If $\theta \le \Delta/(1140 \rch_{\min})$ and 
        $\eta \le \Delta^2/(1140 \rch_{\min})$, then one can build a triangulation 
        $\hat{M} = \hat{M}(\mathcal{X}, \mathbb{T}_\mathcal{X})$ with vertices in $\mathcal{X}$ such that
        \begin{multicols}{2}
        \begin{itemize}[leftmargin=*]
            \item $\dHaus \bigl( M, \hat{M} \bigr) 
            \le
            C_{d}\Delta^2/ \rch_{\min}$,
            \item $M$ and $\hat{M}$ are ambient isotopic.
        \end{itemize}
        \end{multicols}
    \end{theorem}

    \begin{proof}[\proofof \Cref{thm:manifold_reconstruction_deterministic}]
        We apply~\cite[Theorem 4.4]{Aamari18} on a sparsified subset
        $\mathcal{X}'$ of $\mathcal{X}$, which is a pruned version of $\mathcal{X}$
        that is $\varepsilon$-sparse but still dense enough in $M$. This subsample
        $\mathcal{X}'$ can be built explicitly by using the so-called farthest point 
        sampling algorithm to $\mathcal{X}$~\cite[Section 3.3]{Aamari18}. For this, 
        initialize $\mathcal{X}'$ with  $\mathcal{X}'= \set{x_0}$, where 
        $x_0 \in \mathcal{X}$ is chosen arbitrarily. 
        Then, while $\max_{x \in \mathcal{X}} \dd(x, \mathcal{X}') > \Delta$, 
        find the farthest point to $\mathcal{X}'$ in $\mathcal{X}$, and add it to 
        $\mathcal{X}'$. That is,
        $\mathcal{X}' \leftarrow \mathcal{X}' \cup 
            \set{\argmax_{x \in \mathcal{X}} \dd(x, \mathcal{X}')}$ 
        (and if the $\argmax$ is not a singleton, pick an arbitrary element of it).
        The output $\mathcal{X}' \subseteq \mathcal{X}$ of this algorithm clearly satisfies 
        $\min_{x'\neq y' \in \mathcal{X}'} \norm{y'-x'} \ge \Delta$, and furthermore,
        \[
            \max_{p \in M} \dd(p, \mathcal{X}') 
            \le
            \max_{p \in M} \dd(p, \mathcal{X}) 
                + 
                \max_{x \in \mathcal{X}} \dd(x, \mathcal{X}') \le 2\Delta
            .
        \]
        Therefore,~\cite[Theorem 4.4]{Aamari18} applies to $\mathcal{X}'$ and 
        $\mathbb{T}_{\mathcal{X}'}$, and $\hat{M}(\mathcal{X}',\mathbb{T}_{\mathcal{X}'})$ provides the announced triangulation.
    \end{proof}
    
    Although we will not emphasize on exact topology recovery in the present work, let us mention that the triangulation $\hat{M}$ actually exhibits the extra feature of sharing the same topology as $M$, i.e. $M$ and $\hat{M}$ are isotopy equivalent.
    Let us also mention that the triangulation can be built in linear time in $n$, with an explicit polynomial time and space complexity~\cite[Section 4.6]{Boissonnat14}.
    
    Said otherwise, \Cref{thm:manifold_reconstruction_deterministic} asserts that manifold reconstruction with precision $\varepsilon$ can be achieved if a sample that is $(\sqrt{\rch_{\min}\varepsilon})$-dense and $\varepsilon$-close to $M$, together with associated estimated tangent spaces with precision $\sqrt{\varepsilon/\rch_{\min}}$, are available to the learner.
    As opposed to the sample framework, the statistical framework does not provide the learner with such data directly.
    In $\STAT(\tau)$, our strategy will therefore be to build such a point cloud and tangent spaces iteratively from queries, using the following purely geometric \MPA algorithm.

%% file: parts/mpa.tex
We now present the \MPA algorithm and its properties, which works in a setting where only geometric routines are available to the learner.
Although we will eventually apply this algorithm in the context of statistical queries (see \Cref{sec:SQ-manifold-estimation}), let us insist on the fact that the framework detailed in this \Cref{sec:mpa} is purely geometric, and does not rely specifically on statistical queries.

As mentioned in the introduction, the idea is to explore the unknown manifold $M$ via the access to only three complementary geometric routines. 
Roughly speaking, \MPA explores $M$ in a greedy way, while building a point cloud with associated tangent spaces, by using:
\begin{itemize}[leftmargin=*]
    \item 
    A seed point $\hat{x}_0 \in \R^n$, known to be close to $M$, and that allows to initialize the process.
    \item
    A tangent space routine $\hat{T}: \R^n \to \mathbb{G}^{n,d}$, that allows to make linear approximations of $M$ nearby points, and hence to provide local candidate directions to explore next.
    \item
    A projection routine $\hat{\pi} : \R^n \to \R^n$, that compensates for the errors made by the previous steps, by approximately projecting points back to $M$.
\end{itemize}
To avoid redundancy, all this is done while checking that the new candidate points
are not too close to some already-visited region of the space.
More formally, the algorithm runs as described on \cpageref{alg:MPA}.
\begin{algorithm}
    \SetAlgoLined
    \begin{algorithmic}[1]
    \SetAlgoLined
        \REQUIRE 
        ~
        \\
        Seed point $\hat{x}_0 \in \R^n$
        \\
        Tangent space routine $\hat{T}: \R^n \to \mathbb{G}^{n,d}$
        \\ 
        Projection routine $\hat{\pi} : \R^n \to \R^n$ 
        \\
        Tuning parameters $\Delta, \delta > 0$ (scales) and $
        0 < \alpha < \pi/2$ (angle)
        \STATE Initialize $\mathcal{Q} \leftarrow \{\hat{x}_0\}$, 
        $\mathcal{O} \leftarrow \emptyset$ and 
        $\mathbb{T}_\mathcal{O} \leftarrow \emptyset$
        \WHILE{$\mathcal{Q} \neq \emptyset$}
            \STATE Pick $x \in \mathcal{Q}$ 
            \STATE Set $T \leftarrow \hat{T}(x)$ and 
                $\mathbb{T}_\mathcal{O} \leftarrow \mathbb{T}_\mathcal{O} \cup \set{T}$
            \STATE Consider a maximal $(\sin \alpha)$-packing 
                $v_1, \dots, v_k$ of the sphere $\Sphere^{d-1}_T(0,1) \subseteq T$
            \FOR{$i \in \{1, \dots,k\}$}
                \IF{$\dd\bigl( x + \Delta v_i, \mathcal{Q} \cup \mathcal{O} \bigr) \ge \delta$}
                    \STATE 
                        $\mathcal{Q} \leftarrow \mathcal{Q} \cup \bigl\{ \hat{\pi}(x + \Delta v_i)\bigr\}$
                \ENDIF
            \ENDFOR
            \STATE $\mathcal{Q} \leftarrow \mathcal{Q} \setminus\{x\}$ and 
                $\mathcal{O} \leftarrow \mathcal{O} \cup \{x\}$
        \ENDWHILE
        \RETURN $\mathcal{O}$ and $\mathbb{T}_\mathcal{O}$
    \end{algorithmic}  
    \caption{\MPA}
    \label{alg:MPA}
\end{algorithm}

In spirit, \MPA is similar to the \emph{marching cube algorithm} of~\cite{Lorensen87} and the \emph{tracing algorithm} of~\cite{Kachanovich19}, which use calls to an \emph{intersection oracle} answering whether a candidate element of a partition of the ambient space intersects the manifold.
However, the approaches of~\cite{Lorensen87} and~\cite{Kachanovich19} use static partitions of $\R^n$ (cubes and a Coxeter triangulation respectively), which translates into an exploration complexity of $M$ --- measured in the number of calls made to the oracles/routines --- that strongly depends on the ambient dimension~\cite[Theorem~24]{Kachanovich19}.
In contrast, \MPA builds a point cloud nearby $M$ dynamically, which allows to adapt to its intrinsic low-dimensional geometry.
This results in an exploration complexity that is completely oblivious to the ambient space.
That is, the overall dependency in the ambient dimension is fully supported by the geometric routines themselves.
This can be explained by the intermediate tangent space estimation routine, that allows the algorithm to only explore the $d$ local (approximate) tangent directions of $M$ only, while being oblivious to the $(n-d) \gg d$ non-informative codimensions.
As a counterpart, \MPA needs to compensate for these local linear approximations which, although possibly negligible at each iteration, may cumulate into substantial deviations from the manifold after several steps.
This possible global drift is taken care of via the projection routine, which somehow reboots the precision of the process when a point is added.
To the best of our knowledge, \MPA is the first instance of an algorithm working only with the three geometric routines described above. 
We now state the main result presenting its properties.
    
\begin{theorem}[Properties of \MPA]
\label{thm:MPA_properties}
    Let $M \in \manifolds{n}{d}{\rch_{\min}}$, and assume that there exist 
    $0 \le \eta \le \Lambda < \rch_{\min}$ and $0 \le \theta < \pi/2$ such that:
    \begin{enumerate}[label=(\roman*)]
        \item \label{assumption_seed}
            $\dd(\hat{x}_0,M) \le \eta$;
        \item \label{assumption_tangent}
            For all $x \in \R^n$ such that $\dd(x,M) \le \eta$,
            $
            \normop{ 
                \pi_{T_{\pi_M(x)}M}
                -
                \pi_{\hat{T}(x)} 
            }
            \le 
            \sin \theta
            ;
            $
        \item \label{assumption_proj}
            For all $x \in \R^n$ such that $\dd(x,M) \le \Lambda$, 
            $\norm{\pi_M(x) - \hat{\pi}(x)} \le \eta$.
    \end{enumerate}
    Assume furthermore that 
    \begin{align*}
    64 \eta \leq \Delta \le \rch_{\min}/24,
    &~\max\set{\sin \alpha, \sin \theta} \le 1/ 64,
    \\
    5\Delta^2/(8\rch_{\min}) + \eta + \Delta \sin \theta \le \Lambda
    ,&~\text{and}
    ~3 \Delta /10 \le \delta \le 7 \Delta / 10
    .
    \end{align*}
    Then, \MPA terminates, and the number
    $N_{\mathrm{loop}}$ of iterations performed in the \emph{while} loop (Lines 2--12) satisfies 
    \begin{enumerate}
        \item  
        $
        N_{\mathrm{loop}} 
        \le
        \dfrac{\Haus^d(M)}{\omega_d (\delta/32)^d}
        ,
        $
        where $\Haus^d(M)$ denotes the surface area of $M$.
    \end{enumerate}
    Furthermore, it outputs a finite point cloud $\mathcal{O} \subseteq \R^n$ that:
    \begin{enumerate}\addtocounter{enumi}{1}
        \item Is $\eta$-close to $M$: $\max_{x \in \mathcal{O}} \dd(x,M) \le \eta$;
        \item Is a $(\Delta + \eta)$-covering of $M$: 
            $\max_{p \in M} \dd(p, \mathcal{O}) \le \Delta + \eta$;
    \end{enumerate}
    together with a family $\mathbb{T}_{\mathcal{O}} = 
    \bigl\{\hat{T}(x)\bigr\}_{x \in \mathcal{O}} \subseteq \Gr{n}{d}$ of linear spaces that:
    \begin{enumerate}\addtocounter{enumi}{3}
        \item $\theta$-approximate tangent spaces:
            $\max_{x \in \mathcal{O}} \normop{ \pi_{T_{\pi_M(x)}M}-\pi_{\hat{T}(x)}} \le 
                \sin \theta.$
    \end{enumerate}
\end{theorem}

To get to \Cref{thm:MPA_properties}, we will need the following series of lemmas, which are proved in
\Cref{sec:mpa-tech}.
The first statement asserts that the point clouds $\mathcal{Q}$ and $\mathcal{O}$ that the algorithm builds remain $\eta$-close to $M$ at all times. The reason for this resides in the fact that this property holds for the seed point $\hat{x}_0$ by assumption, and that the projection routine $\hat{\pi}$ maintains this $\eta$-closeness when points are added to $\mathcal{Q}$, and hence to 
$\mathcal{Q} \cup \mathcal{O}$.
\begin{lemma}
\label{lem:MPA_distance_maintained}
    Let $M \in \manifolds{n}{d}{\rch_{\min}}$, and assume that $\eta < \rch_{\min}$, 
    $\Delta \le \rch_{\min} / 4$ and 
    $\frac{5}{8} \frac{\Delta^2}{\rch_{\min}} + \eta + \Delta \sin \theta \le \Lambda$.
    Then when running \MPA, the following inequality is maintained:
    \[
        \max_{x \in \mathcal{Q} \cup \mathcal{O}} \dd(x, M) 
        \le 
        \eta
        .
    \]
\end{lemma}

The second statement ensures that points in $\mathcal{Q} \cup \mathcal{O}$ remain far away from each other, so that they always form a packing with fixed radius.
This property, maintained by the proximity test at Line 7 of \MPA, is the key ingredient for the termination of the algorithm and its complexity.
\begin{lemma}
        \label{lem:MPA_sparsity_maintained}
        Let $M \in \manifolds{n}{d}{\rch_{\min}}$, and assume that $\eta < \rch_{\min}$, 
        $\Delta \le \rch_{\min}/4$ and
        $\frac{5}{8} \frac{\Delta^2}{\rch_{\min}} + \eta + \Delta \sin \theta \le \Lambda$.
        Then when running \MPA, the following inequality is maintained:
        \[
            \min_{\substack{x, y \in \mathcal{Q} \cup \mathcal{O} \\ x \neq y}}
            \norm{x-y} 
            \ge
            \delta -  \frac{5}{8} \frac{\Delta^2}{\rch_{\min}} - 2\eta - \Delta \sin \theta
            .
        \]
\end{lemma}

The third and last statement roughly asserts that if $\MPA$ terminates, then all the $\Delta$-neighborhoods of $M$ have been visited by the output $\mathcal{O}$, i.e. that the greedy tangential exploration strategy is somehow exhaustive at scale $\Delta$.
\begin{lemma}
\label{lem:MPA_geodesic_covering}
    Let $M \in \manifolds{n}{d}{\rch_{\min}}$, and assume that $\Delta \le \rch_{\min} / 24$, $\eta < \Delta/64$, and 
    $\max\set{\sin \alpha, \sin \theta} \le 1 / 64$. Assume furthermore that,
    $\frac{5}{8} \frac{\Delta^2}{\rch_{\min}} + \eta + \Delta \sin \theta \le \Lambda$ and 
    $\delta \le 7 \Delta / 10$. If \MPA terminates, then its output $\mathcal{O}$ satisfies
    \[
        \max_{p \in M} \min_{x \in \mathcal{O}} \dd_M(p, \pi_M(x)\bigr) \le \Delta,
    \]
    where $\dd_M(\cdot,\cdot)$ is the geodesic distance of $M$.
\end{lemma}

We are now in position to prove \Cref{thm:MPA_properties}.

\begin{proof}[\proofof \Cref{thm:MPA_properties}]
    \begin{enumerate}[leftmargin=*]
        \item By construction of \MPA, any visit of the \emph{while} loop 
            (Lines 2--12) finishes with the addition of a point to $\mathcal{O}$. Since $\mathcal{O} = \emptyset$ at initialization, the number of already performed loops is maintained to satisfy $N_{\mathrm{loop}} = |\mathcal{O}|$ when the algorithm runs.
            Furthermore, by \Cref{lem:MPA_distance_maintained} and \Cref{lem:MPA_sparsity_maintained}, we have at all times
            \begin{align*}
                \min_{\substack{x, y \in \mathcal{O} \\ x \neq y}}
                &\norm{\pi_M(x)-\pi_M(y)}
                \\
                &
                \geq
                \min_{\substack{x, y \in \mathcal{Q} \cup \mathcal{O} \\ x \neq y}}
                \norm{\pi_M(x)-\pi_M(y)}
                \\
                &\ge
                \min_{\substack{x, y \in \mathcal{Q} \cup \mathcal{O} \\ x \neq y}}
                \bigl(
                \norm{x-y}
                - \norm{x - \pi_M(x)} - \norm{y - \pi_M(y)}
                \bigr)
                \\
                &\ge
                \left(\delta -  \frac{5}{8} \frac{\Delta^2}{\rch_{\min}} - 2\eta - \Delta \sin \theta\right)
                - 2 \eta
                \\
                &\ge
                \frac{173}{960} \Delta
                \ge
                \frac{173}{960} \frac{10}{7} \delta    
                >
                \frac{\delta}{4}
                >
                0.
            \end{align*}
            This shows that $\pi_M : \mathcal{O} \to \pi_M(\mathcal{O})$ is one-to-one, and that the set $\pi_M(\mathcal{O}) \subseteq M$ is a $(\delta / 8)$-packing of $M$. As a consequence, we have at all times
            \begin{align*}
                N_{\mathrm{loop}}
                =
                |\mathcal{O}|
                =
                |\pi_M(\mathcal{O})|
                \le
                \PK_M(\delta / 8)
                \le
                \dfrac{\Haus^d(M)}{\omega_d(\delta/32)^d}
                ,
            \end{align*}
            where the last inequality follows from \Cref{lem:packing_number}.
            \end{enumerate}

            As $N_{\mathrm{loop}} < \infty$, this first item also shows that \MPA terminates.
            
            \begin{enumerate}[leftmargin=*]
            \addtocounter{enumi}{1}
            
            \item This statement follows directly from \Cref{lem:MPA_distance_maintained}.
            \item We have already shown that \MPA terminates. Therefore,  \Cref{lem:MPA_geodesic_covering} applies, and combining it with Item 2, we get
            \begin{align*}
                \max_{p \in M} \dd(p, \mathcal{O})
                &=
                \max_{p \in M} \min_{x \in \mathcal{O}}
                \norm{p - x}
                \\
                &\le
                \max_{p \in M} \min_{x \in \mathcal{O}} \norm{p-\pi_M(x)}
                +
                \max_{x \in \mathcal{O}} \norm{x - \pi_M(x)}
                \\
                &\le
                \max_{p \in M} \min_{x \in \mathcal{O}} \dd_M(p,\pi_M(x))
                +
                \eta
                \\
                &\le
                \Delta + \eta
                .
            \end{align*}
            \item Follows straightforwardly from Item 2 above, and the assumption that 
            $\normop{ \pi_{T_{\pi_M(x)}M}-\pi_{\hat{T}(x)}}
                \le 
                \sin \theta
            $
            for all $x \in \R^n$ such that $\dd(x,M) \le \eta$.
    \end{enumerate}
\end{proof}

%% file: parts/SQ-routines.tex
Coming back to manifold estimation with SQ's, we notice that by combining together:
\vspace{-\topsep}
\begin{description}[font=\it]
\item[(Exploration)]
    the greedy point cloud construction of \Cref{thm:MPA_properties} using geometric routines only,

\vspace{-0.9em}
\item[(Reconstruction)]
    the point cloud-based reconstruction method of 
\Cref{thm:manifold_reconstruction_deterministic},
\end{description}
\vspace{-\topsep}

\noindent
we have reduced the problem to constructing SQ algorithms emulating these routines with a $\STAT(\tau)$ oracle.
We now present constructions of SQ algorithms for 
the projection routine $\hat{\pi}(\cdot)$ (\Cref{subsec:SQ-projection-presentation}),
the tangent space estimation routine $\hat{T}(\cdot)$ 
(\Cref{subsec:SQ-tangent-presentation}), and the seed point detection $\hat{x}_0$
(\Cref{subsec:SQ-seed-presentation}).

\subsection{Projection}
\label{subsec:SQ-projection-presentation}
Given a point $x_0 \in \R^n$ nearby $M = \supp(D)$, we aim at estimating its metric
projection $\pi_M(x_0)$ onto $M$ with statistical queries to $\STAT(\tau)$.
As mentioned earlier, the reasoning we adopt is as follows:
\begin{itemize}[leftmargin=*]
    \item
        For a properly chosen bandwidth $h>0$, the local conditional mean 
        \begin{align*}
            m_D(x_0,h)
            &=
            \E_{x \sim D}\left[ x \left| \B(x_0,h) \right. \right]
            =
            x_0
            +
            h\frac{
                \E_{x \sim D}\left[ \frac{(x-x_0)}{h} 
                    \indicator{\norm{x-x_0}\le h} \right]
            }
            {
                D(\B(x_0,h))
            }
        \end{align*} 
        of $D$ around $x_0$ has small bias for estimating $\pi_M(x_0)$ (\Cref{lem:local_conditional_mean_bias}).
    \item 
        As $m_D(x_0,h) \in \R^{n}$ writes as a functional of the two means
        $D(\B(x_0,h)) = \E_{x \sim D} [\indicator{\norm{x - x_0}\le h}] \in \R$,
        and
        $\E_{x \sim D}\left[ 
            \frac{(x-x_0)}{h} \indicator{\norm{x-x_0}\le h} 
        \right] \in \R^n$,
        it can be estimated using $2n+1$ queries (\Cref{lem:mean_vector_estimation}).
\end{itemize}
The proof of these results are to be found in \Cref{sec:SQ-projection-appendix}.
Combined together, we then prove the correctness of the SQ projection estimation 
procedure (\Cref{thm:routine_projection}) in \Cref{subsec:projection-estimation-with-SQs}.

\begin{theorem}[SQ Projection Estimation]
\label{thm:routine_projection}
    Let $D \in \distributions{n}{d}{\rch_{\min}}{f_{\min}}{f_{\max}}{L}$ have support 
    $M = \supp(D)$. Assume that 
    \[
        \frac{\tau}{\omega_d f_{\min}  \rch_{\min}^d} 
        \le 
        c^{d(d+1)} 
        \Gamma^d
        \text{~and~} 
        \Lambda \le 
        \frac{\rch_{\min}}{16}
    ,
    \]
    for some small enough absolute constant $c>0$, where 
    $
        \Gamma 
        = 
        \Gamma_{f_{\min},f_{\max},L} 
        = 
        \frac{f_{\min}}{f_{\max} + L \rch_{\min}} 
        .
    $
    
    Then for all $x_0 \in \R^n$ such that $\dd(x_0, M) \le \Lambda $, there exists
    a SQ algorithm making $2n + 1$ queries to $\STAT(\tau)$, that outputs a point
    $\hat{\pi}(x_0) \in \R^n$ estimating $\pi_M(x_0)$ with precision 
    \[
            \norm{\hat{\pi}(x_0) - \pi_M(x_0)}
            \leq
            \eta 
            =
            \frac{C^d}{\Gamma}
            \max\set{
                \frac{\Lambda^2}{\rch_{\min}}
            ,
                \Gamma^\frac{2}{d+1}
                \rch_{\min}
                \left(\frac{\tau}{\omega_d f_{\min} \rch_{\min}^d} \right)^{\frac{2}{d+1}}
            }
            ,
    \]
    where $C>0$ is an absolute constant.
\end{theorem}

\subsection{Tangent Space}
\label{subsec:SQ-tangent-presentation}

Given a point $x_0 \in \R^n$ nearby $M = \supp(D)$, we aim at estimating the tangent space $T_{\pi_M(x_0)} M$ with statistical queries to $\STAT(\tau)$. 
The strategy we propose is based on local Principal Components Analysis, combined with low-rank matrix recovery with SQ's.
This local PCA approach is similar to that of \cite{Arias13,Aamari18}.
As described above, the reasoning is as follows:
\begin{itemize}[leftmargin=*]
    \item
        For a properly chosen bandwidth $h>0$, the local (rescaled) covariance matrix 
        \[
            \Sigma_D(x_0,h)
            =
            \E_{x \sim D} \left[
                \frac{(x-x_0)\transpose{(x-x_0)}}{h^2} \mathbbm{1}_{\norm{x-x_0} \le h} 
            \right] \in \R^{n \times n}
        \] 
        of $D$ around $x_0$ is nearly rank-$d$, and its first $d$ components
        span a $d$-plane close to $T_{\pi_M(x_0)} M \in \Gr{n}{d}$
        (\Cref{lem:covariance_decomposed}).
    \item 
        Principal components being stable to perturbations 
        (\Cref{lem:angledeviation}), estimating 
        $\Sigma_D(x_0,h) \in \R^{n\times n}$ is sufficient to estimate 
        $T_{\pi_M(x_0)} M \in \Gr{n}{d}$.
    \item 
        Estimating $\Sigma_D(x_0,h) \in \R^{n\times n} = \R^{n^2}$ using 
        $O(n^2)$ queries (\Cref{lem:mean_vector_estimation}) is too costly 
        and would be redundant since $\Sigma_D(x_0,h)$ is nearly rank $d \ll n$.
        Instead, we use matrix compression arguments 
        (\Cref{thm:RIP_implies_Stable_Reconstruction}) and an explicit
        construction of a matrix sensing operator (\Cref{lem:Pauli_Satisfy_RIP}) 
        to derive a general mean low-rank matrix SQ algorithm 
        (\Cref{lem:mean_matrix_estimation}).
        This result roughly asserts that a mean matrix $\Sigma = \E_{x \sim D}[F(x)] \in \R^{n \times n}$ of a bounded function $F : \R^n \to \R^{n \times n}$ that has nearly rank $d$ can be estimated with precision $C\tau$ using 
        $n d \operatorname{polylog}(n)$ queries to $\STAT(\tau)$.
\end{itemize}
The proof of these results are to be found in \Cref{sec:SQ-tangent-appendix}.
All combined together, we then prove the correctness of the SQ tangent space estimation procedure (\Cref{thm:routine_tangent}) in \Cref{subsec:tangent-space-estimation-with-SQs}.

\begin{theorem}[SQ Tangent Space Estimation]
\label{thm:routine_tangent}
    Let $D \in \distributions{n}{d}{\rch_{\min}}{f_{\min}}{f_{\max}}{L}$ have support $M = \supp(D)$. 
    Assume that
    \[
        \frac{\tau}{\omega_d f_{\max} \rch_{\min}^d} 
        \leq 
        \left(\frac{1}{8\sqrt{d}}\right)^{d+1}
        \text{  and  }
        \eta \leq \frac{\rch_{\min}}{64d}
        .
    \]
    Then for all $x_0 \in \R^n$ such that $\dd(x_0,M) \leq \eta$, there exists a deterministic SQ algorithm making at most 
    $C d n \log^6(n)$ queries to $\STAT(\tau)$, and that outputs a $d$-plane $\hat{T}(x_0) \in \Gr{n}{d}$ estimating $T_{\pi_M(x_0)} M$ with precision
    \begin{align*}
    \normop{\pi_{\hat{T}(x_0)} - \pi_{T_{\pi_M(x_0)} M}}
    &\leq
    \sin \theta 
    \\
    &= 
    \tilde{C}^d
    \frac{f_{\max}}{f_{\min}}
        \max\set{
            \sqrt{\frac{\eta}{\rch_{\min}}} 
            ,
            \left(
                \frac{\tau}{\omega_d f_{\max} \rch_{\min}^d}
            \right)^{\frac{1}{d+1}}
        }
        ,
    \end{align*}
    where $\tilde{C}>0$ is an absolute constant.
\end{theorem}

\subsection{Seed Point}
    \label{subsec:SQ-seed-presentation}

Given a ball of radius $R>0$ guaranteed to encompass 
$M = \supp(D) \subseteq \B(0,R)$, and a target precision $\eta> 0$, we aim at finding a point that is $\eta$-close to $M$ with statistical queries to $\STAT(\tau)$.
The strategy we propose is as follows:
\begin{itemize}[leftmargin=*]
    \item
        Starting from $\B(0,R)$, we use a divide and conquer strategy (\Cref{thm:routine_detection_raw_point}).
        The algorithm (\ABS) queries indicator functions of an interactively chosen union of balls (i.e. the queried balls depend on the previous answers of the oracle), stops when there is only one ball left and outputs its center $\hat{x}_0^{raw}$.
        This method only uses estimates on the local mass of balls for $D$ (\Cref{lem:intrinsic_ball_mass}), and forgets about the differential structure and $\mathcal{C}^2$-smoothness of $M$. 
        Hence, although efficient, it only obtains a precision $O(\max\set{\eta,\tau^{1/d}})$, that can be much larger than the prescribed one $O(\max\set{ \eta, \tau^{2/(d+1)}})$.
    \item 
        Starting from $\hat{x}_0^{raw}$, we then refine this detected point by iterating the SQ projection routine $\hat{\pi}(\cdot)$ (\Cref{thm:routine_projection}), which does use extensively the $\mathcal{C}^2$-smoothness of $M$.
        As $\hat{x}_0^{raw}$ is close to $M$, this procedure is guaranteed to enhance precision quadratically at each step, and is hence satisfactory (i.e. has precision $\eta$) after a logarithmic number of iterations.

\end{itemize}
The proof of these results are to be found in \Cref{sec:SQ-seed-appendix}.

    \begin{theorem}[SQ Point Detection]
    \label{thm:routine_detection_seed_point}
        Let $D \in \distributionsball{n}{d}{\rch_{\min}}{f_{\min}}{f_{\max}}{L}{R}$ have support $M = \supp(D) \subseteq \B(0,R)$.
        Assume that
        \[
        \frac{\tau}{\omega_d f_{\min}  \rch_{\min}^d}
        \le 
        c \Gamma^d
        \min\set{
        c^{d}
        ,
            \left(
                n
                \log\left(
                        R/(\Gamma \rch_{\min}
                \right)
        \right)^{-1/2}
        }^{d},
        \text{  and  }        
        \eta
        \leq
        \frac{\rch_{\min}}{16}
        \]
        for some small enough $c>0$, where 
        $
            \Gamma 
            = 
            \Gamma_{f_{\min},f_{\max},L} 
            = 
            \frac{f_{\min}}{f_{\max} + L \rch_{\min}} 
            ,
        $
        
        Then there exists a deterministic SQ algorithm making at most $6 n \log(6R/\eta)$ queries to $\STAT(\tau)$, 
        and that outputs a point $\hat{x}_0\in \B(0,R)$ such that 
        \[
            \dd(\hat{x}_0,M)
            \le
            \max\set{
                \eta
                ,
                C^d\Gamma^{\frac{2}{d+1}-1}
                \rch_{\min}
                \left(\frac{\tau}{\omega_d f_{\min} \rch_{\min}^d} \right)^{\frac{2}{d+1}}
            }
            ,
        \]
        where $C>0$ is the absolute constant of \Cref{thm:routine_projection}.
    \end{theorem}

%% file: parts/SQ-manifold-estimation.tex
We are now in position to state the main results of this work, namely bounds on the statistical query complexity of manifold estimation in $\STAT(\tau)$.
We split the results into the two studied models 
$\distributionspoint{n}{d}{\rch_{\min}}{f_{\min}}{f_{\max}}{L}{0}$ and $ \distributionsball{n}{d}{\rch_{\min}}{f_{\min}}{f_{\max}}{L}{R}$.
For each model, we get an upper bound by combining the results of \Cref{sec:mpa,sec:SQ-routines}.
It is followed by an informational and a computational lower bound, coming from the general lower bound techniques of \Cref{sec:lower-bounds} and the constructions of \Cref{sec:manifold-lower-bound-constructions}.

\subsection{Fixed Point Model}

In $\distributionspoint{n}{d}{\rch_{\min}}{f_{\min}}{f_{\max}}{L}{0}$, the origin $0 \in \R^n$ is known to belong to $M$. 
The SQ algorithm we propose consists in running \MPA with seed point $\hat{x}_0 = 0$ and the SQ projection and tangent space routines of \Cref{thm:routine_projection,thm:routine_tangent}. This leads to the following upper bound.
Let us mention that one could easily extend this result and relax the assumption that $0 \in M$ to $\dd(0,M)$ being small enough.
\begin{theorem}
\label{thm:SQ_upper_bound_point}
    Let $D \in \distributionspoint{n}{d}{\rch_{\min}}{f_{\min}}{f_{\max}}{L}{0}$ have support $M = \supp(D)$.
    Writing  
    $
        \Gamma 
        = 
        \Gamma_{f_{\min},f_{\max},L} 
        = 
        \frac{f_{\min}}{f_{\max} + L \rch_{\min}} 
        ,
    $
    let us assume that 
    \[
        \frac{\tau}{f_{\min} \rch_{\min}^d} 
        \leq
        c^d \Gamma^{7(d+1)/2}
        \text{~and~}
        \varepsilon 
        \leq 
        \tilde{c}^d \Gamma^3
        \rch_{\min}
        ,
    \]
    for some small enough absolute constants $c,\tilde{c}>0$.
    Then there exists a deterministic SQ algorithm making at most
    \[
         q \leq 
         n \log^6 n \frac{C_d}{f_{\min}\rch_{\min}^d}
         \left( \frac{\rch_{\min}}{\varepsilon} \right)^{d/2}
    \]
    queries to $\STAT(\tau)$, and that outputs a finite triangulation $\hat{M} \subseteq \R^n$ that has the same topology as $M$, and such that
    \begin{align*}
        \dHaus(M,\hat{M})
        \leq
        \max\set{
        \varepsilon
        ,
        \frac{\tilde{C}^d}{\Gamma^3} \rch_{\min}\left( \frac{\tau}{ f_{\min} \rch_{\min}^d}  \right)^{2/(d+1)}
        }
        ,
    \end{align*}
    where $C_d>0$ depends on $d$ and $\tilde{C}>0$ is an absolute constant.
\end{theorem}
The algorithm of \Cref{thm:SQ_upper_bound_point} has a statistical query complexity comparable to the optimal sample complexity 
$s = O(\varepsilon^{-2/d}\log(1/\varepsilon))$ 
over the model $\distributions{n}{d}{\rch_{\min}}{f_{\min}}{f_{\max}}{L}$~\cite{Genovese12b,Kim2015,Aamari18},
and can provably achieve precision $O(\tau^{2/(d+1)})$.
Furthermore, the assumptions made as well as the final precision are completely insensitive to $n$.
The ambient dimension $n$ only appears as a quasi-linear factor in the query complexity. This contrasts with the sample complexity which does not depends on $n$. However, notice that a single sample nearby $M \subseteq \R^n$ consists of $n$ coordinates, while statistical queries are forced to be real-valued (one dimensional) pieces of information, which explains this apparent discrepancy.

Discussing its optimality, one may first wonder if the assumption made on $\tau$ is necessary, and whether the precision barrier of order $O(\tau^{2/(d+1)})$ is improvable in $\STAT(\tau)$. 
The following statement answers to these questions, regardless the statistical query complexity.
\begin{theorem}
\label{thm:SQ_lower_bound_point_informational}
   	    Let $\alpha < 1/2$ be a probability of error. Assume that $f_{\min} \leq f_{\max}/4$ and
   	    \[
    	    2^{d+1}\sigma_d f_{\min} \rch_{\min}^d
    	    \leq
    	    1
    	    .
   	    \]
        Then no randomized SQ algorithm can estimate $M = \supp(D)$ over $\distributionspoint{n}{d}{\rch_{\min}}{f_{\min}}{f_{\max}}{L}{0}$
        with precision 
        \[
            \varepsilon 
            <
            \frac{\rch_{\min}}{2^{21}}
            \min\set{
                \frac{1}{2^{20}d^2}
                ,
                \left(
                \frac{\tau}{\omega_d f_{\min} \rch_{\min}^d}
                \right)^{2/d}
            }
        \]
        and probability $1 - \alpha$, no matter its number of queries.
\end{theorem}
This result justifies why the quantity $\tau/(\omega_d f_{\min} \rch_{\min}^d)$ is required to be small enough in \Cref{thm:SQ_upper_bound_point}: this actually is necessary so as to reach a precision of order at least $O(\rch_{\min})$.
Second, this informational lower bound shows that the learner cannot hope to achieve precision better than 
$$
\varepsilon
=
\Omega\bigl(
\rch_{\min} 
\left(
\tau/(\omega_d f_{\min} \rch_{\min}^d)
\right)^{2/d}
\bigr)
,
$$
even with the most costly randomized SQ algorithms. In fact, the precision $O(\tau^{2/(d+1)})$ of \Cref{thm:SQ_upper_bound_point} is nearly optimal.
Here, the assumptions made on $f_{\min},f_{\max}$ and $\rch_{\min}$ are also necessary to ensure non-degeneracy of the model $\distributionspoint{n}{d}{\rch_{\min}}{f_{\min}}{f_{\max}}{L}{0}$, as mentioned in \Cref{subsubsec:implicit-bounds-on-parameters}.
Beyond the above informational considerations, we turn to the computational one, i.e. to the minimal number of queries to $\STAT(\tau)$ that a learner must make to achieve precision $\varepsilon$.

\begin{theorem}
\label{thm:SQ_lower_bound_point_computational}
    Let $\alpha < 1$ be a probability of error, and $\varepsilon \leq \rch_{\min}/(2^{34}d^2)$. 
    Assume that $f_{\min} \leq f_{\max}/4$ and
   	    \[
    	    2^{d+1}\sigma_d f_{\min} \rch_{\min}^d
    	    \leq
    	    1
    	    .
   	    \]
    Then any randomized SQ algorithm estimating $M = \supp(D)$ over $\distributionspoint{n}{d}{\rch_{\min}}{f_{\min}}{f_{\max}}{L}{0}$ with precision $\varepsilon$ and with probability of success at least $1 - \alpha$ must make at least 
        \[
            q
            \geq 
            \left(
                n
        		\frac{1}{\omega_d f_{\min} \rch_{\min}^d}
    	    	\left(
        		\frac{\rch_{\min}}{2^{21} \varepsilon}
        		\right)^{d/2}
        		+
        		\log(1-\alpha)
            \right)
            \big/
            \log (1+1/\tau)
        \]
        queries to $\STAT(\tau)$.
\end{theorem}

For deterministic SQ algorithms ($\alpha = 0$), the statistical query complexity of the algorithm of \Cref{thm:SQ_upper_bound_point} is therefore optimal up to $\operatorname{polylog}(n,1/\tau)$ factors.
It even performs nearly optimally within all the possible randomized algorithms, provided that their probability of error $\alpha$ is not too close to $1$, which would allow for a naive random pick among an $\varepsilon$-covering of the space $\bigl(\manifoldspoint{n}{d}{\rch_{\min}}{0},\dHaus\bigr)$ (with zero query to $\STAT(\tau)$) to be a valid algorithm.

\subsection{Bounding Ball Model}

In $\distributionsball{n}{d}{\rch_{\min}}{f_{\min}}{f_{\max}}{L}{R}$, no distinguished point of $\R^n$ is known to belong to $M$, but a location area $\B(0,R)$ containing $M$ is available to the learner. Hence, the strategy of the previous section cannot initialize directly.
However, \Cref{thm:routine_detection_seed_point} allows to find a seed point $\hat{x}_0$ close to $M$ using a limited number of queries to $\STAT(\tau)$.
Starting from $\hat{x}_0$ and, as above, running \MPA with the SQ projection and tangent space routines of \Cref{thm:routine_projection,thm:routine_tangent} leads to the following upper bound.

\begin{theorem}
\label{thm:SQ_upper_bound_ball}
    Let $D \in \distributionsball{n}{d}{\rch_{\min}}{f_{\min}}{f_{\max}}{L}{R}$ have support $M = \supp(D)$.
    Writing  
    $
        \Gamma 
        = 
        \Gamma_{f_{\min},f_{\max},L} 
        = 
        \frac{f_{\min}}{f_{\max} + L \rch_{\min}} 
        ,
    $
    let us assume that 
    \[
        \frac{\tau}{f_{\min} \rch_{\min}^d} 
        \leq
        \min\set{
            c^d \Gamma^{7(d+1)/2}
        ,
            \Gamma^d \left(
                n
                \log\left(
                        R/(\Gamma \rch_{\min}
                \right)
        \right)^{-d/2}
        }
    \]
    and
    \[
    \varepsilon 
        \leq 
        \tilde{c}^d \Gamma^3
        \rch_{\min}
        ,
    \]
    for some small enough absolute constants $c,\tilde{c}>0$.
    Then there exists a deterministic SQ algorithm making at most
    \[
     q \leq 
     Cn \log \left( \frac{R}{\varepsilon} \right)
     +
     n \log^6 n \frac{C_d}{f_{\min}\rch_{\min}^d}
     \left( \frac{\rch_{\min}}{\varepsilon} \right)^{d/2}
    \]
    queries to $\STAT(\tau)$, and that outputs a finite triangulation $\hat{M} \subseteq \R^n$ that has the same topology as $M$, and such that
    \begin{align*}
        \dHaus(M,\hat{M})
        \leq
        \max\set{
        \varepsilon
        ,
        \frac{\tilde{C}^d}{\Gamma^3} \rch_{\min}\left( \frac{\tau}{ f_{\min} \rch_{\min}^d}  \right)^{2/(d+1)}
        }
        ,
    \end{align*}
    where $C_d>0$ depends on $d$ and $\tilde{C}>0$ is an absolute constant.
    \end{theorem}
Compared to \Cref{thm:SQ_upper_bound_point}, observe the extra $O(n\log(R/\varepsilon))$ queries made in \Cref{thm:SQ_upper_bound_ball}, which come from the seed point search performed at initialization.
Passed this difference, the two results are tightly similar.
In the same fashion as above, we first discuss the necessity of the assumptions made on $\tau$ and the precision threshold $O(\tau^{2/(d+1)})$.

\begin{theorem}
\label{thm:SQ_lower_bound_ball_informational}
    Let $\alpha < 1/2$ be a probability of error.
        Assume that $\rch_{\min} \leq R/144$ and $f_{\min} \leq f_{\max}/96$, and
        \[
            \min_{1 \leq k \leq n}
        	\left(\frac{192\rch_{\min} \sqrt{k}}{R}\right)^k
    	       \leq
            36 \times 8^{d} \sigma_{d-1} f_{\min} \rch_{\min}^d
        	\leq
        	1
        	.
        \]
        Then no randomized SQ algorithm can estimate $M = \supp(D)$ over $\distributionsball{n}{d}{\rch_{\min}}{f_{\min}}{f_{\max}}{L}{R}$
        with precision 
        \[
            \varepsilon 
            <
                \frac{\rch_{\min}}{2^{31}}
                \min\set{
                    \frac{1}{2^{10}d^2}
                    ,
                    \left(
                    \frac{\tau}{\omega_d f_{\min} \rch_{\min}^d}
                    \right)^{2/d}
                }
        \]
        and probability $1 - \alpha$, no matter its number of queries.
\end{theorem}

As above, we emphasize the fact that the assumptions made on $f_{\min}$, $f_{\max}$, $\rch_{\min}$ and $R$ are necessary to guarantee the non-degeneracy of the model $\distributionsball{n}{d}{\rch_{\min}}{f_{\min}}{f_{\max}}{L}{R}$, and hence a non-trivial estimation problem (see \Cref{subsubsec:implicit-bounds-on-parameters}).
As for the fixed point model, we hence see that the assumptions made on $\tau$ and the precision $\varepsilon$ cannot be omitted.
Let us though notice the slightly more stringent assumption made on $\tau$ that depends on $n$ in the upper bound (\Cref{thm:SQ_upper_bound_ball}) but not in the lower bound (\Cref{thm:SQ_lower_bound_ball_informational}).
This dependency originates from the seed point detection method that we developed (\Cref{thm:routine_detection_seed_point}) and we do not claim it to be optimal.
As about the computational lower bound for this model, we state the following result.

\begin{theorem}
\label{thm:SQ_lower_bound_ball_computational}
    
    Let $\alpha < 1$ be a probability of error, and $\varepsilon \leq \rch_{\min}/(2^{34}d^2)$.
    Assume that $\rch_{\min} \leq R/144$, $f_{\min} \leq f_{\max}/96$, and
        \[
            \min_{1 \leq k \leq n}
        	\left(\frac{192\rch_{\min} \sqrt{k}}{R}\right)^k
    	       \leq
            36 \times 8^{d} \sigma_{d-1} f_{\min} \rch_{\min}^d
        	\leq
        	1
        	.
        \]
    Then any randomized SQ algorithm estimating $M = \supp(D)$ over $\distributionsball{n}{d}{\rch_{\min}}{f_{\min}}{f_{\max}}{L}{R}$ with precision $\varepsilon$ and with probability of success at least $1 - \alpha$ must make at least 
        \[
            q
            \geq 
            \dfrac{
                n
        		\max\set{
        		    \log\left(\frac{R}{4\varepsilon} \right)
            		,
            		\frac{1}{\omega_d f_{\min} \rch_{\min}^d}
    	        	\left(
            		\frac{\rch_{\min}}{2^{31} \varepsilon}
            		\right)^{d/2}
        		}
        		+
        		\log(1-\alpha)
        	}{
            \log (1+1/\tau)
            }
        \]
        queries to $\STAT(\tau)$.

\end{theorem}

As a result, the extra $O(n \log(R / \varepsilon))$ queries of \Cref{thm:SQ_upper_bound_ball} are necessary in the SQ framework.
This contrasts sharply with the sample model, where no prior location information is necessary appears in the sample complexity~\cite{Genovese12b,Kim2015,Aamari18}.
Roughly speaking, this is explained by the fact that a single sample (the first, say) does provide location information for free, while in the SQ framework, the learner is left with the whole ball $\B(0, R)$ to explore at initialization.
However, as mentioned above, note that the term $\Omega(n \log(R / \varepsilon))$ attributable to this initialization step would only dominate in the regime where $R$ is exponentially bigger than $\rch_{\min}$.

%% file: parts/outro.tex
As mentioned above, a byproduct of these results is that manifold estimation is possible in a locally private way. However, the transformation used to pass from statistical
query learning to locally private learning has a polynomial blowup~\cite{KLNRS11:what-can-we-learn-privately}.
Hence, the derived locally private upper bound may not be optimal, so that a close study of the private framework directly is still necessary.
Coming back to SQ's, the derived bounds on the best achievable precision $\varepsilon$ in $\STAT(\tau)$ do not match, as they are respectively of the form
$
    \varepsilon
    =
    O\left(
    \rch_{\min}
    \left(
        \frac{\tau}{f_{\min}\rch_{\min}^d}
    \right)^{2 / (d + 1)}
    \right)
$ 
for the upper bounds, and
$
    \varepsilon
    =
    \Omega
    \left(
    \rch_{\min}
    \left(
        \frac{\tau}{f_{\min}\rch_{\min}^d}
    \right)^{2 / d}
    \right)
$
for the lower bounds, so that this gap remains to be breached.

The case of smoother $\mathcal{C}^k$ manifolds ($k>2$) would be of fundamental interest, as their estimation is at the basis of plugin methods of higher order. This includes, for instance, density estimation~\cite{Berenfeld20}, or distribution estimation in Wasserstein distance~\cite{Divol21}. 
For $k>2$, local linear approximations are not optimal \cite{Aamari19b}, and local polynomials of higher order --- that were shown to be optimal in the sample framework --- might adapt to statistical queries.

On a more technical side, note that we have assumed throughout that the density $f$ is $L$-Lipschitz and satisfies $f_{\min} \leq f \leq f_{\max}$, although the lower bounds do not let $L$ and $f_{\max}$ appear, but only $f_{\min}$.
While the Lipschitz assumption could be dropped for the tangent space routine, it actually is crucial in the proposed projection routine to bound the bias term. Hence, it remains unclear to us how to design an efficient projection routine without this assumption, as well as how to carry the whole analysis with $f_{\max} = \infty$.

\section*{Acknowledgements}
\addcontentsline{toc}{section}{Acknowledgements}
This work was partially funded by CNRS PEPS JCJC.
The authors would like to thank warmly the \textit{Department of Mathematics of UC San Diego}, the \textit{Steklov Institute}, and all the members of the \textit{Laboratoire de Probabilités, Statistiques et Modélisation} for their support and insightful comments.

%% file: parts/mpa-technical-lemmas.tex
When running \MPA, linear approximations of the manifold are done via its (approximate) tangent spaces. A key point in the proof of its correctness is the (quantitative) validity of this approximation, which is ensured by the reach assumption $\rch_M \geq \rch_{\min}$, which bounds curvature.
Recall from \eqref{eq:offset} that $M^r = \{z \in \R^n, \dd(z,M) \leq r\}$ stands for the $r$-offset of $M$.
\begin{lemma}
        \label{lem:MPA_tangent_error_propagation}
        Let $M \in \manifolds{n}{d}{\rch_{\min}}$, and $x \in M^\eta$ with $\eta < \rch_{\min}$. Take
        $T \in \Gr{n}{d}$ such that $\normop{ \pi_{T_{\pi_M(x)}M}-\pi_{T}} \le \sin \theta$.
        Then for all $\Delta \le \rch_{\min}/4$, and all unit vector $v \in T$,
        \[
            \dd\bigl( x + \Delta v, M\bigr) 
            \le
            \frac{5}{8} \frac{\Delta^2}{\rch_{\min}} + \eta + \Delta \sin \theta
            .
        \]
    \end{lemma}

    \begin{proof}[\proofof \Cref{lem:MPA_tangent_error_propagation}]
        By assumption on $T$, there exists a unit vector $v' \in T_{\pi_M(x)}M$ such that 
        $\norm{v-v'} \le \sin \theta$. Hence, since $\dd(\cdot,M)$ is $1$-Lipschitz, we have
        \begin{align*}
            \dd\bigl(x+\Delta v,M\bigr) 
            &\le
            \dd\bigl(x+\Delta v',M\bigr) + \Delta \sin \theta
            \\
            &\le
            \dd\bigl(\pi_M(x)+\Delta v',M\bigr) + \eta + \Delta \sin \theta
            \\
            &\le
            \norm{\pi_M(x)+\Delta v' - \exp_{\pi_M(x)}^M(\Delta v')} + \eta + \Delta \sin \theta
            \\
            &\le
            \frac{5}{8} \frac{\Delta^2}{\rch_{\min}} + \eta + \Delta \sin \theta
            ,
        \end{align*}
        where the last inequality follows from \cite[Lemma 1]{Aamari19b}.
    \end{proof}

We are now in position to prove \Cref{lem:MPA_distance_maintained}, that guarantees that \MPA builds point clouds that do not deviate from $M$.

\begin{proof}[\proofof \Cref{lem:MPA_distance_maintained}]
    The points added to $\mathcal{O}$ are all first added to $\mathcal{Q}$: therefore, it is sufficient 
    to check that all the points $x$ added to $\mathcal{Q}$ satisfy
    $\dd(x,M) \le \eta$.
    To see this, proceed by induction:
    \begin{itemize}[leftmargin=*]
        \item As $\mathcal{Q}$ is initialized to $\set{\hat{x}_0}$ with $\dd(\hat{x}_0,M) \le \eta$,
            the inequality holds true at Line 1, before the first loop.
        \item If $\bar{x} \neq \hat{x}_0$ was added to $\mathcal{Q}$, it can be written as 
            $\bar{x} = \hat{\pi}(x_0+\Delta v_i)$, for some point 
            $x_0 \in \mathcal{Q}$ and a unit vector $v_i \in \hat{T}(x_0)$.
            By induction, we have $\dd(x_0,M)\leq \eta$. 
            But since $\hat{T}(\cdot)$ is assumed to have precision $\sin \theta$ over $M^\eta$, we hence obtain that
            $\normop{ \pi_{T_{\pi_M(x_0)}M}-\pi_{\hat{T}(x_0)}} \le \sin \theta$.
            As a result, from \Cref{lem:MPA_tangent_error_propagation},
            \begin{align*}
                \dd(x_0+\Delta v_i,M)
                \le
                \frac{5}{8} \frac{\Delta^2}{\rch_{\min}} + \eta + \Delta \sin \theta \le \Lambda,
            \end{align*}
            and therefore 
            \[
                \dd(\bar{x},M) 
                \le 
                \norm{\bar{x} - \pi_M(x_0 + \Delta v_i)}
                = 
                \norm{\hat{\pi}(x_0 + \Delta v_i) - \pi_M(x_0 + \Delta v_i)} 
                \le 
                \eta
            \]
            since $\hat{\pi}(\cdot)$ is assumed to have precision $\eta$ over $M^\Lambda$.
    \end{itemize}
    This concludes the induction and hence the proof.
\end{proof}

Next we show \Cref{lem:MPA_sparsity_maintained}, asserting that the radius of sparsity of the point clouds built by \MPA is maintained at all times.
\begin{proof}[\proofof \Cref{lem:MPA_sparsity_maintained}]
    At initialization of \MPA, $\mathcal{Q} \cup \mathcal{O} = \set{\hat{x}_0}$, so that the inequality
    trivially holds at Line 1. Then, if a point $\bar{x}$ is added to $\mathcal{Q}$ at Line 8, it
    means that it can be written as $\bar{x} = \hat{\pi}(x_0+\Delta v_{i_0})$, with 
    $\dd(x_0 + \Delta v_{i_0}, \mathcal{Q} \cup \mathcal{O}) \ge \delta$. Consequently, by 
    induction, we have
    \begin{align*}
        \min_{\substack{x, y \in \mathcal{Q} \cup \mathcal{O} \cup \set{\bar{x}} \\ x \neq y}}
        \norm{x - y} 
        &=
        \min\Bigl\{
            \min_{\substack{x, y \in \mathcal{Q} \cup \mathcal{O} \\ x \neq y}}
            \norm{x-y},
            \dd\left(\bar{x}, \mathcal{Q} \cup \mathcal{O}\right)
        \Bigr\}
        \\
        &\ge
        \min\Bigl\{
            \delta -  \frac{5}{8} \frac{\Delta^2}{\rch_{\min}} - 2\eta - \Delta \sin \theta,
            \\
            &
            \hspace{4em}
            \dd(x_0+\Delta v_{i_0}, \mathcal{Q} \cup \mathcal{O})
            -
            \norm{(x_0+\Delta v_{i_0}) - \bar{x}}
        \Bigr\}
        \\
        &\ge
        \min\left\{
            \delta -  \frac{5}{8} \frac{\Delta^2}{\rch_{\min}} - 2\eta - \Delta \sin \theta,
            \delta - \norm{(x_0+\Delta v_{i_0}) - \bar{x}}
        \right\}
        .
    \end{align*}
    In addition, \Cref{lem:MPA_tangent_error_propagation} and 
    \Cref{lem:MPA_distance_maintained} combined yield
    \begin{align*}
        \norm{(x_0+\Delta v_{i_0}) - \bar{x}}
        &=
        \norm{(x_0+\Delta v_{i_0}) - \hat{\pi}(x_0+\Delta v_{i_0})}
        \\
        &\le
        \norm{\hat{\pi}(x_0+\Delta v_{i_0}) - \pi_M(x_0+\Delta v_{i_0})}
        \\
        &~~~+
        \norm{\pi_M(x_0+\Delta v_{i_0}) - (x_0+\Delta v_{i_0})}
        \\
        &\le
        \eta
        +
        \left(
            \frac{5}{8} \frac{\Delta^2}{\rch_{\min}} + \eta + \Delta \sin \theta
        \right)
        .
    \end{align*}
    As a result, after the update $\mathcal{Q} \leftarrow \mathcal{Q}\cup \set{\bar{x}}$, the announced inequality still holds. 
    Finally, we notice that Line 11, which swaps a point from $\mathcal{Q}$ to $\mathcal{O}$, leaves $\mathcal{Q} \cup \mathcal{O}$ 
    unchanged. By induction, this concludes the proof.
\end{proof}

Finally we prove \Cref{lem:MPA_geodesic_covering}, that states that if \MPA terminates, it outputs a point cloud dense enough nearby $M$.

\begin{proof}[\proofof \Cref{lem:MPA_geodesic_covering}]
    Assume for contradiction that there exists $p_0 \in M$ such that for all $x \in \mathcal{O}$, 
    $\dd_M\bigl(p_0, \pi_M(x)\bigr) > \Delta$. 
    Let $x_0 \in \mathcal{O}$ (which is not empty since $\hat{x}_0 \in \mathcal{O}$) be such that 
    \[
        \dd_M\bigl(p_0, \pi_M(x_0)\bigr) 
        = 
        \min_{x \in \mathcal{O}} \dd_M(p_0, \pi_M(x)\bigr) 
        := r_0
        >
        \Delta
        ,
    \]
    and write $y_0 := \pi_M(x_0)$. 
    Let $\gamma := \gamma_{y_0 \to p_0} : [0, r_0] \to M$ denote an arc-length 
    parametrized geodesic joining $y_0$ and $p_0$. Finally, set 
    $q_0 := \gamma(\Delta) \in M$ and $v_0 := \gamma'(0) \in T_{y_0}M$.
    
    Consider the sets $\mathcal{Q}$ and $\mathcal{O}$ of \MPA right 
    after $x_0$ was removed from $\mathcal{Q}$ and added to $\mathcal{O}$ (Line 11).
    By construction, all the elements $v_1, \dots, v_k$ of a maximal 
    $(\sin \alpha)$-packing of $\Sphere^{d-1}_{\hat{T}(x_0)}$ were tested to 
    enter $\mathcal{Q}$ (Loop from Line 6 to Line 10). Because the packing is maximal, it is also a $(2 \sin \alpha)$-covering of $\Sphere^{d-1}_{\hat{T}(x_0)}$ (see the proof of \Cref{prop:packing_covering_link}).
    As a result, by assumption on the precision of 
    $\hat{T}(x_0)$, there exists $v_{i_0}$ in this packing   such that 
    $\norm{v_0 - v_{i_0}} \le 2 \sin \alpha + \sin \theta$.
    
    As $\gamma$ is a distance-minimizing path on $M$ from $y_0$ to $p_0$, so it is 
    along its two sub-paths with endpoint $q_0$, as otherwise, one could build a 
    strictly shorter path between $y_0$ and $p_0$. In particular, since 
    $\Delta < r_0 = \dd_M(y_0,p_0)$, we have 
    $\dd_M(y_0,q_0) = \dd_M(y_0, \gamma(\Delta)) = \Delta$ and $\dd_M(p_0,q_0) = 
    \dd_M(p_0, \gamma(\Delta)) =  r_0 - \Delta$.
    As a result,
    \begin{align}
            \label{eq:lem:MPA_geodesic_covering:pushed_point_distance}
            \dd_M\left( p_0, \pi_M\left(\hat{\pi}\left(x_0+\Delta v_{i_0}\right)\right) \right)
            &\le
            \dd_M\left( p_0, q_0 \right)
            +
            \dd_M\left( q_0, \pi_M\left(\hat{\pi}\left(x_0+\Delta v_{i_0}\right)\right) \right)
            \nonumber
            \\
            &=
            r_0 - \Delta
            +
            \dd_M\left( q_0, \pi_M\left(\hat{\pi}\left(x_0+\Delta v_{i_0}\right)\right) \right).
        \end{align}
        But from \Cref{lem:geodesic_comparison}, we get
        \begin{align}
            \label{eq:lem:MPA_geodesic_covering:geodesic}
            \dd_M\left( q_0, \pi_M\left(\hat{\pi}\left(x_0+\Delta v_{i_0}\right)\right) \right)
            &\le
            2 \rch_{\min}
            \arcsin{\left(\frac{\norm{q_0-\pi_M\left(\hat{\pi}\left(x_0+\Delta v_{i_0}\right)\right)}}{2\rch_{\min}}\right)}
            \nonumber
            \\
            &\le
            \frac{
                \norm{q_0-\pi_M\left(\hat{\pi}\left(x_0+\Delta v_{i_0}\right)\right)}
            }{
                \sqrt{
                    1 - \left( 
                        \frac{
                            \norm{
                                q_0 - \pi_M \left( 
                                    \hat{\pi} \left( x_0 + \Delta v_{i_0}\right)
                                \right)
                            }
                        }{
                            2\rch_{\min}
                        } 
                    \right)^2 
                }
            },
        \end{align}
        and furthermore,
        \begin{align}
        \label{eq:lem:MPA_geodesic_covering:triangle}
            \norm{q_0-\pi_M\left(\hat{\pi}\left(x_0+\Delta v_{i_0}\right)\right)}
            &\le
            \norm{q_0 - (y_0 + \Delta v_0)}
            +
            \norm{(y_0 + \Delta v_0) - (x_0 + \Delta  v_{i_0})}
            \nonumber
            \\
            &~~~
            +
            \norm{
                (x_0 + \Delta  v_{i_0}) 
                - 
                \hat{\pi}\left(x_0+\Delta v_{i_0}\right)
                }
            \nonumber
            \\
            &~~~
            +
            \norm{
                \hat{\pi}\left(x_0+\Delta v_{i_0}\right)
                -
                \pi_M\left(\hat{\pi}\left(x_0+\Delta v_{i_0}\right)\right)
                }
            .
        \end{align}
        We now bound the right hand side of \Cref{eq:lem:MPA_geodesic_covering:triangle} term by term.
        The first term is bounded by
        \begin{align*}
            \norm{q_0 - (y_0 + \Delta v_0)}
            &=
            \norm{\gamma(\Delta) - (\gamma(0) + \Delta \gamma'(0))}
            \le
            \frac{\Delta^2}{2\rch_{\min}},
        \end{align*}
        where the inequality follows from a Taylor expansion and \Cref{lem:geodesic_comparison}. 
        For the second term, write
        \begin{align*}
            \norm{(y_0 + \Delta v_0) - (x_0 + \Delta  v_{i_0})}
            &\le
            \norm{y_0 - x_0}
            +
            \Delta \norm{v_0 - v_{i_0}}
            \le
            \eta + \Delta (2 \sin \alpha + \sin \theta).
        \end{align*}
        For the third term, we combine \Cref{lem:MPA_tangent_error_propagation} and 
        \Cref{lem:MPA_distance_maintained} to get
        \begin{align*}
            \norm{
                (x_0 + \Delta  v_{i_0}) 
                - 
                \hat{\pi}\left(x_0+\Delta v_{i_0}\right)
                }
            &\le
                \dd\left(x_0 + \Delta  v_{i_0},M\right)
            \\
            &~~~+
            \norm{
                \pi_M(x_0 + \Delta  v_{i_0}) 
                - 
                \hat{\pi}\left(x_0+\Delta v_{i_0}\right)
                }
            \\
            &\le
            \frac{5}{8} \frac{\Delta^2}{\rch_{\min}} + 2\eta + \Delta \sin \theta
            ,
        \end{align*}
        and for the fourth term, applying again \Cref{lem:MPA_tangent_error_propagation} and \Cref{lem:MPA_distance_maintained} yields
        \begin{align*}
            \norm{
                \hat{\pi}\left(x_0+\Delta v_{i_0}\right)
                -
                \pi_M\left(\hat{\pi}\left(x_0+\Delta v_{i_0}\right)\right)
                }
                &=
                \dd\bigl( \hat{\pi}\left(x_0+\Delta v_{i_0}\right) , M \bigr)
                \\
                &\le
                \eta
                .
        \end{align*}
        Plugging these four bounds in \Cref{eq:lem:MPA_geodesic_covering:triangle}, we have shown that
        \begin{align}
            \label{eq:lem:MPA_geodesic_covering:triangle_simplified}
            \norm{q_0-\pi_M\left(\hat{\pi}\left(x_0+\Delta v_{i_0}\right)\right)}
            &\le
            \frac{9\Delta^2}{8\rch_{\min}} + 4\eta + 2\Delta(\sin \alpha + \sin \theta)
            .
        \end{align}
        Combining \Cref{eq:lem:MPA_geodesic_covering:triangle_simplified},
        \Cref{eq:lem:MPA_geodesic_covering:geodesic}, and the assumptions on the parameters 
        $\Delta, \eta, \theta, \alpha$ hence yields
        \begin{align*}
             \dd_M\left( y_0, \pi_M\left(\hat{\pi}\left(x_0+\Delta v_{i_0}\right)\right) \right)
             \le
             2 \norm{q_0-\pi_M\left(\hat{\pi}\left(x_0+\Delta v_{i_0}\right)\right)}
             \le
             \Delta/2,
        \end{align*}
        so that \Cref{eq:lem:MPA_geodesic_covering:pushed_point_distance} gives
        \begin{align*}
            \dd_M\left( p_0, \pi_M\left(\hat{\pi}\left(x_0+\Delta v_{i_0}\right)\right) \right)
            &\le
            r_0 - \Delta + \Delta/2 
            \\
            &< r_0 
            \\
            &= 
            \dd_M(p_0, \pi_M(x_0)) 
            = 
            \min_{x \in \mathcal{O}} \dd_M(p_0, \pi_M(x)\bigr).
        \end{align*}
        In particular, $\hat{\pi}\left(x_0+\Delta v_{i_0}\right)$ was not added to $\mathcal{Q}$ in the
        Loop of Lines 6 to 10 investigating the neighbors of $x_0$ (i.e.  when $x_0$ was picked Line 3). 
        Since $\mathcal{Q} \cup \mathcal{O}$ is an increasing sequence of sets as \MPA runs and
        that $\mathcal{Q} = \emptyset$ when it terminates, this means that there exists $x_1$ in the 
        final output $\mathcal{O}$ such that $\norm{x_0 + \Delta v_{i_0} - x_1} \le \delta$.
        
        The existence of this particular point $x_1$ in $\mathcal{O}$ which is $\delta$-close to 
        $x_0 + \Delta v_{i_0}$ will lead us to a contradiction: we will show that $\pi_M(x_1)$ will be 
        closer to $p_0$ than $\pi_M(x_0)$ is in geodesic distance.        
        To get there, we first notice that any such $x_1 \in \mathcal{O}$ would satisfy 
        $\dd(x_1,M) \le \eta$ from \Cref{lem:MPA_distance_maintained}, so that   
        \begin{align*}
            \norm{\pi_M\left(x_0+\Delta v_{i_0}\right) - \pi_M(x_1)}
            &\le
            \norm{\pi_M\left(x_0+\Delta v_{i_0}\right) - \left(x_0+\Delta v_{i_0}\right)}
            \\
            &~~~+
            \norm{\left(x_0+\Delta v_{i_0}\right) - x_1}
            +
            \norm{x_1 - \pi_M(x_1)}
            \\
            &\le
            \delta + \frac{5}{8} \frac{\Delta^2}{\rch_{\min}} + 2\eta + \Delta \sin \theta
            \\
            &\le
            \delta + \frac{17}{192} \Delta
            \le
            \frac{25}{192}\rch_{\min}
            ,
        \end{align*}
        where the last-but-one line follows from \Cref{lem:MPA_tangent_error_propagation}, and the last 
        one from the assumptions on the parameters $\Delta, \eta, \theta$ and $\delta$.
        As a result, from \Cref{lem:geodesic_comparison},
        \begin{align}
            \label{eq:lem:MPA_geodesic_covering:contradiction_1}
            \dd_M(\pi_M\left(x_0+\Delta v_{i_0}\right), \pi_M(x_1))
            &\le
            \frac{\norm{\pi_M\left(x_0+\Delta v_{i_0}\right) - \pi_M(x_1)}}{\sqrt{1- \left(\frac{25}{2 \times 192} \right)^2}}
            \nonumber
            \\
            &\le
            \left(1 + \frac{3}{10000} \right) \left( \delta + \frac{17}{192} \Delta \right)
            .
        \end{align}
        Furthermore, using a similar decomposition as for
        \Cref{eq:lem:MPA_geodesic_covering:triangle_simplified}, we have
        \begin{align*}
            \norm{q_0 - \pi_M\left(x_0+\Delta v_{i_0}\right)}
            &\le
            \norm{q_0 - \left(y_0+\Delta v_{0}\right)}
            + \norm{\left(y_0+\Delta v_{0}\right) - \left(x_0+\Delta v_{i_0}\right)}
            \\
            &~~~+
            \norm{\left(x_0+\Delta v_{i_0}\right) - \pi_M\left(x_0+\Delta v_{i_0}\right)}
            \\
            &\le
            \frac{\Delta^2}{2 \rch_{\min}}
            +\left( \eta + \Delta \left( 2 \sin \alpha + \sin \theta \right) \right) 
            \\
            &~~~+
            \left( \frac{5}{8} \frac{\Delta^2}{\rch_{\min}} + \eta + \Delta \sin \theta \right)
            \\
            &\le
            \frac{11}{64} \Delta \le \frac{11}{1536} \rch_{\min}
            ,
        \end{align*}
        from which we finally get
        \begin{align}
            \label{eq:lem:MPA_geodesic_covering:contradiction_2}
            \dd_M\left( q_0, \pi_M\left(x_0+\Delta v_{i_0}\right)\right)
            &\le
            \frac{\norm{q_0 - \pi_M\left(x_0+\Delta v_{i_0}\right)}}{\sqrt{1- \left(\frac{11}{2\times 1536} \right)^2}}
            \nonumber
            \\
            &\le
            \frac{3}{16} \Delta
            .
        \end{align}
        This takes us to the desired contradiction, since:
        \begin{itemize}[leftmargin=*]
        \item 
            on one hand, $x_1 \in \mathcal{O}$ forces to have
            \begin{align*}
                \dd_M(p_0, \pi_M(x_1))
                \geq
                r_0
                =
                \min_{x \in \mathcal{O}} \dd_M(p_0, \pi_M(x)\bigr)
                =
                \dd_M(p_0, \pi_M(x_0))
                ;
            \end{align*}
        \item
            on the other hand, \Cref{eq:lem:MPA_geodesic_covering:contradiction_1} and \Cref{eq:lem:MPA_geodesic_covering:contradiction_2} combined yield 
            \begin{align*}
                \dd_M(p_0, \pi_M(x_1))
                &\le
                \dd_M(p_0,q_0) 
                + 
                \dd_M\left( q_0, \pi_M\left(x_0+\Delta v_{i_0}\right)\right) 
                \\
                &~~~+ 
                \dd_M(\pi_M\left(x_0+\Delta v_{i_0}\right), \pi_M(x_1))
                \\
                &\le
                r_0 - \Delta 
                + 
                \frac{3}{16} \Delta
                + 
                \left(1 + \frac{3}{10000} \right) \left( \delta +     \frac{17}{192} \Delta \right) 
                \\
                & < r_0,
            \end{align*}
            where we used that $\delta \le 7 \Delta / 10$. 
        \end{itemize}
        As a result, we have proved that
        \[
            \max_{p \in M} \min_{x \in \mathcal{O}} \dd_M(p, \pi_M(x)\bigr) \le \Delta,
        \]
        which is the announced result.
    \end{proof}

%% file: parts/miscellaneous.tex
\subsection{Local Mass of Balls Estimates}

To prove the properties of the statistical query routines, we will need the following two geometric results about manifolds with bounded reach.
In what follows, $t_+ := \max\{0,t \}$ stands for the positive part of $t \in \R$.

\begin{proposition}[{\cite[Proposition 8.2]{Aamari18}}]
\label{prop:ball_projection}
    Let $M \in \manifolds{n}{d}{\rch_{\min}}$, $x \in \R^n$ such that $\dd(x,M) \le  \rch_{\min}/8$, 
    and $h \le \rch_{\min}/8$. Then,
    \begin{align*}
         \B\left(\pi_M(x), r_h^- \right )\cap M 
         \subseteq 
         \B(x,h) \cap M 
         \subseteq 
         \B \left(\pi_M(x), r_h^+ \right ) \cap M,
    \end{align*}
    where $r_h = (h^2- \dd(x,M)^2)_+^{1/2}$, $(r_h^-)^2 = 
    \left (1-\frac{\dd(x,M)}{\rch_{\min}}\right )r_h^2$,  
    and $(r_h^+)^2 = \left (1+\frac{ 2 \dd(x,M)}{\rch_{\min}}\right )r_h^2$.
\end{proposition}
    
As a result, one may show that any ball has large mass with respect to a measure 
$D \in \distributions{n}{d}{\rch_{\min}}{f_{\min}}{f_{\max}}{L}$.
\begin{lemma}
\label{lem:intrinsic_ball_mass}
    Let $D \in \distributions{n}{d}{\rch_{\min}}{f_{\min}}{f_{\max}}{L}$ have support
    $M = \supp(D)$.
    \begin{itemize}
        \item For all $p \in M$ and  $h \le \rch_{\min}/4$,
            \[
                a_d f_{\min} h^d
                \le
                D\bigl(\B(p,h) \bigr)
                \le
                A_d f_{\max} h^d,
            \]
            where $a_d = 2^{-d} \omega_d$ and $A_d = 2^d \omega_d$.
        \item For all $x_0 \in \R^n$ and $h \le \rch_{\min}/8$,
            \[
                a'_d f_{\min} (h^2-\dd(x_0,M)^2)_+^{d/2}
                \le
                D\bigl(\B(x_0,h) \bigr)
                \le
                A'_d f_{\max} (h^2-\dd(x_0,M)^2)_+^{d/2},
            \]
            where $a'_d = (7/8)^{d/2}a_d$ and $A'_d = (5/4)^{d/2}A_d$.
    \end{itemize}
\end{lemma}
        
\begin{proof}[\proofof \Cref{lem:intrinsic_ball_mass}]
The first statement is a direct consequence of 
                \cite[Propositions~8.6 \&~8.7]{Aamari18}.
The second one follows by combining the previous point with 
            \Cref{prop:ball_projection}.
\end{proof}

\subsection{Euclidean Packing and Covering Estimates}
\label{subsec:euclidean-packing-and-covering}
For sake of completeness, we include in this section some standard packing and covering bounds that are used in our analysis. We recall the following definitions. 

A $r$-covering of $K \subseteq \R^n$ is a subset $\mathcal{X}= \set{ x_1,\ldots,x_k } \subseteq K$ such that for all $x \in K$, $\dd(x,\mathcal{X}) \leq r$. 
A $r$-packing of $K$ is a subset $\mathcal{Y} = \left\lbrace y_1,\ldots,y_k \right\rbrace \subseteq K$ such that for all $y,y' \in \mathcal{Y}$, $\B(y,r) \cap \B(y',r) = \emptyset$ (or equivalently $\norm{y'-y}>2r$).

\begin{definition}[Covering and Packing numbers]\label{def:packing_covering}
For $K \subseteq \R^n$ and $r>0$, the covering number $\CV_K(r)$ of $K$ is the minimum number of balls of radius $r$ that are necessary to cover $K$:
\begin{align*}
\CV_K(r)
&=
\min
\set{
{k > 0}~|
\text{ there exists a } r\text{-covering of cardinality } k
}
.
\end{align*}
The packing number $\PK_K(r)$ of $K$ is the maximum number of disjoint balls of radius $r$ that can be packed in $K$:
\begin{align*}
\PK_K(r)
&=
\max
\set{
{k > 0}~|
\text{ there exists a } r\text{-packing of cardinality } k
}
.
\end{align*}
\end{definition}

Packing and covering numbers are tightly related, as shown by the following well-known statement.

\begin{proposition}
    \label{prop:packing_covering_link}
    For all subset $K \subseteq \R^n$ and $r>0$,
    \begin{align*}
    \PK_K(2r)
    \leq 
    \CV_K(2r)
    \leq 
    \PK_K(r).
    \end{align*}
\end{proposition}

\begin{proof}[\proofof \Cref{prop:packing_covering_link}]
    For the left-hand side inequality, notice that if $K$ is covered by a family of balls of radius $2r$, each of these balls contains at most one point of a maximal $2r$-packing.
    Conversely, the right-hand side inequality follows from the fact that a maximal $r$-packing is always a $2r$-covering. Indeed, if it was not the case one could add a point $x_0 \in K$ that is $2r$-away from all of the $r$-packing elements, which would contradict the maximality of this packing.
\end{proof}

We then bound the packing and covering numbers of the submanifolds with reach bounded below. Note that these bounds depend only on the intrinsic dimension and volumes, but not on the ambient dimension.

\begin{proposition}
    \label{prop:packing_covering_manifold}
    For all $M \in \manifolds{n}{d}{\rch_{\min}}$ and $r \leq \rch_{\min} / 8$,
    \[
        \PK_M(r) 
        \geq 
        \frac{\Haus^d(M)}{\omega_d (4r)^d}
        ,
    \]
    and
    \[
        \CV_M(r) 
        \leq
        \frac{\Haus^d(M)}{\omega_d (r/4)^d}
        .
    \]    
\end{proposition}

\begin{proof}[\proofof \Cref{prop:packing_covering_manifold}]
First, we have $\PK_{M}(r) \geq \CV_{M}(2r)$ from \Cref{prop:packing_covering_link}. In addition, if $\set{p_i}_{1 \leq i \leq N} \subseteq M$ is a minimal $(2r)$-covering of $M$, then by considering the uniform distribution $D_M = \indicator{M} \Haus^d /\Haus^d(M)$ over $M$, using a union bound and applying \Cref{lem:intrinsic_ball_mass}, we get
    \begin{align*}
        1
        =
        D_M\left( \cup_{i = 1}^N \B(p_i,2r) \right)
        \leq
        \sum_{i = 1}^N
        D_M(\B(p_i,2r))
        \leq
        N 2^d \omega_d (2r)^d / \Haus^d(M)
        .
    \end{align*}
    As a result, 
    $
        \PK_{M}(r) 
        \geq 
        \CV_{M}(2r) 
        = 
        N
        \geq
        \frac{\Haus^d(M)}{\omega_d (4r)^d}
    .
    $
    
    For the second bound, use again \Cref{prop:packing_covering_link} to get $\CV_{M}(r) \leq \PK_{M}(r/2)$. Now, by definition, a maximal $(r/2)$-packing $\set{q_j}_{1 \leq j \leq N'}\subseteq M$ of $M$ provides us with a family of disjoint balls of radii $r/2$. Hence, from  \Cref{lem:intrinsic_ball_mass}, we get
    \begin{align*}
        1
        \geq
        D_M\left( \cup_{i = j}^{N'} \B(q_j,r/2) \right)
        =
        \sum_{j = 1}^{N'}
        D_M(\B(q_j,r/2))
        \geq
        N' 2^{-d} \omega_d (r/2)^d / \Haus^d(M)
        ,
    \end{align*}
    so that     
    $
        \CV_{M}(r) 
        \leq 
        \PK_{M}(r/2) 
        =
        N'
        \leq
        \frac{\Haus^d(M)}{\omega_d (r/4)^d}
    .
    $
\end{proof}

Bounds on the same discretization-related quantities computed on the Euclidean $n$-balls and $k$-spheres will also be useful.

\begin{proposition}
    \label{prop:packing_covering_ball_sphere}
    \begin{itemize}
        \item 
        For all $r > 0$,
        \[
            \PK_{\B(0,R)}(r) \geq \left(\frac{R}{2r}\right)^n
            \text{  and  }
            \CV_{\B(0,R)}(r)
            \leq
             \left(1+\frac{2R}{r}\right)^n
             .
        \]
        \item
        For all integer $1 \leq k < n$ and $r \leq 1/8$,
        \[
            \PK_{\Sphere^{k}(0,1)}(r) 
            \geq 
            2 \left(\frac{1}{4r}\right)^k
            .
        \]        
    \end{itemize}
\end{proposition}

\begin{proof}[\proofof \Cref{prop:packing_covering_ball_sphere}]
    \begin{itemize}[leftmargin=*]
        \item 
        From \Cref{prop:packing_covering_link}, we have $\PK_{\B(0,R)}(r) \geq \CV_{\B(0,R)}(2r)$. Furthermore, if $\cup_{i = 1}^N \B(x_i,2r) \supseteq \B(0,R)$ is a minimal $2r$-covering of $\B(0,R)$, then by a union bound, $
            \omega_n R^n
            =
            \Haus^n(\B(0,R))
            \leq
            N \omega_n (2r)^n
            ,
        $
        so that $\PK_{\B(0,R)}(r) \geq \CV_{\B(0,R)}(2r) = N \geq (R/(2r))^n$.
        
        For the second bound, we use again \Cref{prop:packing_covering_link} to get $\CV_{\B(0,R)}(r) \leq \PK_{\B(0,R)}(r/2)$, and we notice that any maximal $(r/2)$-packing of $\B(0,R)$ with cardinality $N'$ provides us with a family of disjoint balls of radii $r/2$, all contained in $\B(0,R)^{r/2} = \B(0,R+r/2)$. A union bound hence yields
        $
            \omega_n (R+r/2)^n 
            = 
            \Haus^n(\B(0,R+r/2))
            \geq
            N' \Haus^n(\B(0,r/2))
            =
            N' \omega_n (r/2)^n
        $,
        yielding $\CV_{\B(0,R)}(r) \leq \PK_{\B(0,R)}(r/2) = N' \leq (1+2R/r)^n$.
        \item
        Notice that $\Sphere^{k}(0,1) \subseteq \R^n$ is a compact $k$-dimensional submanifold without boundary, reach $\rch_{\Sphere^{k}(0,1)} = 1$, and volume $\Haus^k(\Sphere^k(0,1)) = \sigma_k$. 
        Applying \Cref{prop:packing_covering_manifold} together with elementary calculations hence yield
	    \begin{align*}
	        \PK_{\Sphere^k(0,1)}(r)
	        &\geq
	        \frac{\sigma_k}{\omega_k}
	        \left(\frac{1}{4r}\right)^k
	        \\
	        &=
	       \left(
	            \dfrac{
	                2\pi^{(k+1)/2}
	            }{
	                \Gamma\left( \frac{k+1}{2} \right)
	            }
	       \right)
	       \left(
	            \dfrac{
	                \pi^{k/2}
	            }{
	                \Gamma\left( \frac{k}{2} +1 \right)
	            }
            \right)^{-1}
	        \left(\frac{1}{4r}\right)^k
	        \\
	        &=
	        2\sqrt{\pi}
	        \frac{\Gamma\left( \frac{k}{2} +1 \right)}{\Gamma\left( \frac{k+1}{2} \right)}
	        \left(\frac{1}{4r}\right)^k
	        \\
	        &\geq
	        2 \left(\frac{1}{4r}\right)^k
	        .
	    \end{align*}
    \end{itemize}
    
\end{proof}

    \subsection{Global Volume Estimates}
    
    The following bounds on the volume and diameter of low-dimensional submanifolds of $\R^n$ with positive reach are at the core of \Cref{subsubsec:implicit-bounds-on-parameters}. They exhibit some implicit constraints on the parameters for the statistical models not to be degenerate.
    \begin{proposition}
        \label{prop:volume_bounds_under_reach_constraint}
        For all $M \in \manifolds{n}{d}{\rch_{\min}}$,
        \[\Haus^d(M) \geq \sigma_d \rch_{\min}^d,
        \]
        with equality if and only if $M$ is a $d$-dimensional sphere of radius $\rch_{\min}$.
        Furthermore, if $M \subseteq \B(0,R)$ then $\rch_{\min} \leq \sqrt{2} R$ and
        \[
            \Haus^d(M)
            \leq
	         \left(\frac{18R}{\rch_{\min}} \right)^n
    	     \omega_d \left( \frac{\rch_{\min}}{2} \right)^d
    	     .
        \]
    \end{proposition}
    
    \begin{proof}[\proofof \Cref{prop:volume_bounds_under_reach_constraint}]
    For the first bound, note that the operator norm of the second fundamental form of $M$ is everywhere bounded above by $1/\rch_{\min}$ \cite[Proposition~6.1]{Niyogi08}, so that \cite[(3)]{Almgren86} applies and yields the result.
    
    For the next statement, note that \cite[Theorem 3.26]{Hatcher02} ensures that $M$ is not homotopy equivalent to a point. As a result, \cite[Lemma A.3]{Aamari19} applies and yields
    \begin{align*}
        \rch_{\min}
        &\leq
        \rch_M
        \\
        &\leq
        \diam(M)/\sqrt{2}
        \\
        &\leq
        \diam(\B(0,R))/\sqrt{2}
        \\
        &=
        \sqrt{2} R.
    \end{align*}    
    For the last bound, consider a $(\rch_{\min}/8)$-covering $\set{z_i}_{1 \leq i \leq N}$ of $\B(0,R)$, which can be chosen so that $N \leq \left( 1 + \frac{2R}{\rch_{\min}/8} \right)^n \leq \left(\frac{18R}{\rch_{\min}} \right)^n$ from \Cref{prop:packing_covering_ball_sphere}.
    Applying \Cref{lem:intrinsic_ball_mass} with $h=\rch_{\min}/8$, we obtain
        \begin{align*}
            \Haus^d(M \cap \B(z_i,\rch_{\min}/8)) 
            &\leq
            (5/4)^{d/2} \times 2^d \omega_d ((\rch_{\min}/8)^2 - \dd(z_i,M)^2)_+^{d/2}
            \\
            &\leq
            \omega_d \left( \frac{\rch_{\min}}{2} \right)^d
            ,
        \end{align*}
    for all $i \in \set{1,\ldots,N}$. A union bound then yields
    \begin{align*}
        \Haus^d(M)
        &=
        \Haus^d\left( \cup_{i=1}^N M \cap \B(z_i,\rch_{\min}/8) \right)
        \\
        &\leq
        N \omega_d \left( \frac{\rch_{\min}}{2} \right)^d
        \\
        &\leq
        \left(\frac{18R}{\rch_{\min}} \right)^n
        \omega_d \left( \frac{\rch_{\min}}{2} \right)^d
         ,
    \end{align*}
    which concludes the proof.
    \end{proof}

%% file: parts/projection.tex
We now build the SQ projection routine $\hat{\pi}: \R^n \to \R^n$ (\Cref{thm:routine_projection}), which is used repeatedly in the SQ emulation of \MPA (\Cref{thm:SQ_upper_bound_point,thm:SQ_upper_bound_ball}).
Recall that given a point $x_0 \in \R^n$ nearby $M = \supp(D)$, we aim at estimating its metric projection $\pi_M(x_0)$ onto $M$ with statistical queries to $\STAT(\tau)$.
We follow the strategy of proof described in \Cref{subsec:SQ-projection-presentation}.

\subsection{Bias of the Local Conditional Mean for Projection}

In what follows, we will write
\begin{align}
    \label{eq:local_conditional_mean_def}
    m_D(x_0,h)
    &=
    \E_{x \sim D}\left[ x \left| \B(x_0,h) \right. \right]
    =
    \frac{
        \E_{x \sim D}\left[ x \indicator{\norm{x-x_0}\le h} \right]
    }
    {
    D(\B(x_0,h))
    }
\end{align}
for the local conditional mean of $D$ given $\B(x_0,h)$.
In order to study the bias of $m_D(x_0,h)$ with respect to $\pi_M(x_0)$, it will be convenient to express it (up to approximation) with intrinsic geodesic balls $\B_M(\cdot,\cdot)$ instead of the extrinsic Euclidean balls $\B(\cdot,\cdot)$ that appears in its definition (\Cref{eq:local_conditional_mean_def}).
This change of metric is stated in the following result.

\begin{lemma}
    \label{lem:local_conditional_mean_geodesic_bias}
    Let $D \in \distributions{n}{d}{\rch_{\min}}{f_{\min}}{f_{\max}}{L}$ have support $M = \supp(D)$, and $p \in M$.
    Recall that $\omega_d = \Haus^d\left(\B_d(0,1) \right)$ denotes the volume of the $d$-dimensional unit Euclidean ball.
    Then for all $r \le \rch_{\min}/4$,
    \begin{align*}
        \norm{
            \E_{x \sim D}
            \left[
                x
                \indicator{\B_M(p, r)}(x)
            \right]
            -
             D\left( \B_M(p, r) \right) p
        }
        &\le
        C^d \omega_d \left( \frac{f_{\max}}{\rch_{\min}} + L \right) r^{d+2}
        ,
    \end{align*}
    and for $r \le \bar{r} \le \rch_{\min}/4$,
    \begin{align*}
        D\left( \B_M(p,\bar{r}) \setminus \B_M(p, r) \right)
        &\le
        (C')^d \omega_d f_{\max} \bar{r}^{d-1} (\bar{r}-r),
    \end{align*}
    where $C,C'>0$ are absolute constants.
\end{lemma}        
\begin{proof}[\proofof \Cref{lem:local_conditional_mean_geodesic_bias}]
        First apply the area formula \cite[Section 3.2.5]{Federer69} to write the mean of any measurable function $G$ defined on $M$ as
        \begin{align*}
            \E_{x \sim D}
            \left[
                G(x)
                \indicator{\B_M(p, r)}(x)
            \right]            
            &=
            \int_0^r
            \int_{\Sphere^{d-1}}
            J(t, v)
            f\left( \exp_{p}^M(tv) \right)
            G\left(\exp_{p}^M(tv)\right)
            dv
            dt
            ,
        \end{align*}
        where $J(t, v)$ is the Jacobian of the volume form of $M$ expressed in polar coordinates around 
        $p$ for $0 \le t \le r \le \rch_{\min}/4$ and unit $v \in T_p M$. 
        That is, $J(t, v) = t^{d-1} \sqrt{\det \left( \transpose{A}_{t, v} A_{t, v} \right)}$ where $A_{t, v} = \dd_{tv} \exp_p^M$.
        But from \cite[Proposition A.1 (iv)]{Aamari19}, for all $w \in T_p M$, we have
        \[ 
        \left( 1 - \frac{t^2}{6\rch_{\min}^2} \right) \norm{w} 
        \le 
        \norm{A_{t, v} w}
        \le 
        \left(1+ \frac{t^2}{\rch_{\min}^2} \right)\norm{w}
        .
        \]
        As a consequence, 
        \[ 
        \left( 1 - \frac{t^2}{6\rch_{\min}^2} \right)^d \le \sqrt{\det \left( \transpose{A}_{t, v} A_{t, v} \right)} \le \left(1+ \frac{t^2}{\rch_{\min}^2} \right)^d
        \]
        and in particular,
        \begin{align*}
            R_J(t, v) :=
            \left| J(t, v) - t^{d-1} \right| 
            \le 
            C^d t^{d-1}\left(\frac{t}{\rch_{\min}} \right)^2,
        \end{align*}
        where $C > 0$ is an absolute constant.
        Also, by assumption on the model, $f$ is $L$-Lipschitz, so
        \begin{align*}
            |R_f (t, v)|
            := 
            \left|f\left( \exp_{p}^M(tv) \right) - f(p) \right|
            &=
            \left|f\left( \exp_{p}^M(tv) \right) - f(\exp_{p}^M(0)) \right|
            \\
            &\le
            L \norm{ \exp_{p}^M(tv) - \exp_{p}^M(0)}
            \\
            &\le
            L d_M(\exp_{p}^M(0), \exp_{p}^M(tv))
            \\
            &=
            Lt.
        \end{align*}
        Finally, from \cite[Lemma 1]{Aamari19b}, we have
        \begin{align*}
            \norm{R_{\exp}(t, v)}
            :=
            \norm{\exp_{p}^M(tv)- (p + tv)} 
            \le 
            5 t^2/(8 \rch_{\min})
            .
        \end{align*}
        Putting everything together, we can now prove the first bound by writing
        \begin{align*}
            &
            \norm{
                \E_{x \sim D}
                \left[
                    x
                    \indicator{\B_M(p, r)}(x)
                \right]
                -
                D\left( \B_M(p, r) \right) p
            }
            \\
            &=
            \norm{
            \int_0^r
            \int_{\Sphere^{d-1}}
            J(t, v)
            f\left( \exp_{p}^M(tv) \right)
            \left(
            \exp_{p}^M(tv)
            - 
            p
            \right)
            dv
            dt
            }
            \\
            &=
            \norm{
            \int_0^r
            \int_{\Sphere^{d-1}}
            \left(t^{d-1} + R_{J}(t, v)\right)
            \left(f(p) + R_{f}(t, v)\right)
            \left(
            tv + R_{\exp}(t, v)
            \right)
            dv
            dt
            }
            \\
            &\le
            \tilde{C}^d \omega_d \left( \frac{f_{\max}}{\rch_{\min}} + L \right) r^{d+2},
        \end{align*}
        where the last inequality used the fact that $ \int_0^r \int_{\Sphere^{d-1}}t^d f(p)v dv dt = 0$.
        Similarly, to derive the second bound, we write
        \begin{align*}
            D\left( \B_M(p,\bar{r}) \setminus \B_M(p, r) \right)
            &=
            \int_r^{\bar{r}}
            \int_{\Sphere^{d-1}}
            J(t, v)
            f\left( \exp_{p}^M(tv) \right)
            dv
            dt
            \\
            &\le
            \sigma_{d-1} f_{\max}
            \int_r^{\bar{r}}
            t^{d-1}
            \left(
            1
            +
            C^d \left(t/\rch_{\min} \right)^2
            \right)
            dt
            \\
            &\le
            (C')^d \sigma_{d-1} f_{\max}
            \int_r^{\bar{r}}
            t^{d-1}
            dt
            \\
            &\le
            (C'')^d \omega_d f_{\max} \bar{r}^{d-1} (\bar{r}-r)
            ,
        \end{align*}
        which concludes the proof.        
    \end{proof}

We are now in position to bound the bias of $m_D(x_0,h)$.
    
    \begin{lemma}
        \label{lem:local_conditional_mean_bias}
        Let $D \in \distributions{n}{d}{\rch_{\min}}{f_{\min}}{f_{\max}}{L}$ have support $M = \supp(D)$, and $x_0 \in \R^n$ be such that $\dd(x_0,M) < h \le \rch_{\min}/8$.
        Then,
        \begin{align*}
            \norm{\pi_M(x_0) - m_D(x_0,h)}
            \le
            C^d
            \left( \frac{f_{\max} + L \rch_{\min} }{f_{\min}} \right) \frac{h r_h}{\rch_{\min}}
            ,
        \end{align*}
        where $r_h = (h^2- \dd(x_0,M)^2)^{1/2}$ and $C>0$ is an absolute constant.
    \end{lemma}
    
    \begin{proof}[\proofof \Cref{lem:local_conditional_mean_bias}]
        For short, let us write $p_0 = \pi_M(x_0)$. All the expected values $\E$ are taken with respect to $x \sim D$.
        Before any calculation, we combine \Cref{prop:ball_projection} and \Cref{lem:geodesic_comparison} to assert that
        \begin{align}
            \label{eq:mean_bias_inclusions}
             \B_M\left(p_0, r_h^- \right )
             \subseteq 
             \B(x_0,h) \cap M
             \subseteq 
             \B_M\left(p_0, R_h^+ \right),
        \end{align}
        where we wrote
        $(r_h^-)^2 = \left (1-\dd(x_0,M)/\rch_{\min}\right )r_h^2$ 
        and
        $R_h^+ = r_h^+\left(1 + (r_h^+/\rch_{\min})^2 \right)$,
        with 
        $(r_h^+)^2 = \left (1+{2 \dd(x_0,M)}/{\rch_{\min}}\right )r_h^2$.
        We note by now from the definition $0 < r_h^- \le R_h^+ \le \rch_{\min}/4$ since $\dd(x_0,M) < h \le \rch_{\min}/8$, and that
        \begin{align}
            \label{eq:rh}
            R_h^+ - r_h^-
            &\le
            \frac{C'r_h}{\rch_{\min}}
            \left(
            \dd(x_0,M) + r_h^2/\rch_{\min}
            \right)
            \le
            \frac{2C'hr_h}{\rch_{\min}},
        \end{align}
        for some absolute constant $C' > 0$.
        
        We can now proceed and derive the asserted bound. From triangle inequality,
        \begin{align*}
            &\norm{m_D(x_0,h)-\pi_M(x_0)}
            \\
            &=
            \frac{
                \norm{
                \E_{x \sim D}
                \left[ 
                (x-p_0) \indicator{\B(x_0,h)}(x)
                \right]
                }
            }{
                D(\B(x_0,h))
            }
            \\
            &\le
            \norm{
                \frac{
                    \E
                    \left[ 
                    (x-p_0) \indicator{\B(x_0,h)}(x)
                    \right]
                }{
                    D(\B(x_0,h))
                }
            -
                \frac{
                    \E
                    \left[ 
                    (x-p_0) \indicator{\B_M(p_0,R_h^+)}(x)
                    \right]
                }{
                    D(\B_M(p_0,R_h^+))
                }
            }
            \\
            &~~~+
            \frac{
                \norm{
                \E
                \left[ 
                (x-p_0) \indicator{\B_M(p_0,R_h^+)}(x)
                \right]
                }
            }{
                D(\B_M(p_0,R_h^+))
            }
            .
        \end{align*}
        Combining \Cref{eq:mean_bias_inclusions}, \Cref{lem:local_conditional_mean_geodesic_bias}, \Cref{prop:ball_projection} and \Cref{lem:intrinsic_ball_mass}, the first term of the right hand side can be further upper bounded by
        \begin{align*}
            &\norm{
                \frac{
                    \E
                    \left[ 
                    (x-p_0) \indicator{\B(x_0,h)}(x)
                    \right]
                }{
                    D(\B(x_0,h))
                }
            -
                \frac{
                    \E
                    \left[ 
                    (x-p_0) \indicator{\B_M(p_0,R_h^+)}(x)
                    \right]
                }{
                    D(\B_M(p_0,R_h^+))
                }
            }
            \\
            &\le
            \frac{
                \norm{
                    \E
                    \left[ 
                    (x-p_0) \indicator{\B(x_0,h)}(x)
                    \right]
                }
            }
            {
                D(\B(x_0,h))D(\B_M(p_0,R_h^+))
            }
            \left|
                D(\B_M(p_0,R_h^+)) - D(\B(x_0,h))
            \right|
            \\
            &~~~+
            \frac{
                  \norm{
                  \E
                    \left[ 
                    (x-p_0)
                    \left(
                        \indicator{\B_M(p_0,R_h^+)}(x)
                        -
                        \indicator{\B(x_0,h)}(x)
                        \right)
                    \right]
                    }
            }
            {
                D(\B_M(p_0,R_h^+))
            }
            \\
            &\le
            \frac{
                2R_h^+ D\left(\B_M(p_0,R_h^+) \setminus \B_M(p_0, r_h^-)\right)
            }
            {
                D(\B_M(p_0,R_h^+))
            }
            \\
            &\le
            \frac{(C'')^d \omega_d f_{\max} (R_h^+)^d(R_h^+ - r_h^-)}{c^d \omega_d f_{\min} (R_h^+)^d}
            \\
            &\le
            \tilde{C}^d \frac{f_{\max}}{f_{\min}}
            \frac{hr_h}{\rch_{\min}},
        \end{align*}
        where the last bound uses \eqref{eq:rh}.
        For the second term, we use \Cref{lem:local_conditional_mean_geodesic_bias} and \Cref{lem:intrinsic_ball_mass} to derive
        \begin{align*}
            \frac{
                \norm{
                \E
                \left[ 
                (x-p_0) \indicator{\B_M(p_0,R_h^+)}(x)
                \right]
                }
            }{
                D(\B_M(p_0,R_h^+))
            }
            &\le
            \frac{
                (\tilde{C}')^d \omega_d \left( \frac{f_{\max}}{\rch_{\min}} + L \right) (R_h^+)^{d+2}
            }
            {
                c^d \omega_d f_{\min} (R_h^+)^d
            }
            \\
            &\le
            (\tilde{C}'')^d
            \left( \frac{f_{\max} + L \rch_{\min} }{f_{\min}} \right) \frac{r_h^{2}}{\rch_{\min}}.
            \end{align*}
            Since $r_h \leq h$, this concludes the proof by setting $C = \tilde{C} + \tilde{C}''$.
    \end{proof}
    
\subsection{Metric Projection with Statistical Queries}
 \label{subsec:projection-estimation-with-SQs}
    We finally prove the main announced statement of \Cref{sec:SQ-projection-appendix}.
    \begin{proof}[\proofof \Cref{thm:routine_projection}]
        First note that under the assumptions of the theorem, $\dd(x_0, M) \le \Lambda \le \rch_{\min} / 8$. We hence let $h >0$ be a bandwidth to be specified later, but taken such that $\dd(x_0,M) < \sqrt{2} \Lambda \le h \le \rch_{\min}/8$.
        
        Consider the map $F(x) = \frac{(x-x_0)}{h}\indicator{\norm{x-x_0}\le h}$ for $x \in \R^n$.
        As $\norm{F(x)} \leq 1$ for all $x\in \R^n$, \Cref{lem:mean_vector_estimation} asserts that there exists a deterministic statistical query algorithm making $2n$ queries to $\STAT(\tau)$ and that outputs a vector $\hat{W} = \hat{V}/h \in \R^n$ such that $\norm{\E_{x \sim D}\left[ F(x)\right]- \hat{V}/h} \le C \tau$.
        Furthermore, with the single query $r = \indicator{\B(x_0,h)}$ to $\STAT(\tau)$, we obtain $\hat{a} \in \R$ such that $\left| D(\B(x_0,h) - \hat{a} \right| \le \tau$.
        Let us set $\hat{\pi}(x_0) := x_0 + \hat{V}/\hat{a}$ and prove that it satisfies the claimed bound. 
        For this, use $|V/a - \hat{V}/\hat{a}| \leq |a-\hat{a}|V/(a\hat{a})+|V-\hat{V}|/\hat{a}$ to  write
        \begin{align*}
            &\norm{m_D(x_0,h) - \hat{\pi}(x_0)}
            \\
            &=
            \norm{
            \frac{
                \E_{x \sim D}\left[ (x-x_0) \indicator{\norm{x-x_0}\le h} \right]
            }
            {
                D(\B(x_0,h))
            }
            -
            \frac{\hat{V}}{\hat{a}}
            }
            \\
            &\le
            \frac{
                \left| D(\B(x_0,h)) - \hat{a} \right|
                \norm{
                    \E_{x \sim D}\left[ (x-x_0) \indicator{\norm{x-x_0}\le h} \right]
                }
            }{
            D(\B(x_0,h)) \hat{a}
            }
            +
            \frac{
                \norm{
                    \E_{x \sim D}\left[ (x-x_0) \indicator{\norm{x-x_0}\le h} \right]
                    -
                    \hat{V}
                }
            }{
                \hat{a}
            }
            \\
            &\le
            \frac{
                \left| D(\B(x_0,h)) - \hat{a} \right| h
                +
                \norm{
                    \E_{x \sim D}\left[ (x-x_0) \indicator{\norm{x-x_0}\le h} \right]
                    -
                    \hat{V}
                }
            }{
            D(\B(x_0,h)) - \left| D(\B(x_0,h) - \hat{a} \right|
            }
            \\
            &\le
            \frac{
                (C+1) \tau h
            }{
            D(\B(x_0,h)) - \tau
            }
            \\
            &\le
            \frac{
                (C+1)\tau h
            }{
            \tilde{c}^d \omega_d f_{\min} (h/\sqrt{2})^d - \tau
            }
            ,
        \end{align*}
         where the last inequality comes from \Cref{lem:intrinsic_ball_mass}, and $r_h = (h^2-\dd(x_0,M)^2)^{d/2} \ge h/\sqrt{2}$ since $h \ge \sqrt{2} \Lambda$.
         If in addition, one assumes that $\tilde{c}^d \omega_d f_{\min} (h/\sqrt{2})^d \ge 2\tau$, we obtain the lower bound $\tilde{c}^d \omega_d f_{\min} (h/\sqrt{2})^d - \tau \geq \tilde{c}^d \omega_d f_{\min} (h/\sqrt{2})^d /2$, so that the previous bound further simplifies to
         \begin{align*}
            \norm{m_D(x_0,h) - \hat{\pi}(x_0)}
            &\le
            \frac{
                (C')^d
            }{
            \omega_d f_{\min}
            }         
            \tau h^{1-d}
            .
         \end{align*}
        On the other hand, \Cref{lem:local_conditional_mean_bias} yields that the bias term is not bigger than
        \begin{align*}
            \norm{\pi_M(x_0) - m_D(x_0,h)}
            \le
            \tilde{C}^d
            \left( \frac{f_{\max} + L \rch_{\min} }{f_{\min}} \right) \frac{h r_h}{\rch_{\min}}
            ,
        \end{align*}
        with $r_h \le h$.
        As a result,
        \begin{align*}
            \norm{\pi_M(x_0) - \hat{\pi}(x_0)}
            &\le
            \norm{\pi_M(x_0) - m_D(x_0,h)}
            +
            \norm{m_D(x_0,h)- \hat{\pi}(x_0)}
            \\
            &\le
            \frac{(C' \vee \tilde{C})^d}{f_{\min}}
            \left(
            \left( f_{\max} + L \rch_{\min} \right) \frac{h}{\rch_{\min}}
            +
            \frac{\tau}{\omega_d h^d}
            \right)
            h
            .
        \end{align*}
        Taking bandwidth 
        \begin{align*}
        h 
        &= 
        \max\set{
        2\Lambda
        ,
        \left( \frac{\rch_{\min}}{f_{\max} + L \rch_{\min}} \right)^\frac{1}{d+1} \left(\frac{\tau}{\omega_d} \right)^{\frac{1}{d+1}}
        }
        \\
        &=
        \max\set{
        2\Lambda
        ,
        \rch_{\min}
        \left( \frac{f_{\min}}{f_{\max} + L \rch_{\min}} \right)^\frac{1}{d+1} 
        \left(\frac{\tau}{\omega_d f_{\min} \rch_{\min}^d} \right)^{\frac{1}{d+1}}
        }
        ,
        \end{align*}
        we have by assumption on the parameters of the model that $ \rch_{\min}/8 \ge h \ge 2 \Lambda \geq \sqrt{2} \Lambda$, and that $\tilde{c}^d \omega_d f_{\min} (h/\sqrt{2})^d \ge 2\tau$ as soon as $c>0$ is small enough.
        Finally, plugging the value of $h$ in the above bound and recalling that
        $
            \Gamma 
            = 
            \frac{f_{\min}}{f_{\max} + L \rch_{\min}} 
        $
        yields
        \begin{align*}
            \norm{\pi_M(x_0) - \hat{\pi}(x_0)}
            &\le
            \frac{\tilde{C}^d}{\Gamma}
            \max\set{
                \frac{\Lambda^2}{\rch_{\min}}
            ,
                \Gamma^\frac{2}{d+1}
                \rch_{\min}
                \left(\frac{\tau}{\omega_d f_{\min} \rch_{\min}^d} \right)^{\frac{2}{d+1}}
            }
            ,
        \end{align*}
        which concludes the proof.
    \end{proof}

%% file: parts/tangent.tex
We now build the SQ tangent space routine 
$\hat{T} : \R^n \to \Gr{n}{d}$ (\Cref{thm:routine_tangent}), 
which is used repeatedly in the SQ emulation of \MPA (\Cref{thm:SQ_upper_bound_point,thm:SQ_upper_bound_ball}).
Recall that given a point $x_0 \in \R^n$ nearby $M = \supp(D)$, we aim at estimating the tangent space $T_{\pi_M(x_0)} M$ with statistical queries to $\STAT(\tau)$. 
We follow the strategy of proof described in \Cref{subsec:SQ-tangent-presentation}.

To fix notation from now on, 
we let $\inner{A}{B} = \tr(A^\ast B)$ stand for the Euclidean 
inner product between $A,B \in \R^{k \times k}$. We also write 
$\normF{\Sigma} = \sqrt{\inner{\Sigma}{\Sigma}}$ for the 
Frobenius norm, 
$\normop{\Sigma} = \max_{\norm{v}\le 1} \norm{\Sigma v}$ for the
operator norm, and 
$\normNuc{\Sigma} = \max_{\normop{X}\le 1} \inner{\Sigma}{X}$ for
the nuclear norm.
In what follows, for a symmetric matrix $A \in \mathbb{R}^{n \times n}$, we let $\mu_i(A)$ denote its $i$-th largest singular value.
            
            \subsection{Bias of Local Principal Component Analysis}
        
            In what follows, we will write
            \begin{align}
                \label{eq:covariance_local_definition}
                \Sigma_D(x_0,h)
                =
                \E_{x \sim D} \left[ \frac{(x-x_0)\transpose{(x-x_0)}}{h^2} \mathbbm{1}_{\norm{x-x_0} \le h} \right]
            \end{align}
            for the re-scaled local covariance-like matrix of $D$ at $x_0 \in \R^n$ with bandwidth $h>0$.
            Notice that for simplicity, this local covariance-like matrix is computed with centering at the current point $x_0$, and not at the local conditional mean $\E_{x \sim D}\left[x | \norm{x-x_0} \le h \right]$.
            This choice simplifies our analysis and will not impact the subsequent estimation rates.
            Let us first decompose this matrix and exhibit its link with the target tangent space $T_{\pi_M(x_0)} M \in \Gr{n}{d}$.
        
            \begin{lemma}
                \label{lem:covariance_decomposed}
                Let $D \in \distributions{n}{d}{\rch_{\min}}{f_{\min}}{f_{\max}}{L}$ have support $M = \supp(D)$, $x_0 \in \R^n$ and $h > 0$.
                If $\dd(x_0,M) \le \eta \le h/\sqrt{2}$ and $h \le \rch_{\min}/(8\sqrt{d})$, then there exists a symmetric matrix $\Sigma_0 \in \R^{n \times n}$ with $\operatorname{Im}(\Sigma_0) = T_{\pi_M(x_0)} M$ such that
                \[
                    \Sigma_D(x_0,h) = \Sigma_0 + R,
                \]
                with $\mu_d(\Sigma_0) \ge \omega_d f_{\min}(ch)^d$ and $\normNuc{R} \le \omega_d f_{\max} (C h)^d \left( \frac{\eta}{h} + \frac{h}{\rch_{\min}}\right)$, where $c,C>0$ are absolute constants.
            \end{lemma}
        
            \begin{proof}[\proofof \Cref{lem:covariance_decomposed}]
                This proof roughly follows the ideas of \cite[Section E.1]{Aamari18}, with a different center point in the covariance matrix ($x_0$ itself instead of the local mean around $x_0$) and finer (nuclear norm) estimates on residual terms.
                For brevity, we let $p_0 = \pi_M(x_0)$. 
                We first note that the integrand defining $h^2\Sigma_D(x_0,h)$ decomposes as
                \begin{align}
                    \label{eq:covariance_decomposed_1}
                    (x-x_0)\transpose{(x-x_0)}
                    &=
                    (x-p_0)\transpose{(x-p_0)}
                    +
                    (x_0-p_0)\transpose{(x_0-p_0)}
                    \\
                    &~~~~+
                    (x-p_0)\transpose{(x_0-p_0)}
                    +
                    (x_0-p_0)\transpose{(x-p_0)},
                    \nonumber
                \end{align}
                for all $x \in \B(x_0,h) \cap M$.
                After integrating them with respect to $x \sim D$, we bound the last two terms, by writing
                \begin{align*}
                     \normNuc{\E_{x \sim D}\left[(x-p_0)\transpose{(x_0-p_0)}\mathbbm{1}_{\norm{x-x_0} \le h}\right]}
                     &=
                     \normNuc{\E_{x \sim D}\left[(x_0-p_0)\transpose{(x-p_0)}\mathbbm{1}_{\norm{x-x_0} \le h}\right]}
                     \\
                     &\le
                     \E_{x \sim D}\left[
                        \normNuc{
                        (x_0-p_0)\transpose{(x-p_0)}}
                        \mathbbm{1}_{\norm{x-x_0} \le h}
                    \right]
                    \\
                    &=
                     \E_{x \sim D}
                        \left[
                            \norm{x_0-p_0}\norm{x-p_0}\mathbbm{1}_{\norm{x-x_0} \le h}
                        \right]       
                    \\
                    &\le
                    \eta h D(\B(x_0,h))
                    \\
                    &\le
                    C^d \omega_d f_{\max} \eta h^{d+1},
                \end{align*}
                where the last inequality uses \Cref{lem:intrinsic_ball_mass}.
                Similarly, for the second term of \Cref{eq:covariance_decomposed_1}, we have
                \begin{align*}
                     \normNuc{\E_{x \sim D}\left[(x_0-p_0)\transpose{(x_0-p_0)}\mathbbm{1}_{\norm{x-x_0} \le h}\right]}
                     &=
                     \norm{x_0-p_0}^2 D(\B(x_0,h))
                     \\
                     &\le
                     C^d \omega_d f_{\max} \eta^2 h^{d}.
                \end{align*}
                Given $v\in \R^n$, write $v_\sslash = \pi_{T_{p_0} M}(v)$ and $v_\perp = v - v_\sslash = \pi_{T_{p_0} M^\perp}(v)$. We now focus on the first term of \Cref{eq:covariance_decomposed_1}, which we further decompose as
                \begin{align}
                    \label{eq:covariance_decomposed_2}
                    (x-p_0)\transpose{(x-p_0)}
                    &=
                    (x-p_0)_\sslash\transpose{(x-p_0)_\sslash}
                    +
                    (x-p_0)_\perp\transpose{(x-p_0)_\perp}
                    \\
                    &~~~~+
                    (x-p_0)_\perp\transpose{(x-p_0)_\sslash}
                    +
                    (x-p_0)_\sslash\transpose{(x-p_0)_\perp}
                    ,
                    \nonumber
                \end{align}
                for all $x \in \B(x_0,h) \cap M$. Note that for those points $x \in \B(x_0,h) \cap M$, we have $\norm{(x-p_0)_\sslash} \le \norm{x-p_0} \le 2 h$, and from \cite[Theorem 4.18]{Federer59}, $\norm{(x-p_0)_\perp} \le \norm{x-p_0}^2/(2\rch_{\min}) \le 4 h^2/(2\rch_{\min})$.
                Hence, for the last two terms of \Cref{eq:covariance_decomposed_2},
                \begin{align*}
                     \normNuc{\E_{x \sim D}\bigl[(x-p_0)_\perp\transpose{(x-p_0)_\sslash}\mathbbm{1}_{\norm{x-x_0} \le h}\bigr]}
                     &=
                     \normNuc{\E_{x \sim D}\left[(x-p_0)_\sslash\transpose{(x-p_0)_\perp}\mathbbm{1}_{\norm{x-x_0} \le h}\right]}
                     \\
                     &\le
                     \E_{x \sim D} 
                        \bigl[
                        \norm{(x-p_0)_\sslash}\norm{(x-p_0)_\perp}
                        \mathbbm{1}_{\norm{x-x_0} \le h}
                        \bigr]
                     \\
                     &\le
                     C^d \omega_d f_{\max} h^{d+3}/\rch_{\min}
                     ,
                \end{align*}
                where we used \Cref{lem:intrinsic_ball_mass} again. 
                Dealing now with the second term of \Cref{eq:covariance_decomposed_2},
                \begin{align*}
                     \normNuc{\E_{x \sim D}
                     \bigl[(x-p_0)_\perp\transpose{(x-p_0)_\perp}\mathbbm{1}_{\norm{x-x_0} \le h}\bigr]}
                     &\le
                     \E_{x \sim D} 
                        \bigl[
                        \norm{(x-p_0)_\perp}\norm{(x-p_0)_\perp}
                        \mathbbm{1}_{\norm{x-x_0} \le h}
                        \bigr]
                    \\
                    &\le
                    C^d \omega_d f_{\max} h^{d+4}/(4 \rch_{\min}^2).
                \end{align*}
                Finally, let us write
                \begin{align*}
                    \Sigma_0
                    =
                    \E_{x \sim D} \left[ \frac{(x-p_0)_\sslash\transpose{(x-p_0)_\sslash}}{h^2} \mathbbm{1}_{\norm{x-x_0} \le h} \right]
                    .
                \end{align*}
                The matrix $\Sigma_0$ is symmetric and clearly has image $\operatorname{Im}(\Sigma_0) \subseteq T_{p_0} M$. Furthermore, since $\dd(x_0,M) \le \eta \le h/\sqrt{2}$ and $h \le \rch_{\min}/8$, \Cref{prop:ball_projection} and \Cref{lem:geodesic_comparison} yield that $M \cap \B(x_0,h) \supseteq M \cap \B\bigl(p_0, \sqrt{7}h/4\bigr) \supseteq \B_M(p_0,h/2)$. Hence, for all $u \in T_{p_0} M$,
                \begin{align*}
                    h^2\inner{\Sigma_0 u}{u}
                    &=
                    \E_{x \sim D} \left[ \inner{(x-p_0)_\sslash}{u}^2 \mathbbm{1}_{\norm{x-x_0} \le h} \right]
                    \\
                    &=
                    \E_{x \sim D} \left[ \inner{x-p_0}{u}^2 \mathbbm{1}_{\norm{x-x_0} \le h} \right]
                    \\
                    &\ge
                    f_{\min}
                    \int_{\B_M(p_0,h/2)}
                    \inner{x-p_0}{u}^2
                    \dd \Haus^d(x)
                    \\
                    &=
                    f_{\min}
                    \int_{\B_d(0,h/2)}
                    \inner{\exp_{p_0}^M(v)-p_0}{u}^2
                    \left| \det \left( \dd_v \exp_{p_0}^M \right) \right|
                    \dd v
                    ,
                \end{align*}
                where $\Haus^d$ is the $d$-dimensional Hausdorff measure on $\R^n$, and $\exp_{p_0}^M : T_{p_0} M \to M$ is the exponential map of $M$ at $p_0$.
                But \cite[Proposition 8.7]{Aamari18} states that there exists $c > 0$ such that for all $v \in \B_d(0, \rch_{\min}/4)$, $\left| \det \left( \dd_v \exp_{p_0}^M \right) \right| \ge c^d$, and \cite[Lemma 1]{Aamari19b} yields the bound $\norm{\exp_{p_0}^M(v)- (p_0 + v)} \le 5\norm{v}^2/(8 \rch_{\min})$. As a result, using the fact that $(a-b)^2 \ge a^2/2 - 3b^2$ for all $a,b \in \R$, we have
                \begin{align*}
                    h^2\inner{\Sigma_0 u}{u}
                    &\ge
                    c^d f_{\min}
                    \int_{\B_d(0,h/2)}
                    \left(
                     \inner{v}{u}
                     -
                      \inner{\exp_{p_0}^M(v)-(p_0+v)}{u}
                    \right)^2
                    \dd v
                    \\
                    &\ge
                    c^d f_{\min} \int_{\B_d(0,h/2)}
            \inner{v}{u}^2/2
             -
             3 \inner{\exp_{p_0}^M(v)-(p_0+v)}{u}^2
            \dd v
            \\
            &\ge
            c^d f_{\min} \int_{\B_d(0,h/2)}
            \inner{v}{u}^2/2
            -
            3 \norm{u}^2 \left( 5\norm{v}^2/(8 \rch_{\min}) \right)^2
            \dd v
            \\
            &=
            c^d f_{\min} \sigma_{d-1} \left( \frac{1}{2d(d+2)} - \frac{3(5/8)^2}{d+4} \left( \frac{h}{2\rch_{\min}} \right)^2 \right) \left(\frac{h}{2}\right)^{d+2}  \norm{u}^2
            \\
            &\ge
            (c')^d \omega_d f_{\min} h^{d+2} \norm{u}^2,
        \end{align*}
        as soon as $h \leq \rch_{\min}/\sqrt{d}$.
        In particular, the last bound shows that the image of $\Sigma_0$ is
        exactly $T_{p_0} M$, and that $\mu_d(\Sigma_0) \ge \omega_d f_{\min}(c'h)^d$.
        Summing up the above, we have shown that 
        \[
            \Sigma_D(x_0,h) = \Sigma_0 + R,
        \] 
        where $\Sigma_0$ is symmetric, $\operatorname{Im}(\Sigma_0) = T_{\pi_M(x_0)} M$, 
        $\mu_d(\Sigma_0) \ge \omega_d f_{\min}(c'h)^d$, and 
        \begin{align*}
            \normNuc{R} 
            &\le \omega_d f_{\max} (C'h)^d
            \left( \frac{\eta}{h} + \frac{\eta^2}{h^2} + \frac{h}{\rch_{\min}} + \frac{h^2}{\rch_{\min}^2}\right)
            \\
            &\le \omega_d f_{\max} (C'' h)^d \left( \frac{\eta}{h} + \frac{h}{\rch_{\min}} \right),
        \end{align*}
        which is the announced result.
    \end{proof}

            \subsection{Matrix Decomposition and Principal Angles}
                The following lemma ensures that the principal components of a matrix $A$ are stable to perturbations, provided that $A$ has a large-enough spectral gap.
                For a symmetric matrix $A \in \mathbb{R}^{n \times n}$, recall that $\mu_i(A)$ denotes its $i$-th largest singular value.
            \begin{lemma}[Davis-Kahan]
                \label{lem:angledeviation}
                Let $\hat{A}, A \in \R^{n\times n}$ be symmetric matrices such that $\rank(A) = d$.
                If $\hat{T} \in \Gr{n}{d}$ denotes the linear space spanned by the first $d$ eigenvectors of $\hat{A}$, and $T = \operatorname{Im}(A) \in \Gr{n}{d}$, then
                \[
                    \angle\bigl( T, \hat{T} \bigr)
                    :=
                    \normop{\pi_{\hat{T}} - \pi_{T}}
                    \le 
                    \frac{2 \normF{\hat{A}- A}}{\mu_d(A)}.
                \]
            \end{lemma}
            
            \begin{proof}[\proofof \Cref{lem:angledeviation}]
                    It is a direct application of \cite[Theorem 2]{Yu15} with $r = 1$ and $s = d$.
            \end{proof}

\subsection{Low-rank Matrix Recovery}

Proceeding further in the strategy described in \Cref{subsec:SQ-tangent-presentation}, we now explain how to estimate the local covariance matrix $\Sigma_D(x_0,h) \in \R^{n \times n}$ (\Cref{eq:covariance_local_definition}) in $\STAT(\tau)$.

Because $\Sigma_D(x_0,h) \in \R^{n \times n} = \R^{n^2}$ can be seen as a mean vector with respect to the unknown distribution $D$, $2n^2$ queries to $\STAT(\tau)$ would yield error $O(\tau)$ from \Cref{lem:mean_vector_estimation}.
However, this would not use the low-rank structure of $\Sigma_D(x_0,h)$, i.e. some redundancy of its entries. To mitigate the query complexity of this estimation problem, we will use compressed sensing techniques~\cite{Fazel08}.
Mimicking the vector case (\Cref{lem:mean_vector_estimation}), we put our problem in the broader context of the estimation of $\Sigma = \E_{x \sim D}[F(x)] \in \R^{k \times k}$ in $\STAT(\tau)$, where $F : \R^n \to \R^{k \times k}$ and $\Sigma$ are approximately low rank (see \Cref{lem:mean_matrix_estimation}).

\subsubsection{Restricted Isometry Property and Low-Rank Matrix Recovery}
    \label{subsubsec:RIP-implies-recovery}
Let us first present some fundamental results coming of matrix recovery.
Following \cite[Section II]{Fazel08}, assume that we observe 
$y \in \R^q$ such that
\begin{align}
    y = \sampling(\Sigma) + z,
\end{align}
where $\Sigma \in \R^{k \times k}$ is the matrix of interest, 
$\sampling : \R^{k\times k} \to \R^q$ is a linear map seen as a 
sampling operator, and $z \in \R^q$ encodes noise and has small 
Euclidean norm $\norm{z} \le \xi$.

In general, when $q < k^2$, $\sampling$ has non-empty kernel, and hence one has no hope to recover $\Sigma$ only from $y$, even with no noise. 
However, if $\Sigma$ is (close to being) low-rank and that $\sampling$ does not shrink low-rank matrices too much, $\sampling(\Sigma)$ may not actually censor information on $\Sigma$, while compressing the dimension from $k^2$ to $q$. 
A way to formalize this idea states as follows.
\begin{definition}[Restricted Isometry Property]
\label{def:RIP}
    Let $\sampling : \R^{k\times k} \to \R^q$ be a linear map, and 
    $d \le k$. We say that $\sampling$ satisfies the $d$-restricted
    isometry property with constant $\delta > 0$ if for all matrix 
    $X \in \R^{k \times k}$ of rank at most~$d$,
    \begin{align*}
        (1 - \delta) \normF{X}
        \le
        \norm{\sampling(X)}
        \le
        (1 + \delta) \normF{X}.
    \end{align*}
    We let $\delta_d(\sampling)$ denote the smallest such $\delta$. 
\end{definition}
To recover $\Sigma$ only from the knowledge of $y$, consider the 
convex optimization problem (see~\cite{Fazel08}) over $X \in \R^{k \times k}$:
            \begin{align}
                \label{eq:matrix_minimization_problem}
                \begin{array}{ll}
                    \text{minimize} & \normNuc{X}  \\
                    \text{subject to} & \norm{y - \sampling(X)} \le \xi.
                \end{array}
            \end{align}
            Let $\Sigma_\opt$ denote the solution of \Cref{eq:matrix_minimization_problem}. To give insights, the nuclear norm is seen here as a convex relaxation of the rank function \cite{Fazel08}, so that \Cref{eq:matrix_minimization_problem} is expected to capture a low-rank matrix close to $\Sigma$. 
            If $\sampling$ satisfies the restricted isometry property, the next result states that \eqref{eq:matrix_minimization_problem} does indeed capture such a low-rank matrix.
            In what follows, we let $\Sigma^{(d)} \in \R^{k \times k}$ denote the matrix closest to $\Sigma$ among all the matrices of rank $d$, where closeness is indifferently measured in nuclear, Frobenius, or operator norm. That is, $\Sigma^{(d)}$ is the truncated singular value decomposition of~$\Sigma$.
    
            \begin{theorem}[{\cite[Theorem 4]{Fazel08}}]
                \label{thm:RIP_implies_Stable_Reconstruction}
                Assume that $\delta_{5d} < 1/10$. Then the solution $\Sigma_\opt$ of \Cref{eq:matrix_minimization_problem} satisfies
                 \begin{align*}
                    \normF{\Sigma_\opt - \Sigma}
                    \le
                    C_0
                    \frac{\normNuc{\Sigma - \Sigma^{(d)}}}{\sqrt{d}}
                    +
                    C_1 \xi,
                \end{align*}
                where $C_0, C_1 > 0$ are universal constants.
            \end{theorem}

    \subsubsection{Building a Good Matrix Sensing Operator}
    \label{subsubsec:pauli_measurement}
    We now detail a standard way to build a sampling operator $\sampling$ that satisfies the restricted isometry property (\Cref{def:RIP}), thereby allowing to recover low-rank matrices from a few measurements (\Cref{thm:RIP_implies_Stable_Reconstruction}). For purely technical reasons, we shall present a construction over the complex linear space $\C^{k \times k}$. This will eventually enable us to recover results over $\R^{k \times k}$ via the isometry $\R^{k \times k} \hookrightarrow \C^{k \times k}$.
        
    First, we note that given an orthonormal $\C$-basis $\mathbb{W} = (W_1, \dots, W_{k^2})$ of $\C^{k \times k}$ for the Hermitian inner product $\inner{A}{B} = \tr(A^* B)$, we can build a sampling operator $\sampling_{\mathbb{S}}: \C^{k\times k} \to \C^q$ by projecting orthogonally onto the space spanned by only $q$ (randomly) pre-selected $\mathbb{S} \subseteq \mathbb{W}$ elements of the basis.
 
    When $k = 2^\ell$, an orthonormal basis of $\C^{k \times k}$ of particular interest is the so-called \emph{Pauli} basis~\cite{Liu11}. Its construction goes as follows:
    \begin{itemize}[leftmargin=*]
        \item For $k = 2$ ($\ell = 1$), it is defined by  $W^{(1)}_i = \sigma_i / \sqrt{2}$, where
            \[
                \sigma_1 
                = 
                \begin{pmatrix}
                    0 & 1 \\
                    1 & 0
                \end{pmatrix}
                ,
                \quad
                \sigma_2 
                = 
                \begin{pmatrix}
                    0 & -i \\
                    i & 0
                \end{pmatrix}
                ,
                \quad
                \sigma_3 
                = 
                \begin{pmatrix}
                    1 & 0 \\
                    0 & -1
                \end{pmatrix}
                ,
                \quad
                \sigma_4 
                = 
                \begin{pmatrix}
                    1 & 0 \\
                    0 & 1
                \end{pmatrix}.
            \]
            Note that the $\sigma_i$'s have two eigenvalues, both belonging to $\set{-1,1}$, so that they are both Hermitian and unitary. In particular, $\normop{W^{(1)}_i} = 1 / \sqrt{2}$ and $\normF{W^{(1)}_i} = 1$ for all $i \in \set{1, \dots, 4}$.  One easily checks that $\bigl(W^{(1)}_i\bigr)_{1 \le i \le 4}$ is an orthonormal basis of $\C^{2 \times 2}$. 
        \item For $k = 2^\ell$ ($\ell \ge 2$), the Pauli basis $\bigl(W^{(\ell)}_i\bigr)_{1 \le i \le 2^\ell}$ is composed of matrices acting on the tensor space $\left(\C^2\right)^{\otimes \ell} \simeq \C^{2^\ell}$, and defined as the family of all the possible $\ell$-fold tensor products of elements of $\bigl(W^{(1)}_i\bigr)_{1 \le i \le 4}$. As tensor products preserve orthogonality, we get that $\bigl(W^{(\ell)}_i\bigr)_{1 \le i \le 2^\ell}$ is an orthonormal basis of $\C^{2^\ell \times 2^\ell}$. Furthermore, as $\normop{W \otimes W'} = \normop{W} \normop{W'}$, we get that for all $i \in \set{1, \dots, 2^\ell}$,
                \begin{align}
                    \label{eq:Pauli_operator_norm}
                    \normop{ W^{(k)}_i}
                    =
                    \left(\frac{1}{\sqrt{2}}\right)^\ell
                    =
                    \frac{1}{\sqrt{k}}
                    .
                \end{align}
                Since $\normF{W} \le \sqrt{k} \normop{W}$, the value $1/\sqrt{k}$ actually is the smallest possible common operator norm of an orthonormal basis of $\C^{k \times k}$.
                As will be clear in the proof of \Cref{lem:Pauli_Satisfy_RIP}, this last property --- called incoherence in the matrix completion literature \cite{Liu11} --- is key to design a good sampling operator.
        \end{itemize}
        Still considering the case $k = 2^\ell$, we let $\samplingPauli: \C^{k \times k} \to \C^q$ denote the random sampling operator defined by
        \begin{align}
        \label{eq:samplingPauli}
            \samplingPauli(X)
            =
            \left(
                \frac{k}{\sqrt{q}}\inner{W_{I_i}^{(\ell)}}{X}
            \right)_{1 \le i \le q},
        \end{align}
        where $(I_i)_{1 \le i \le q}$ is an i.i.d. sequence with uniform distribution over $\set{1, \dots, k^2}$.
        Up to the factor $k/\sqrt{q}$, $\samplingPauli$ is the orthogonal projector onto the space spanned by $(W^{(\ell)}_{I_1}, \dots,W^{(\ell)}_{I_q})$. This normalisation $k/\sqrt{q}$ is chosen so that for all $X \in \C^{k \times k}$,
        \begin{align*}
            \E\bigl[\norm{\samplingPauli(X)}^2 \bigr] 
            &=
            \frac{k^2}{q}
            \sum_{i = 1}^{k^2}
            q \Pr\left(I_1 = i\right)\left|\inner{W_{i}^{(\ell)}}{X}\right|^2
            \\
            &=
            \sum_{i = 1}^{k^2}
            \left|\inner{W_{i}^{(\ell)}}{X}\right|^2
            =
            \normF{X}^2
            .
        \end{align*}
        That is, roughly speaking, $\samplingPauli$ satisfies the restricted isometry property (RIP, \Cref{def:RIP}) on average.
        Actually, as soon as $q$ is large enough compared to $d$, the result below states that $\samplingPauli$ does fulfill a restricted isometry property with high probability.
    
            \begin{lemma}
                \label{lem:Pauli_Satisfy_RIP}
                Assume that $k = 2^\ell$, and fix $0 < \alpha \le 1$. 
                There exist universal constants $c_0,c_1 > 0$ such that if $q \ge c_0 k d \log^6(k) \log(c_1/\alpha)$, then with probability at least $1 - \alpha$, the following holds.
                
                For all $X \in \R^{k \times k}$ such that $\normNuc{X} \le \sqrt{5d} \normF{X}$,
                \[
                \left| \norm{\samplingPauli(X)} - \normF{X} \right| \le \frac{\normF{X}}{20}.
                \]
                In particular, on the same event of probability at least $1 - \alpha$, $\delta_{5d}\left(\samplingPauli\right) < 1/10$.
            \end{lemma}
    
            \begin{proof}[\proofof \Cref{lem:Pauli_Satisfy_RIP}]
                The Pauli basis is an orthonormal basis of $\C^{k \times k}$, and from \Cref{eq:Pauli_operator_norm}, its elements all have operator norm smaller than $1/\sqrt{k}$. 
                Hence, applying \cite[Theorem 2.1]{Liu11} with $K=\sqrt{k} \max_{1 \le i \le k} \normop{W^{(\ell)}_i} = 1$, $r = 5d$, $C = c_0 \log(c_1/\alpha)$, and $\delta = 1/20$ yields the first bound.
                The second one follows by recalling that any rank-$r$ matrix $X \in \R^{k \times k}$ satisfies $\normNuc{X} \le \sqrt{r} \normF{X}$.
            \end{proof}

        \subsubsection{Mean Matrix Completion with Statistical Queries}    
        
        The low-rank matrix recovery of \Cref{subsubsec:RIP-implies-recovery,subsubsec:pauli_measurement} combined with mean vector estimation in $\STAT(\tau)$ for the Euclidean norm (see \Cref{lem:mean_vector_estimation}) lead to the following result.
        
            \begin{lemma}
            \label{lem:mean_matrix_estimation}
                For all $\alpha \in (0,1]$, there exists a family of statistical query algorithms indexed by maps $F: \R^n \to \R^{k \times k}$ such that the following holds on an event of probability at least $1-\alpha$ (uniformly over $F$).
                
                Let $D$ be a Borel probability distribution over $\R^n$, and $F: \R^n \to \R^{k \times k}$ be a map such that for all $x \in \R^n$, $\normF{F(x)}\le 1$ and $\normNuc{F(x)} \le \sqrt{5d} \normF{F(x)}$.
                Write $\Sigma = \E_{x \sim D}\left[ F(x) \right]$, and $\Sigma^{(d)}$ for the matrix closest to $\Sigma$ among all the matrices of rank $d\le k$.
                Assume that $\Sigma \in \Xi$, where $\Xi \subseteq \R^{k \times k}$ is a known linear subspace of $\R^{k \times k}$.
            
                Then, there exists a statistical query algorithm making at most $c_0 dk\log^6(k) \log(c_1/\alpha)$ queries to $\STAT(\tau)$, and that outputs a matrix $\hat{\Sigma} \in \Xi$ that satisfies
                \begin{align*}
                    \normF{\hat{\Sigma} - \Sigma}
                    \le
                    C_0
                    \frac{\normNuc{\Sigma - \Sigma^{(d)}}}{\sqrt{d}}
                    +
                    C_1 \tau
                \end{align*}
                on the event of probability at least $1-\alpha$ described above, where $C_0,C_1 > 0$ are universal constants.
        \end{lemma}
        
        \begin{proof}[\proofof \Cref{lem:mean_matrix_estimation}]
                Without loss of generality, we can assume that $k = 2^\ell$. Indeed, one can always embed $\R^{k\times k}$ isometrically into $\R^{2^\ell \times 2^\ell}$, with $2^\ell = 2^{\ceil{\log_2(k)}} \le 2k $, via the linear map
                \[
                    \R^{k \times k} 
                    \ni 
                    A
                    \longmapsto
                    \tilde{A}
                    = \left( 
                        \begin{array}{c|c} 
                            A & 0 \\
                            \hline
                            0 &  0  
                        \end{array} 
                    \right)
                    \in
                    \R^{2^\ell \times 2^\ell},
                \]
                which preserves both the rank, the Frobenius and nuclear norms.
                
                Let $q \ge 1$ be a fixed integer to be specified later, and $(I_i)_{1 \le i \le q}$ be and i.i.d. sequence with uniform distribution over $\set{1, \dots, k^2}$, and for $X \in \R^{k \times k}$, write
                    \begin{align}
                        \samplingPauli(X)
                        =
                        \left(
                        \frac{k}{\sqrt{q}}\inner{W_{I_i}^{(\ell)}}{X}
                        \right)_{1 \le i \le q}
                        \in 
                        \C^{q}
                        =
                        \R^{2q}
                        ,
                    \end{align}
                as in \Cref{eq:samplingPauli}.
                For $x \in \R^n$, write $G(x) = \samplingPauli(F(x))/2 \in \R^{2q}$.
                From \Cref{lem:Pauli_Satisfy_RIP}, with probability at least $1- \alpha$ (over the randomness of $(I_i)_{1 \le i \le q}$),
                \begin{align*}
                    \norm{G(x)}
                    &=
                    \norm{\samplingPauli(F(x))}/2
                    \le
                    (1 + 1/20)\normF{F(x)}/2
                    \le
                    \normF{F(x)}
                    \le
                    1
                \end{align*}
                holds simultaneously for all the described $F: \R^{k\times k} \to \R^{2q}$.
                Hence, on this event of probability at least $1-\alpha$, \Cref{lem:mean_vector_estimation} applies to $G$ and provides a deterministic statistical query algorithm making $4q$ queries to $\STAT(\tau)$, and that outputs a vector $y \in \R^{2q}$ such that
                \[
                    \norm{
                    y
                    -
                    \E_{x \sim D}\left[G(x)\right]
                    }
                    \le
                    C \tau
                    ,
                \]
                where $C>0$ is a universal constant.
                But on the other hand, by linearity,
                \begin{align*}
                    \E_{x \sim D}\left[2 G(x) \right]
                    &=
                    \E_{x \sim D}\left[\samplingPauli(F(x)) \right]
                    =
                    \samplingPauli\left(\E_{x \sim D}\left[F(x)\right]\right)
                    =
                    \samplingPauli(\Sigma),
                \end{align*}
                where all the expected values are taken with respect $D$, conditionally on $(I_i)_{1 \le i \le q}$.
                Hence, as soon as $q \ge c_0 dk\log^6(k) \log(c_1/\alpha)$, \Cref{thm:RIP_implies_Stable_Reconstruction} and \Cref{lem:Pauli_Satisfy_RIP} combined together yields the following: on the same event of probability at least $1 - \alpha$ as before, the solution $\Sigma_\opt$ to the convex optimization problem over $X \in \R^{k \times k}$ given by
                \begin{align*}
                    \begin{array}{ll}
                        \text{minimize} & \normNuc{X}  \\
                        \text{subject to} & \norm{2y - \samplingPauli(X)} \le 2 C\tau,
                    \end{array}
                \end{align*}
                satisfies
                \begin{align*}
                    \normF{\Sigma_\opt - \Sigma}
                    \le
                    C_0
                    \frac{\normNuc{\Sigma - \Sigma^{(d)}}}{\sqrt{d}}
                    +
                    C_1 (2 C \tau)
                    .
                \end{align*} 
                Hence, the projected solution $\hat{\Sigma} = \pi_\Xi(\Sigma_\opt)$ onto $\Xi \subseteq \R^{k \times k}$ belongs to $\Xi$ and satisfies
                \begin{align*}
                    \normF{\hat{\Sigma} - \Sigma}
                    =
                    \normF{\pi_\Xi(\Sigma_\opt - \Sigma)}
                    &\le
                    \normF{\Sigma_\opt - \Sigma}
                    \\
                    &\le
                    C_0
                    \frac{\normNuc{\Sigma - \Sigma^{(d)}}}{\sqrt{d}}
                    +
                    C_1'\tau
                    ,
                \end{align*}
                which concludes the proof.
        \end{proof}
     
    \subsection{Tangent Space Estimation with Statistical Queries}
        \label{subsec:tangent-space-estimation-with-SQs}
        We finally prove the main announced statement of \Cref{sec:SQ-tangent-appendix}.
    \begin{proof}[\proofof \Cref{thm:routine_tangent}]
        Let $h >0$ be a bandwidth to be specified later, such that 
        $\eta \le h/\sqrt{2}$ and $h \le \rch_{\min}/(8\sqrt{d})$.
        First note that 
        $\Sigma_D(x_0,h) = \E_{x \sim D} \left[ F(x) \right]$, where 
        the function $F(x) = (x-x_0)\transpose{(x-x_0)}/h^2 \mathbbm{1}_{\norm{x-x_0} \le h}$ is defined for all $x \in \R^n$, and is such that $\normF{F(x)}\le 1$ and 
        $\rank(F(x)) \le 1$. In particular, 
        $\normNuc{F(x)} = \normF{F(x)} \le \sqrt{5d} \normF{F(x)}$ 
        for all $x \in \R^{n}$. Furthermore, $\Sigma_D(x_0,h)$ belongs to
        the linear space $\Xi$ of symmetric matrices. 
        Working on the event on which \Cref{lem:mean_matrix_estimation} holds (with $\alpha = 1/2$, say), yields the existence of a deterministic
        SQ algorithm making at most $c_0 dn\log^6(n) \log(2c_1)$
        queries to $\STAT(\tau)$, and that outputs a symmetric matrix $\hat{\Sigma}$ that satisfies
        \begin{align*}
            \normF{\hat{\Sigma} - \Sigma_D(x_0,h)}
            \le
            C_0
            \frac{\normNuc{\Sigma_D(x_0,h) - \Sigma^{(d)}_D(x_0,h)}}{\sqrt{d}}
            +
            C_1 \tau
            ,
        \end{align*}
        with probability at least $1 - \alpha$.
        On the other hand, from \Cref{lem:covariance_decomposed}, provided that $\sqrt{2} \eta \le h \le \rch_{\min}/(8\sqrt{d})$, one can write
        \[
            \Sigma_D(x_0,h) = \Sigma_0 + R,
        \]
        where the symmetric matrix $\Sigma_0$ satisfies 
        $\operatorname{Im}(\Sigma_0) = T_{\pi_M(x_0)} M$, 
        $\mu_d(\Sigma_0) \ge \omega_d f_{\min}(ch)^d$
        and 
        $\normF{R} \le \normNuc{R} \le \omega_d f_{\max} (C h)^d \left( \frac{\eta}{h} + \frac{h}{\rch_{\min}}\right)$.
        As $\rank(\Sigma_0) = d$, we have in particular that, 
        \[
            \normNuc{\Sigma_D(x_0,h) - \Sigma^{(d)}_D(x_0,h)}
            \le
            \normNuc{\Sigma_D(x_0,h) - \Sigma_0}
            =
            \normNuc{R}. 
        \]
        Therefore, taking $\hat{T}(x_0)$ as the linear space 
        spanned by the first $d$ eigenvectors of $\hat{\Sigma}$, \Cref{lem:angledeviation} yields
        \begin{align*}
            \angle\bigl(T_{\pi_M(x_0)} M, \hat{T}(x_0)\bigr)
            &=
            \normop{\pi_{\hat{T}(x_0)} - \pi_{T_{\pi_M(x_0)} M}}
            \\
            &\le
            \frac{
                2 \normF{\hat{\Sigma} - \Sigma_0}
            }{
                \mu_d(\Sigma_0)
            }
            \\
            &\le
            2
            \frac{
                \normF{\hat{\Sigma} - \Sigma_D(x_0,h)} +  
                    \normF{\Sigma_D(x_0,h) - \Sigma_0}
            }
            {
                \mu_d(\Sigma_0)
            }
            \\
            &\le
            \frac{
                2
            }{
                 \omega_d f_{\min}(ch)^d
            }
            \left(
                C_0
               \frac{\normNuc{R}}{\sqrt{d}}
                +
                C_1 \tau
                +
                \normF{R}
            \right)
            \\
            &\le
            \frac{C'^{d}}{\omega_d f_{\min}}
            \left( 
                \omega_d f_{\max}
                \left\{
                    \frac{\eta}{h} + \frac{h}{\rch_{\min}} 
                \right\}
                + 
                \frac{\tau}{h^d}
            \right)
            .
        \end{align*}
        We conclude by setting $
        h 
        = 
        \rch_{\min}
        \left\{
            \sqrt{\frac{\eta}{\rch_{\min}}}
            \vee 
            \left( \frac{\tau}{\omega_d f_{\max} \rch_{\min}^d} \right)^{1/(d+1)}
        \right\}
        $
        in this last bound.
        This value for $h$ does satisfy 
        $\sqrt{2} \eta \le h \le \rch_{\min}/(8\sqrt{d})$ since $\eta \le \rch_{\min}/(64d)$ and 
        $\frac{\tau}{\omega_d f_{\max} \rch_{\min}^d} \leq \left(\frac{1}{8\sqrt{d}}\right)^{d+1}$,
        so that the whole analysis applies, and yields the announced result.
    \end{proof}

%% file: parts/seed.tex
We now build the SQ point detection algorithm $\hat{x}_0 \in \R^n$
(\Cref{thm:routine_detection_seed_point}), which is used to initialize
in the SQ emulation of \MPA yielding the SQ reconstruction algorithm 
in the model 
$\distributionsball{n}{d}{\rch_{\min}}{f_{\min}}{f_{\max}}{L}{R}$ 
where no seed point is available (\Cref{def:models_bounded}).

Recall that given a ball of radius $R>0$ guaranteed to encompass 
$M = \supp(D) \subseteq \B(0,R)$, and a target precision $\eta > 0$,
we aim at finding a point that is $\eta$-close to $M$ with 
statistical queries to $\STAT(\tau)$. We follow the strategy of proof 
described in \Cref{subsec:SQ-seed-presentation}.

\subsection{Detecting a Raw Initial Point}

Starting from the whole ball $\B(0,R)$, the following result allows us to 
find a point nearby $M$ using a binary search, with best precision of 
order $\Omega(\tau^{1/d})$. Let us note that it does not explicitly rely on
any differential property of $M$, but only the behavior of the mass of
balls for $D$ (\Cref{lem:intrinsic_ball_mass}).
    \begin{theorem}
    \label{thm:routine_detection_raw_point}
        Let $D \in \distributions{n}{d}{\rch_{\min}}{f_{\min}}{f_{\max}}{L}$
        have support $M = \supp(D) \subseteq \B(0,R)$. Let 
        $\Lambda_0 \leq \rch_{\min}/8$ be fixed, and assume that
        $
            \frac{\Lambda_0}{\sqrt{\log(6R/\Lambda_0)}}
            \geq
            21 \rch_{\min} \sqrt{n}
            \left(\frac{\tau}{\omega_d f_{\min}\rch_{\min}^d}\right)^{1/d} 
        $.
        
        Then there exists a deterministic statistical query algorithm making
        at most $3 n \log(6R/\Lambda_0)$ queries to $\STAT(\tau)$, 
        and that outputs a point $\hat{x}_0^{raw} \in \B(0,R)$ such that 
        \[
            \dd(\hat{x}_0^{raw},M)
            \le 
            \Lambda_0
            .
        \]
    \end{theorem}

    \begin{remark}
        Recall from \Cref{subsubsec:implicit-bounds-on-parameters} that we always assume that $R \geq \rch_{\min}/\sqrt{2}$ to ensure that the model is nonempty. As a result $\log(6R/\Lambda_0) \geq 0$ for all $\Lambda_0 \leq \rch_{\min}/8$.
    \end{remark}

    \begin{proof}[\proofof \Cref{thm:routine_detection_raw_point}]
        The idea is to use a divide and conquer strategy over a covering $\set{x_i}_{1 \leq i \leq N}$ of $\B(0,R)$. The algorithm recurses over a subset of indices $\mathcal{I} \subseteq \set{1,\ldots,N}$ that is maintained to fulfill $\cup_{i \in \mathcal{I}} \B(x_i,h) \cap M \neq \emptyset$ for some known $h>0$.
        This property can be checked with the single query $r = \indicator{\cup_{i \in \mathcal{I}} \B(x_i,h)}$ to $\STAT(\tau)$, provided that $D(\cup_{i \in \mathcal{I}} \B(x_i,h)) > \tau$. To ensure the later, the radius $h>0$ is dynamically increased at each iteration. The algorithm stops when $\mathcal{I}$ is reduced to a singleton.
        More formally, we consider \ABS.
        \begin{algorithm}
            \begin{algorithmic}[1]
                \REQUIRE ~\\
                Model parameters $d,\rch_{\min},f_{\min}$ \\
                Precision $\Lambda_0 > 0$
                ~
                \STATE Initialize value $h \leftarrow \Lambda_0/2$, and set $\Delta = 6\rch_{\min} \left(\frac{\tau}{\omega_d f_{\min} \rch_{\min}^d}\right)^{1/d}$
                \STATE
                Consider a minimal $(\Lambda_0/2)$-covering $\set{x_i}_{1 \leq i \leq N}$ of $\B(0,R)$, where $N = \CV_{\B(0,R)}(\Lambda_0/2)$
                \STATE Initialize sets $\mathcal{I} \leftarrow \set{1,\ldots,N}$, $\mathcal{L} \leftarrow \emptyset$ and $\mathcal{R} \leftarrow \emptyset$
                \WHILE{$|\mathcal{I}| > 1$}
                    \STATE Split $\mathcal{I} = \mathcal{L} \cup \mathcal{R}$ into two disjoint sets $\mathcal{L} \cap \mathcal{R} = \emptyset$ such that $\left| |\mathcal{L}| - |\mathcal{R}| \right| \leq 1$
                    \STATE Query $r= \indicator{\cup_{i \in \mathcal{L}} \B\left(x_i,\sqrt{h^2+\Delta^2}\right)}$ to the $\STAT(\tau)$ oracle
                    \STATE $\answer$ $\leftarrow$ Value answered by the oracle
                    \IF{$\answer > \tau$}
                        \STATE $\mathcal{I} \leftarrow \mathcal{L}$
                    \ELSE
                        \STATE $\mathcal{I} \leftarrow \mathcal{R}$
                    \ENDIF
                    \\
                    $h \leftarrow \sqrt{h^2 + \Delta^2}$
                \ENDWHILE
                \RETURN The only element of $\hat{x}_0^{raw}$ of $\set{x_i}_{i \in \mathcal{I}}$
            \end{algorithmic} 
            \caption{\ABS}
            \label{alg:routine_detection_raw_point}
        \end{algorithm}
        
        Because $|\mathcal{I}|$ is a decreasing sequence of integers, it is clear that \ABS terminates, and that $|\mathcal{I}_{final}|=1$ so that the output $\hat{x}_0^{raw}$ is well defined. 
        As each \emph{while} loop does only one query to $\STAT(\tau)$, and that $N = \CV_{\B(0,R)}(\Lambda_0/2) \leq (6R/\Lambda_0)^n$ from \Cref{prop:packing_covering_ball_sphere} and $\Lambda_0 \leq R$, it makes at most $\floor{\log_2(N) +1 } \leq \floor{ n \log(6R/\Lambda_0)/\log(2) +1} \leq 3 n \log(6R/\Lambda_0)$ queries in total.
        
        Let us now prove that the output $\hat{x}_0^{raw}$ satisfies $\dd(\hat{x}_0^{raw},M) \leq \Lambda_0$.
         For this, we show that when running \ABS, the inequality $\min_{i \in \mathcal{I}} \dd(x_i,M) \leq h$ is maintained (recall that both $\mathcal{I}$ and $h$ are dynamic), or equivalently that $\cup_{i \in \mathcal{I}} \B\left(x_i,h\right) \cap M \neq \emptyset$.
         At initialization, this is clear because $\mathcal{I} = \set{1,\ldots,N}$, $h= \Lambda_0/2$, and $\set{x_i}_{1 \leq i \leq N}$ is a $(\Lambda_0/2)$-covering of $ \B(0,R) \supseteq M$.
         Then, proceeding by induction, assume that $\cup_{i \in \mathcal{I}} \B\left(x_i,h\right) \cap M \neq \emptyset$ when entering an iteration of the \emph{while} loop. Let $i_0 \in \mathcal{I}$ be such that $\dd(x_{i_0},M) \leq h$.
         From \Cref{lem:intrinsic_ball_mass}, provided that $\sqrt{h^2+\Delta^2} \leq \rch_{\min}/8$, we have
            \begin{align}
                D\left(
                    \cup_{i \in \mathcal{I}}
                    \B\left(x_i,\sqrt{h^2+\Delta^2}\right)
                \right)
                &\geq
                D\left( \B\left(x_{i_0},\sqrt{h^2+\Delta^2}\right) \right)
                \notag
                \\
                &\geq
                (\sqrt{7/24})^d \omega_d f_{\min} \left( (h^2+\Delta^2) - \dd(x_{i_0},M)^2 \right)^{d/2}
                \notag
                \\
                &\geq
                (\sqrt{7/24})^d \omega_d f_{\min} \Delta^{d}
                \notag
                \\
                &=
                (\sqrt{7/24})^d 6^d \tau
                \notag
                \\
                &>
                2 \tau
                .
                \label{eq:mass_thm:routine_detection_raw_point}
            \end{align}
            Hence, if we let $\answer$ denote the answer of the oracle to the query 
            $r= \indicator{\cup_{i \in \mathcal{L}} \B\left(x_i,\sqrt{h^2+\Delta^2}\right)}$,
            we have:
            \begin{itemize}[leftmargin=*]
                \item 
                If $\answer > \tau$, then 
                \begin{align*}
                D\left(
                    \cup_{i \in \mathcal{L}}
                    \B\left(x_i,\sqrt{h^2+\Delta^2}\right)
                \right)
                \geq 
                \answer - \tau 
                >
                0
                ,
                \end{align*}
                so that after the updates $\mathcal{I} \leftarrow \mathcal{L}$ and $h \leftarrow \sqrt{h^2 + \Delta^2}$, we still have 
                $
                \cup_{i \in \mathcal{I}}
                \B\left(x_i,h\right) \cap M \neq \emptyset
                .
                $
                \item 
                Otherwise $a \leq \tau$, so that from \Cref{eq:mass_thm:routine_detection_raw_point},
                \begin{align*}
                D\left(
                    \cup_{i \in \mathcal{R}}
                    \B\left(x_i,\sqrt{h^2+\Delta^2}\right)
                \right)
                &\geq
                D\left(
                    \cup_{i \in \mathcal{I}}
                    \B\left(x_i,\sqrt{h^2+\Delta^2}\right)
                \right)
                -
                D\left(
                    \cup_{i \in \mathcal{L}}
                    \B\left(x_i,\sqrt{h^2+\Delta^2}\right)
                \right)
                \\
                &>
                2\tau - (\answer+\tau)
                \\
                &\geq
                0
                .
                \end{align*}
                So as above, after the updates $\mathcal{I} \leftarrow \mathcal{R}$ and $h \leftarrow \sqrt{h^2 + \Delta^2}$, we still have 
                $
                \cup_{i \in \mathcal{I}}
                \B\left(x_i,h\right) \cap M \neq \emptyset
                .
                $
            \end{itemize}
            Consequently, when the algorithm terminates, we have 
            \begin{align*}
                \dd(\hat{x}_0^{raw},M)
                &\leq
                h_{final}
                \\
                &\leq
                \sqrt{
                    \left( \frac{\Lambda_0}{2} \right)^2
                    +
                    3n \log(6R/\Lambda_0)
                    \Delta^2
                }
                \\
                &\leq
                \frac{\Lambda_0}{2}
                +
                \sqrt{3n \log(6R/\Lambda_0) }
                6 \rch_{\min} \left(\frac{\tau}{\omega_d f_{\min}\rch_{\min}^d}\right)^{1/d}
                \\
                &\leq
                \Lambda_0
                ,
            \end{align*}
            since 
            $
            \frac{\Lambda_0}{\sqrt{\log(6R/\Lambda_0)}}
            \geq
            21 \rch_{\min} \sqrt{n}
            \left(\frac{\tau}{\omega_d f_{\min}\rch_{\min}^d}\right)^{1/d} 
            $.
            The above also shows that when running the algorithm we have $\sqrt{h^2+\Delta^2} \leq h_{final} \leq \Lambda_0 \leq \rch_{\min}/8$, which ensures that \Cref{eq:mass_thm:routine_detection_raw_point} is valid throughout and concludes the proof.
    \end{proof}
    
\subsection{Refined Point Detection}
    \label{subsec:seed-point-with-SQs}
    
    We finally prove the main announced statement of \Cref{sec:SQ-seed-appendix}.
    
\begin{proof}[\proofof \Cref{thm:routine_detection_seed_point}]
    The idea is to first detect a possibly coarse base point $\hat{x}_0^{raw}$ using a divide and conquer strategy in the ambient space (\Cref{thm:routine_detection_raw_point}), and then refine it by considering iterated projections of $\hat{x}_0^{raw}$ given by the local conditional mean (\Cref{thm:routine_projection}).
    More precisely, let $\hat{x}_0^{raw}$ be the output of the point detection SQ algorithm of \Cref{thm:routine_detection_raw_point} applied with parameter
    \[
        \Lambda_0
        =
        \max
        \set{
        \eta
        ,
        \min\set{\frac{1}{16} ,\frac{\Gamma}{2C^d} } \rch_{\min}
        }
        ,
    \]
    where $C^d,\Gamma>0$ are the constants of \Cref{thm:routine_projection}.
    From the assumptions on the parameters (recall also that we necessarily have $R \geq \rch_{\min}/\sqrt{2}$, see \Cref{subsubsec:implicit-bounds-on-parameters}), we have $\Lambda_0 \leq \rch_{\min}/8$ and 
    \[
        \frac{\Lambda_0}{\sqrt{\log(6R/\Lambda_0)}}
        \geq
        21 \rch_{\min} \sqrt{n}
        \left(\frac{\tau}{\omega_d f_{\min}\rch_{\min}^d}\right)^{1/d}
        ,
    \]
    so that \Cref{thm:routine_detection_raw_point} applies and guarantees that $\hat{x}_0^{raw}$ can be obtained with at most $3 n \log(6R/\Lambda_0)$ queries to $\STAT(\tau)$ and satisfies $\dd(\hat{x}_0^{raw},M) \leq \Lambda_0$.
    
    If $\Lambda = \eta$ --- condition which can be checked by the learner since the parameters $\eta, \Gamma, d$ and $\rch_{\min}$ are assumed to be known ---, then $\hat{x}_0 := \hat{x}_0^{raw}$ clearly satisfies $\dd(\hat{x}_0,M) = \dd(\hat{x}_0^{raw},M) \leq \eta$, and has required at most $3 n \log(6R/\Lambda_0) = 3 n \log(6R/\eta)$ queries to $\STAT(\tau)$.
    Otherwise, $\eta \leq \Lambda_0$, and we iterate the SQ approximate projections $\hat{\pi}(\cdot)$ given by \Cref{thm:routine_projection}.
    Namely, we let $\hat{y}_0 = \hat{x}_0^{raw}$ and for all integer $k\geq 1$, $\hat{y}_k = \hat{\pi}(\hat{y}_{k-1})$.
    In total, note that the computation of $\hat{y}_k$ requires at most $3 n \log(6R/\eta) + k(2n+1) \leq 3n\bigl(\log(6R/\eta)+k\bigr)$ queries to $\STAT(\tau)$. Similarly as above, from the assumptions on the parameters, one easily shows by induction that since $\dd(\hat{y}_0,M) \leq \Lambda_0 \leq \frac{\rch_{\min}}{16}$, \Cref{thm:routine_projection} applies to each $\hat{y}_k$ and guarantees that
    \begin{align*}
        \dd(\hat{y}_k,M)
        &=
        \dd(\hat{\pi}(\hat{y}_{k-1}),M)
        \\
        &\leq
        \norm{\hat{\pi}(\hat{y}_{k-1}) - \pi_M(\hat{y}_{k-1}))}
        \\
        &\leq
            \max\set{
                \frac{C^d\dd(\hat{y}_{k-1},M)^2}{\Gamma \rch_{\min}}
            ,
                C^d\Gamma^{\frac{2}{d+1}-1}
                \rch_{\min}
                \left(\frac{\tau}{\omega_d f_{\min} \rch_{\min}^d} \right)^{\frac{2}{d+1}}
            }
            \\
            &\leq
            \max\set{
                \frac{\dd(\hat{y}_{k-1},M)}{2}
            ,
                C^d\Gamma^{\frac{2}{d+1}-1}
                \rch_{\min}
                \left(\frac{\tau}{\omega_d f_{\min} \rch_{\min}^d} \right)^{\frac{2}{d+1}}
            }
            \\
            &\leq
            \max\set{
                \frac{\Lambda_0}{2^k}
            ,
                C^d\Gamma^{\frac{2}{d+1}-1}
                \rch_{\min}
                \left(\frac{\tau}{\omega_d f_{\min} \rch_{\min}^d} \right)^{\frac{2}{d+1}}
            }
            .
    \end{align*}
    To conclude, fix $k_0 := \ceil{\log_2\left(\Lambda_0/\eta\right)} \leq
    \log\left(6 \Lambda_0/\eta\right)$, and set $\hat{x}_0 := \hat{y}_{k_0}$.
    From the previous bound, we obtain that
    \begin{align*}
        \dd(\hat{x}_0,M)
        &\leq
        \max\set{
            \eta
        ,
            C^d\Gamma^{\frac{2}{d+1}-1}
            \rch_{\min}
            \left(\frac{\tau}{\omega_d f_{\min} \rch_{\min}^d} \right)^{\frac{2}{d+1}}
        }
        ,
    \end{align*}
    with $\hat{x}_0$ requiring at most $3n\bigl(\log(6R/\eta)+\log\left(6 \Lambda_0/\eta\right)\bigr) \leq 6 n \log(6R/\eta)$ queries to $\STAT(\tau)$ to be computed, which concludes the proof.
\end{proof}

%% file: parts/SQ-upper-bounds.tex
This section is devoted to the proof of the two SQ manifold estimation upper bounds: the first one in the fixed point model $\distributionspoint{n}{d}{\rch_{\min}}{f_{\min}}{f_{\max}}{L}{0}$ (\Cref{thm:SQ_upper_bound_point}), and the second one for the bounding ball model $\distributionsball{n}{d}{\rch_{\min}}{f_{\min}}{f_{\max}}{L}{R}$ (\Cref{thm:SQ_upper_bound_ball}).

\begin{proof}[\proofof \Cref{thm:SQ_upper_bound_point}]
Let us write 
$$
\Delta := \rch_{\min} \max\set{\sqrt{ \varepsilon/(\rch_{\min}\bar{C}_d)}, \mathbf{C}^\frac{3}{2}\bigl( \tau/(\omega_d f_{\min} \rch_{\min}^d) \bigr)^{\frac{1}{d+1}}}
,
$$
for some large enough $\bar{C}_d>0$ depending on $d$ and $\mathbf{C}$ to be chosen later, and $\delta = \Delta/2$. 
We will run \MPA with scale parameters $\Delta$, $\delta$, angle $\sin \alpha = 1/64$, and initialization point $\hat{x}_0 = 0 \in M$, the SQ projection routine $\hat{\pi}(\cdot)$ of \Cref{thm:routine_projection} and the SQ tangent space routine $\hat{T}(\cdot)$ of \Cref{thm:routine_tangent}. 
If we prove that these routines are precise enough, then \Cref{thm:MPA_properties} will assert that the output point cloud $\mathcal{O}$ and associated tangent space estimates $\mathbb{T}_\mathcal{O}$ of \MPA fulfill the assumptions of \Cref{thm:manifold_reconstruction_deterministic}. This will hence allow to reconstruct $M$ with a good triangulation, as claimed.

Note by now that at each iteration \MPA, exactly one call to each SQ routine $\hat{\pi}(\cdot)$ and $\hat{T}(\cdot)$ are made, yielding at most $(2n+1) + Cdn \log^6(n) \leq C'd n \log^6(n)$ statistical queries.
But if \Cref{thm:MPA_properties} applies, we get that the number of iteration $N_{\mathrm{loop}}$ of \MPA satisfies
\begin{align*}
    N_{\mathrm{loop}}
    &\leq
    \dfrac{\Haus^d(M)}{\omega_d (\delta/32)^d}
    \\
    &\leq
    \frac{\bar{C}'_d}{f_{\min}(\sqrt{\rch_{\min} \varepsilon})^d}
    \\
    &=
    \frac{\bar{C}'_d}{f_{\min}\rch_{\min}^d}
    \left( \frac{\rch_{\min}}{\varepsilon} \right)^{d/2}
    ,
\end{align*}
where the second inequality comes from the fact that $1 = \int_{M} f \dd \Haus^d \geq f_{\min} \Haus^d(M)$.
In total, the resulting SQ algorithm hence makes at most
\begin{align*}
    q
    &\leq
    \bigl(C'd n \log^6(n)\bigr)
    \frac{\bar{C}'_d}{f_{\min}\rch_{\min}^d}
    \left( \frac{\rch_{\min}}{\varepsilon} \right)^{d/2}
    \\
    &=    
    n \log^6 n \frac{C_d}{f_{\min}\rch_{\min}^d}
     \left( \frac{\rch_{\min}}{\varepsilon} \right)^{d/2}
\end{align*}
queries to $\STAT(\tau)$, which is the announced complexity.
It only remains to verify that the SQ routines $\hat{\pi}(\cdot)$ and $\hat{T}(\cdot)$ are indeed precise enough so that \Cref{thm:MPA_properties} applies, and to bound the final precision given by the triangulation of \Cref{thm:manifold_reconstruction_deterministic}.

To this aim, we notice that the assumption made on $\tau$ puts it in the regime of validity of \Cref{thm:routine_projection} and \Cref{thm:routine_tangent}.
Let us write
\begin{align*}
    \mathbf{C}
    :=
    \max\set{
    C^d \Gamma^{\frac{2}{d+1}-1}
    ,
    \Tilde{C}^d \frac{f_{\max}}{f_{\min}}
    }
    \leq
    \frac{(\max\{C, \Tilde{C}\})^d}{\Gamma}
    ,
\end{align*}
where $C>0$ is the constant of \Cref{thm:routine_projection} and $\Tilde{C}>1$ that of \Cref{thm:routine_tangent}.
Note by now that since $f_{\max} \geq f_{\min}$, we have $\mathbf{C} \geq 1$.
For short, we also let $\ttau := \tau/(\omega_d f_{\min} \rch_{\min}^d)$.
    
    At initialization, and since $D \in \distributionspoint{n}{d}{\rch_{\min}}{f_{\min}}{f_{\max}}{L}{0}$, the seed point $\hat{x}_0 = 0$ belongs to $M$, meaning that
    \begin{align*}
    \dd(\hat{x}_0,M)
    =
    0 
    \leq
    \eta
    :=
    \rch_{\min} 
    \max \set{
    \frac{1}{\mathbf{C}^2}
    \left(\frac{\Delta}{\rch_{\min}}\right)^2
    ,
    \mathbf{C} \ttau^{\frac{2}{d+1}}
    }
    .
    \end{align*}
    Note that from the assumptions on the parameters, $\eta \leq \rch_{\min}/(64d)$. Hence, on the $\eta$-offset $M^\eta$ of $M$, \Cref{thm:routine_tangent} asserts that $\hat{T}(\cdot)$ has precision
    \begin{align*}
    \sin \theta
    &\leq
    \max \set{
        \Tilde{C}^d \frac{f_{\max}}{f_{\min}} 
        \frac{1}{\mathbf{C}}
        \frac{\Delta}{\rch_{\min}}
    ,
         \Tilde{C}^d \frac{f_{\max}}{f_{\min}}
         \sqrt{\mathbf{C}} \ttau^{\frac{1}{d+1}}
    }
    \\
    &\leq
    \max \set{
        \frac{\Delta}{\rch_{\min}}
    ,
         \mathbf{C}^{\frac{3}{2}} \ttau^{\frac{1}{d+1}}
    }
    \end{align*}
    for estimating tangent spaces.
    As a result, we have
    \begin{align*}
        \frac{5\Delta^2}{8\rch_{\min}}
        +
        \eta
        +
        \Delta \sin \theta
        &\leq
        3 \rch_{\min}
        \max\set{
            \left(\frac{\Delta}{\rch_{\min}}\right)^2
            ,
            \frac{\eta}{\rch_{\min}}
            ,
            \frac{\Delta}{\rch_{\min}}
            \sin \theta 
        }
        \\
        &\leq
        3 \rch_{\min}
        \max\set{
            \left(\frac{\Delta}{\rch_{\min}}\right)^2
            ,
            \frac{\eta}{\rch_{\min}}
            ,
            \sin^2 \theta 
        }
        \\
        &\leq
        3\rch_{\min}
        \max\set{
            \left(\frac{\Delta}{\rch_{\min}}\right)^2
            ,
            \mathbf{C}^3 \ttau^{\frac{2}{d+1}}
        }
        \\
        &:=
        \Lambda
        .
    \end{align*}
    Using again the assumptions on the parameters, we have $\Lambda \leq \rch_{\min}/8$.
    Hence, applying \Cref{thm:routine_projection} and elementary simplifications given by the assumptions on the parameters yield that, over the $\Lambda$-offset $M^\Lambda$ of $M$, the projection $\hat{\pi}(\cdot)$ has precision at most
    \begin{align*}
        \eta'
        &\leq
        \rch_{\min}
        \max\left\{
            \left(
                \frac{9 C^d}{\Gamma}
                \left(\frac{\Delta}{\rch_{\min}}\right)^2
            \right)
            \left(\frac{\Delta}{\rch_{\min}}\right)^2
            ,
            \right.
        \\
        &\hspace{8em}
        \left.
            \max\set{
                \frac{9 C^d \mathbf{C}^6}{\Gamma} \ttau^{2/(d+1)}
                ,
                C^d \Gamma^{\frac{2}{d+1}-1}
            }
            \ttau^{2/(d+1)}
        \right\}
        \\
        &=
        \rch_{\min}
        \max\set{
            \left(
                \frac{9 C^d}{\Gamma}
                \left(\frac{\Delta}{\rch_{\min}}\right)^2
            \right)
            \left(\frac{\Delta}{\rch_{\min}}\right)^2
            ,
                C^d \Gamma^{\frac{2}{d+1}-1}
            \ttau^{2/(d+1)}
        }
        \\
        &\leq
        \rch_{\min}
        \max\set{
            \frac{1}{\mathbf{C}^2}
            \left(\frac{\Delta}{\rch_{\min}}\right)^2
            ,
            \mathbf{C}
            \ttau^{2/(d+1)}
        }
        \\
        &=
        \eta
        .
    \end{align*}
    Additionally, one easily checks that $\Delta \leq \rch_{\min}/24$, $\eta \leq \Delta/24$ and $\max\set{\sin \alpha, \sin \theta} \leq 1/64$, so that \Cref{thm:MPA_properties} applies:
    \MPA terminates and outputs a finite point cloud $\mathcal{O}$ such that
    $\max_{x \in \mathcal{O}} \dd(x,M) \le \eta$
    and 
    $\max_{p \in M} \dd(p, \mathcal{O}) \le \Delta + \eta \leq 2\Delta$,  together with tangent space estimates $\mathbb{T}_\mathcal{O}$ with error at most $\sin \theta$.
    Hence, applying \Cref{thm:manifold_reconstruction_deterministic} with parameters $\Delta' = 2\Delta$, $\eta$ and $\sin \theta$ (for which one easily checks that they fulfill its assumptions), we get that the triangulation $\hat{M}$ of \Cref{thm:manifold_reconstruction_deterministic} computed over $\mathcal{O}$ and $\mathbb{T}_\mathcal{O}$ achieves precision
    \begin{align*}
        \dHaus(M,\hat{M})
        &\leq
        \frac{C_d \Delta'^2}{\rch_{\min}}
        \leq
        \max\set{
        \varepsilon
        ,
        \mathbf{C}^3
        \rch_{\min}
        \ttau^\frac{2}{d+1}
        },
    \end{align*}
    which yields the announced result since $\mathbf{C} \leq (C \vee \tilde{C})^d/\Gamma$.
\end{proof}

\begin{proof}[\proofof \Cref{thm:SQ_upper_bound_ball}]
The proof follows the same lines as that of \Cref{thm:SQ_upper_bound_point}, except for the seed point $\hat{x}_0$ which is not trivially available, and requires extra statistical queries.
More precisely, we let $\hat{x}_0$ be the output point given by the SQ detection algorithm of \Cref{thm:routine_detection_seed_point} applied with precision parameter $\varepsilon/2$.
This point requires no more than $ 6 n \log(6R/\varepsilon)$ statistical queries to $\STAT(\tau)$. Furthermore, adopting the same notation as in the proof of \Cref{thm:SQ_upper_bound_point} we have
\begin{align*}
    \dd(\hat{x}_0,M)
    &\leq
    \max\set{
        \frac{\varepsilon}{2}
        ,
        C^d\Gamma^{\frac{2}{d+1}-1}
        \rch_{\min}
        \left(\frac{\tau}{\omega_d f_{\min} \rch_{\min}^d} \right)^{\frac{2}{d+1}}
    }
    \\
    &\leq
    \rch_{\min} 
    \max \set{
        \frac{1}{\mathbf{C}^2}
        \left(\frac{\Delta}{\rch_{\min}}\right)^2
        ,
        \mathbf{C} \ttau^{\frac{2}{d+1}}
    },
\end{align*}
so that the rest of the proof runs exactly as that of \Cref{thm:SQ_upper_bound_point}, and yields the result.
\end{proof}

%% file: parts/lower-bounds.tex
In spirit, the lower bound techniques developed below are similar to the \emph{statistical dimension}
of~\cite{Feldman17b}, developed for general search problems.
However, when working with manifold models, this tool appears difficult
to handle, due to the singular nature of low-dimensional distributions, yielding non-dominated models.
Indeed, if $D_0$ and $D_1$ are distributions that have supports being 
$d$-dimensional submanifolds $M_0,M_1 \subseteq \R^n$, and that 
$M_0 \neq M_1$, then $D_0$ and $D_1$ cannot be absolutely continuous with respect to one another.
As a result, any lower bound technique involving Kullback-Leibler or 
chi-squared divergences becomes non-informative 
(see for instance \cite{Feldman17b,DKS17}).

Instead, we present techniques that are well-suited for non-dominated models. They apply for SQ estimation in metric spaces $(\metric,\dist)$ (see~\Cref{subsec:SQ-general-setting}), as opposed to the (more general)
setting of \emph{search problems} of~\cite{Feldman17b}.
We decompose these results into two different types of lower bounds:
\begin{itemize}[leftmargin=*]
    \item
        (\Cref{subsec:lower-bound-general-informational})
        The \emph{information-theoretic} ones, yielding a maximal estimation
        precision $\varepsilon = \varepsilon(\tau)$ given a tolerance $\tau$;
    \item 
        (\Cref{subsec:lower-bound-general-computational})
        The \emph{computational} ones, yielding a minimal number of queries 
        $q = q(\varepsilon)$ to achieve a given precision $\varepsilon$.    
\end{itemize}

\subsection{Information-Theoretic Lower Bound for Randomized SQ Algorithms}
\label{subsec:lower-bound-general-informational}
    
The proofs of the informational lower bounds 
\Cref{thm:SQ_lower_bound_point_informational,thm:SQ_lower_bound_ball_informational}
are based on the following \Cref{thm:lecam_SQ}, which is similar to so-called 
\emph{Le Cam's Lemma}~\cite{Yu97}. To introduce this result we define the 
\emph{total variation distance} between probability distributions.
\begin{definition}[Total Variation Distance]
    \label{def:tv}
    Given two probability distributions $D_0$ and $D_1$ over 
    $(\R^n, \mathcal{B}(\R^n))$, the \emph{total variation distance} between them
    is defined by
    \begin{align*}
        \TV(D_0,D_1)
        &=
        \sup_{B \in \mathcal{B}(\R^n)}
        \bigl| D_0(B) - D_1(B) \bigr|
        \\
        &=
        \sup_{\substack{r : \R^n \to [-1,1] \\ \text{measurable}}}
        \frac{1}{2}
        \bigl| \E_{D_0}[r] - \E_{D_1}[r] \bigr|
        .
    \end{align*}
\end{definition}
The second formula above for the total variation suggests how it can measure an
impossibility of estimation with $\STAT(\tau)$ oracles: two distributions that
are close in total variation distance provide a malicious oracle to make them 
--- and their parameter of interest --- indistinguishable using SQ's, .
This lower bound insight is what underlies Le Cam's Lemma~\cite{Yu97} in the sample
model, and it adapts easily to (randomized) SQ's in the following way.

\begin{theorem}[Le Cam's Lemma for Statistical Queries]
\label{thm:lecam_SQ}
    Consider a model $\model$ and a parameter of interest
    $\functional : \model \to \metric$ in the metric space $(\metric,\dist)$.
    Assume that there exist \emph{hypotheses} $D_0, D_1 \in \model$, such that 
    \[
    \TV(D_0, D_1) \le \tau/2
    \text{~and~}
    \delta < \dist\bigl(\functional(D_0), \functional(D_1)\bigr)/2
    .
    \]
    If $\alpha < 1 / 2$, then no $\STAT(\tau)$ randomized SQ algorithm can estimate
    $\functional$ with precision $\varepsilon \le \delta$ and probability of success
    $1 - \alpha$ over $\model$ (no matter how many queries it does).
\end{theorem}
\begin{proof}[\proofof \Cref{thm:lecam_SQ}]
    We prove the contrapositive. For this purpose, assume that a randomized SQ
    algorithm $\Algorithm{A} \sim \RandmizedAlgorithm{A}$ estimates $\theta$ with 
    precision $\varepsilon \le \delta$ and probability at least $1-\alpha$ over 
    $\model$. We will show that $\alpha \ge 1 / 2$.
        
    Consider the oracle which, given a query $r:\R^n \to [-1,1]$ to the distribution 
    $D \in \model$, returns the answer:
    \begin{itemize}
        \item $\answer = \E_{D_0}[r]$ if $D = D_1$;
        \item $\answer = \E_D[r]$ if $D \in \model \setminus \set{D_1}$.
    \end{itemize}
    As for all query $r : \R^n \to [-1, 1]$, 
    $|\E_{D_0}[r] - \E_{D_1}[r]| \le 2 \TV(D_0,D_1) \le \tau$, it is a valid 
    $\STAT(\tau)$ oracle. 
    Furthermore, notice that the answers of this oracle
    are the same for $D = D_0$ and $D = D_1$. Writing 
    $\Algorithm{A} = (r_1, \dots, r_q, \hat{\theta}) \sim \RandmizedAlgorithm{A}$,
    we denote these answers by $\answer_1, \dots, \answer_q$. 
    The $\answer_i$'s are random variables, with randomness driven by the randomness of 
    $\Algorithm{A} \sim \RandmizedAlgorithm{A}$. 
    For $i \in \set{0, 1}$, let us consider the event 
    \[
        B_i = 
        \set{ 
            \rho\left(
                \functional(D_i), \hat{\theta}(\answer_1, \dots, \answer_q)
            \right) \le \varepsilon
        }
        .
    \] 
    The fact that $\RandmizedAlgorithm{A}$ estimates $\functional$ 
    with precision $\varepsilon$ and probability at least $1-\alpha$ over $\model$
    translates into 
    $\Pr_{\Algorithm{A} \sim \RandmizedAlgorithm{A}}(B_i) \ge 1 - \alpha$, for 
    $i\in \set{0, 1}$. But since 
    $\varepsilon \le \delta < \rho(\functional(D_0), \functional(D_1)) / 2$, 
    the events $B_0$ and $B_1$ are disjoint (i.e. $B_0 \cap B_1 = \emptyset$). As a result,
    \begin{align*}
        1 
        \ge 
        \Pr_{\Algorithm{A} \sim \RandmizedAlgorithm{A}}\left(B_0 \cup B_1\right)
        =
        \Pr_{\Algorithm{A} \sim \RandmizedAlgorithm{A}}(B_0) + 
            \Pr_{\Algorithm{A} \sim \RandmizedAlgorithm{A}}(B_1)
        \ge
        2(1 - \alpha),
     \end{align*}
    which yields $\alpha \ge 1/2$ and concludes the proof.
\end{proof}
    
\subsection{Computational Lower Bound}
\label{subsec:lower-bound-general-computational}
This section is dedicated to prove the following 
\Cref{theorem:lower_bound_computational_manifold}, that provides a computational lower
bound for support estimation in Hausdorff distance. It involves the generalized notion
of metric packing, which is defined right below.

\begin{theorem}
\label{theorem:lower_bound_computational_manifold}
    Given a model $\model$ over $\R^n$, any randomized SQ algorithm estimating 
    $M = \supp(D) \subseteq \R^n$ with precision $\varepsilon$ for the Hausdorff distance, 
    and with probability of success at least $1 - \alpha$, must make at least 
    \[
        q
        \geq 
        \frac{
            \log \bigl( (1-\alpha) \PK_{(\parameterspace, \dHaus)}(\varepsilon) \bigr)
        }{
            \log (1+\floor{1/\tau})
        }
    \]
    queries to $\STAT(\tau)$, where $\parameterspace = \set{\supp(D), D \in \model}$.
\end{theorem}

Similarly to \Cref{subsec:lower-bound-general-informational}, we put 
\Cref{theorem:lower_bound_computational_manifold} in the broader context of SQ estimation 
in metric spaces (see \Cref{subsec:SQ-general-setting}), and state the more general 
\Cref{theorem:lower_bound_computational} below. To this aims, and similarly to the Euclidean
case (\Cref{def:packing_covering}), let us recall the definitions of metric packings and
coverings. We let $(\metric,\dist)$ be a metric space, $\mathcal{M} \subseteq \metric$ a subset
of $\Theta$, and a radius $\varepsilon>0$.
\begin{itemize}[leftmargin=*]
    \item 
        An \emph{$\varepsilon$-covering} of $\mathcal{M}$ is a subset 
        $\set{\theta_1, \dots, \theta_k } \subseteq \mathcal{M}$ such that for all 
        $\theta \in \mathcal{M}$, we have
        $\min_{1 \leq i \leq k}\dist(\theta,\theta_i) \leq \varepsilon$. 
        The covering number $\CV_{(\mathcal{M},\dist)}(\varepsilon)$ of $\mathcal{M}$ at scale
        $\varepsilon$ is the smallest cardinality $k$ of such an $\varepsilon$-covering.
    \item 
        An \emph{$\varepsilon$-packing} of $\mathcal{M}$ is a subset 
        $\left\lbrace \theta_1, \dots,\theta_k \right\rbrace \subseteq \mathcal{M}$ such that for 
        all $1 \leq i < j \leq k$, 
        $\Bmetric(\theta_i,\varepsilon) \cap \Bmetric(\theta_j,\varepsilon) = \emptyset$
        (or equivalently $\dist(\theta_i,\theta_j) > 2\varepsilon$), where 
        $\Bmetric(\theta,\varepsilon) = 
            \set{\theta' \in \metric, \dist(\theta,\theta')\leq \varepsilon}$ is the closed 
        ball in $(\metric,\dist)$. The covering number $\PK_{(\mathcal{M},\dist)}(\varepsilon)$
        of $\mathcal{M}$ at scale $\varepsilon$ is the largest cardinality $k$ of such an 
        $\varepsilon$-packing.
\end{itemize}

\begin{theorem}
\label{theorem:lower_bound_computational}
    Given a model $\model$ and a parameter of interest $\functional : \model \to \metric$ in the
    metric space $(\metric,\dist)$, any randomized SQ algorithm estimating $\functional(D)$ over 
    $\model$ with precision $\varepsilon$ and probability of success at least $1 - \alpha$, must
    make at least 
    \[
        q
        \geq 
        \frac{\log \bigl( (1-\alpha) \PK_{(\functional(\model), \dist)}(\varepsilon) \bigr)}{\log (1+\floor{1/\tau})}
    \]
    queries to $\STAT(\tau)$, where $\functional(\model) = \set{\functional(D), D \in \model}$.
\end{theorem}

\begin{proof}[\proofof \Cref{theorem:lower_bound_computational_manifold}]
Apply \Cref{theorem:lower_bound_computational} with parameter of interest $\theta(D) = \supp(D)$ and distance $\dist = \dHaus$.
\end{proof}

\subsubsection{Probabilistic Covering and Packing Number}

To prove \Cref{theorem:lower_bound_computational}, we will use the following notion of probabilistic covering.
Given a set $S$ and an integer $k \geq 0$, we denote by $\binom{S}{\le k}$ the set of all subsets of $S$ of cardinality at
most $k$.

\begin{definition}
\label{def:epsilon-alpha-covering}
    Let $(\metric, \dist)$ be a metric space. We say that a probabilistic 
    measure $\mu$ over $\binom{\metric}{\le d}$ is a probabilistic 
    $(\varepsilon, \alpha)$-covering of $(\metric, \dist)$ by $d$ points if for
    all $\theta \in \metric$,
    \[
        \Pr_{\mathbf{p} \sim \mu}\left[
            \theta \in \bigcup_{q \in \mathbf{p}} \Bmetric(q,\varepsilon)
        \right] \ge 1 - \alpha.
    \]
    We denote by $\CV_{(\metric, \dist)}(\varepsilon, \alpha)$ the minimal 
    $d$ such that there is a probabilistic $(\varepsilon, \alpha)$-covering
    of $(\metric, \dist)$ with $d$ points.
\end{definition}
This clearly generalizes (deterministic) coverings, since 
$\CV_{(\metric, \dist)}(\varepsilon, \alpha = 0)$ coincides with the standard
covering number $\CV_{(\metric, \dist)}(\varepsilon)$. 
However, this quantity might be involved to compute since it involves randomness.
Before proving \Cref{theorem:lower_bound_computational}, let us show how to 
lower bound $\CV_{(\metric, \dist)}(\varepsilon, \alpha)$ in practice.
\begin{theorem}
\label{theorem:approximate-cv-lower-bound}
    Let $(\metric, \dist)$ be a metric space. Assume that there is a 
    probability measure $\nu$ on $\metric$ such that 
    for all $q_1, \dots, q_\ell \in \metric$,
    \begin{equation*}
    \label{equation:pseudo-packing-measure}
        \nu\left(
            \bigcup_{i = 1}^\ell \Bmetric(q_i,\varepsilon)
        \right) < 1 - \alpha.
    \end{equation*}
    Then  $\CV_{(\metric, \dist)}(\varepsilon, \alpha) > \ell$.
\end{theorem}
\begin{proof}[\proofof \Cref{theorem:approximate-cv-lower-bound}]
    Take any probability measure $\mu$ over $\binom{\metric}{\le \ell}$, 
    and consider the map $f(\mathbf{p}, \theta) = 
    \indicator{\cup_{q \in \mathbf{p}}\Bmetric(q,\varepsilon)}(\theta)$ 
    for all $\mathbf{p} \in \binom{\metric}{\le \ell}$ and $\theta \in \metric$.
    By assumption, for all fixed $\mathbf{p} \in \binom{\metric}{\le \ell}$,
    \[
        1 - \alpha 
        >
        \nu
        \left(
            \bigcup_{q \in \mathbf{p}}\Bmetric(q,\varepsilon)
        \right)
        = \int_{\metric}
        f(\mathbf{p}, \theta)
        \nu(d\theta);
    \]
    hence, by integration with respect to $\mu(d \mathbf{p})$ and 
    Fubini--Tonelli,
    \begin{align*}
        1 - \alpha
        &>
        \int_{\binom{\metric}{\le \ell}} \left( \int_{\metric}
        f(\mathbf{p}, \theta)
        \nu(d\theta) \right)
        \mu(d \mathbf{p})
        =
        \int_{\metric} \left(\int_{\binom{\metric}{\le \ell}}
        f(\mathbf{p}, \theta)
        \mu(d \mathbf{p}) \right)
        \nu(d\theta).
    \end{align*}
    As $\nu$ is a probability distribution, this yields the existence of a fixed 
    $\theta = \theta_\mu \in \metric$ such that 
    \[ 
        1 - \alpha
        >
        \int_{\binom{\metric}{\le \ell}} f(\mathbf{p}, \theta) \mu(d\mathbf{p}) 
        =
        \Pr_{\mathbf{p} \sim \mu}\left[
        \theta \in \bigcup_{q \in \mathbf{p}} \Bmetric(q,\varepsilon)
        \right] 
        .
    \]
    In other words, we have shown that no probability distribution $\mu$ over 
    $\binom{\metric}{\le \ell}$ can be an $(\varepsilon, \alpha)$-covering of 
    $(\metric, \rho)$ (\Cref{def:epsilon-alpha-covering}). Hence, 
    $\CV_{(\metric, \dist)}(\varepsilon, \alpha) > \ell$.
\end{proof}
As a byproduct of \Cref{theorem:approximate-cv-lower-bound}, we can now show that probabilistic coverings are closely related to the usual notions of metric covering and packing numbers. 
\begin{theorem}
\label{thm:prob_covering_VS_packing}
    Let $(\metric, \dist)$ be a metric space, and $\alpha < 1$. 
    Then, 
    \[
        \CV_{(\metric, \dist)}(\varepsilon) 
        \ge
        \CV_{(\metric, \dist)}(\varepsilon, \alpha) 
        \ge
        (1-\alpha) \PK_{(\metric, \dist)}(\varepsilon)
        .
    \]
\end{theorem}

\begin{proof}[\proofof \Cref{thm:prob_covering_VS_packing}]
    If any of the three terms is infinite, then all the terms involved clearly are infinite, so that the announced bounds hold.
    Otherwise, any given  $\varepsilon$-covering of $(\metric,\dist)$ is also a $(\varepsilon, \alpha)$-covering (where we identify a finite set to the uniform measure on it), which gives the left-hand bound.
    For the right-hand bound, write $k = \PK_{(\metric, \dist)}(\varepsilon) < \infty$, and let $\{\theta_1, \dots, \theta_k\}$ be an
    $\varepsilon$-packing  of $(\metric, \dist)$. 
    That is, for all $i \neq j$, $\rho(\theta_i, \theta_j) > 2\varepsilon$.
    
    Take $\nu$ to be the uniform probability distribution over this packing, that is set
    $\nu(S) = |\{\theta_1, \dots, \theta_k\} \cap S|/k$ for all $S \subseteq \metric$.
    Note that since $\{\theta_1, \dots, \theta_k\}$ is an $\varepsilon$-packing, we have 
    $\nu\bigl(\Bmetric(\theta,\varepsilon)\bigr) \le 1 / k$ for all $\theta \in \metric$, 
    and as a result,
    \[
        \nu\left(
            \bigcup_{i = 1}^\ell \Bmetric(\theta_i,\varepsilon)
        \right) \le \frac{\ell}{k}
    \]
    for all $\theta_1, \dots, \theta_\ell \in \metric$.
        
    Taking $\ell = \ceil{(1 - \alpha) k} - 1$, \Cref{theorem:approximate-cv-lower-bound}
    implies that $\CV_{(\metric, \dist)}(\varepsilon, \alpha) > \ceil{(1 - \alpha) k} - 1$, and hence 
    \[
        \CV_{(\metric, \dist)}(\varepsilon, \alpha) \ge 
        \ceil{(1 - \alpha) k} 
        \ge 
        (1-\alpha)k 
        = 
        (1 - \alpha) \PK_{(\metric, \dist)}(\varepsilon)
        .
    \]
\end{proof}

\subsubsection{Proof of the Computational Lower Bounds for Randomized SQ Algorithms}
    
We are now in position to prove the lower bounds on (randomized) SQ algorithms in general metric spaces.
\begin{proof}[\proofof \Cref{theorem:lower_bound_computational}]
    For all $i \in \set{0, \dots, \ceil{1 / \tau}}$, write 
    $L_i = \min\set{-1 + (2i + 1) \tau, 1}$. The $L_i$'s form a $\tau$-cover
    of $[-1,1]$, meaning that for all $t \in [-1,1]$, there is a least one 
    $0 \le i \le \floor{1 / \tau}$ with $|L_i - t| \le \tau$.
    Hence we can define $f : [-1, 1] \to [-1, 1]$ by $f(t) = L_{i_0}$, 
    where $L_{i_0}$ is smallest $L_i$ such that $|L_{i}-t| \le \tau$.
    Note that $f$ takes only $\floor{1 / \tau} + 1$ different values, and that
    $|f(t) - t| \le \tau$ for all $t \in [-1,1]$.
        
    Let us now consider the oracle $\oracle$ which, given a query
    $r : \R^n \to [-1, 1]$ to the distribution $D$, returns the answer 
    $\answer_D(r) = f(\E_D [r])$.
    Roughly speaking, the oracle discretizes the 
    segment $[-1, 1]$ into $\floor{1 /\tau}+1$ points and returns the projection
    of the correct mean value $\E_D [r]$ onto this discretization. Clearly, 
    $\oracle$ is a valid $\STAT(\tau)$ oracle since $|f(t) - t| \le \tau$ for all $t \in [-1,1]$.
    
    Let $\RandmizedAlgorithm{A}$ be a randomized SQ algorithm estimating $\functional$ over $\model$, and 
    $\Algorithm{A} = (r_1, \dots,r_q,\hat{\theta}) \sim \RandmizedAlgorithm{A}$.
    Let us write $d = (\floor{1 /\tau}+1)^q$, 
    and consider the random subset of $\metric$ given by
    \[
        C(\Algorithm{A}) = \set{
            \hat{\theta}
            \biggl(
                \answer_D(r_1),
                \dots,
                \answer_D(r_q)
            \biggr)
        }_{D \in \model}
        .
    \]
    Note that by construction of the oracle $\oracle$, $C(\Algorithm{A}) \in \binom{\model}{\le d}$.
    Let us consider the probability distribution $\mu$ over $\binom{\model}{\le d}$ such that the measure of a set $S$ is equal to 
    $\Pr_{\Algorithm{A} \sim \RandmizedAlgorithm{A}}[C(\Algorithm{A}) \in S].$
        
    It is clear that if a deterministic algorithm $\Algorithm{A}_0$ estimates $\theta(D)$ with precision $\varepsilon$ using the oracle $\oracle$, then $\theta(D) \in \cup_{t \in C(\Algorithm{A}_0)} \Bmetric(t,\varepsilon)$.
    As $\RandmizedAlgorithm{A}$ estimates $\functional$ with precision $\varepsilon$ and probability at least $1 - \alpha$ over $\model$, this means that $\mu$ is a probabilistic $(\varepsilon, \alpha)$-covering of 
    $\functional(\model)$ with $(\floor{1 / \tau} + 1)^q$ points (\Cref{def:epsilon-alpha-covering}). As a result, by definition of $\CV_{(\theta(\model), \dist)}(\varepsilon, \alpha)$, we have
    $(\floor{1 / \tau}+1)^q \ge \CV_{(\theta(\model), \dist)}(\varepsilon, \alpha)$. Finally, from \Cref{thm:prob_covering_VS_packing} we have 
    $\CV_{(\theta(\model), \dist)}(\varepsilon, \alpha) \geq (1-\alpha)\PK_{(\theta(\model), \dist)}(\varepsilon)$, which gives the announced result.
\end{proof}

%% file: parts/manifold-lower-bound-constructions.tex
\subsection{Diffeomorphisms and Geometric Model Stability }

The following result will allow us to build different elements of 
$\manifolds{n}{d}{\rch_{\min}}$ in a simple way, by considering 
diffeomorphic smooth perturbations of a base manifold $M_0$.
Here and below, $I_n$ is the identity map of $\R^n$. Given a regular map $\Phi : \R^n \rightarrow \R^n$, $d_x\Phi$ and $d_x^2 \Phi$ stand for its first and second order differentials at $x \in \R^n$.

\begin{proposition}
\label{prop:diffeomorphism_stability}
    Let $M_0 \in  \manifolds{n}{d}{2 \rch_{\min}}$ and 
    $\Phi : \R^n \rightarrow \R^n$ be a proper $\mathcal{C}^2$ map, i.e. 
    $\lim_{\norm{x} \to \infty} \norm{\Phi(x)} = \infty$.
    If $\sup_{x \in \R^n} \normop{I_n - d_x \Phi} \le 1/(10d)$ and 
    $\sup_{x \in \R^n} \normop{d^2_x \Phi} \le 1/\left(4\rch_{\min}\right)$, 
    then $\Phi$ is a global diffeomorphism, and 
    $
        \Phi(M_0) \in \manifolds{n}{d}{\rch_{\min}}
    $.
    Furthermore, $1/2 \le \Haus^d(\Phi(M_0))/\Haus^d(M_0) \le 2$.
\end{proposition}

\begin{proof}[\proofof \Cref{prop:diffeomorphism_stability}]
    As $\sup_x \normop{d_x \Phi - I_n} < 1$, $d_x \Phi$ is invertiblefor all $x \in \R^n$. Hence, the inverse function theorem yields 
    that $\Phi$ is everywhere a local diffeomorphism. As, 
    $\lim_{\norm{x} \rightarrow \infty}\norm{\Phi(x)} = \infty$ this 
    diffeomorphism is global by the Hadamard-Cacciopoli 
    theorem~\cite{DeMarco1994}. In particular, $\Phi(M_0)$ is a compact 
    connected $d$-dimensional submanifold of $\R^n$ without boundary.
    In addition, by Taylor's theorem, $\Phi$ is Lipschitz with constant 
    $\sup_x \normop{d_x \Phi} \le (1 + \sup_x \normop{I_n - d_x \Phi }) 
        \le 11/10$, $\Phi^{-1}$ is Lipschitz with constant 
    $\sup_x\normop{d_x \Phi^{-1}} \le 
        ({1-\sup_x \normop{ I_n - d_x \Phi}})^{-1} \le 10/9$, and 
    $d \Phi$ is Lipschitz with constant 
    $\sup_x \normop{d^2_x \Phi } \le 1/(4\rch_{\min}) \le 1 / (2\rch_{M_0})$. 
    Hence,~\cite[Theorem 4.19]{Federer59} yields
    \begin{align*}
        \rch_{\Phi(M)} 
        &\geq 
        \frac{
            (2\rch_{M_0})(1 - \sup_x\normop{I_n - d_x \Phi })^2
        }{
            \sup_x \normop{d^2_x \Phi } (2\rch_{M}) + 
                (1 + \sup_x \normop{  I_n - d \Phi })
        } 
        \geq
        (2\rch_{M_0})/2
        \geq
        \rch_{\min}
        .
    \end{align*}
    As a result, we have $\Phi(M_0) \in \manifolds{n}{d}{\rch_{\min}}$.
    For the last claim, we use the properties of the Hausdorff measure $\Haus^d$ under Lipschitz maps
   ~\cite[Lemma 6]{Arias13} to get
    \begin{align*}
        \Haus^d(\Phi(M)) 
        &\le 
        \sup_x \normop{d_x \Phi}^d \Haus^d(M)
        \\
        &\le
        \left(1 + 1/(10d) \right)^d
        \Haus^d(M)
        \\
        &\le
        2
        \Haus^d(M),
    \end{align*}
    and symmetrically,
    \begin{align*}
        \Haus^d(M) 
        &\le 
        \sup_x \normop{d_x \Phi^{-1}}^d \Haus^d(\Phi(M))
        \\
        &\le
        \frac{1}{\left(1 - 1/(10d) \right)^d} \Haus^d(\Phi(M))
        \\
        &=
        2
        \Haus^d(\Phi(M))
        ,
    \end{align*}
    which concludes the proof.
\end{proof}

Among the smooth perturbations $\Phi: \R^n \to \R^n$ nearly preserving $\manifolds{n}{d}{\rch_{\min}}$, the following localized bump-like functions will be of particular interest for deriving lower bounds.

\begin{lemma}
    \label{lem:multiple_bump_map_is_nice}
    Let $\delta, \eta >0$ be positive reals.
    Fix $p_1, \dots,p_N \in \R^n$ be such that $\norm{p_i - p_j} > 2\delta$
    for all $i \neq j \in \set{1, \dots,N}$.
    Given a family of unit vectors 
    $\mathbf{w} = (w_i)_{1 \le i \le N} \in \left( \R^n \right)^N$, we let 
    $\Phi_{\mathbf{w}}$ be the function that maps any $x\in \R^n$ to
    \begin{align*}
        \Phi_{\mathbf{w}}(x)
        &=
        x
        +
        \eta
        \left(
            \sum_{i = 1}^N
            \phi\left(\frac{x-p_i}{\delta}\right)
            w_i
        \right)
        ,
    \end{align*}
	where $\phi : \R^n \to \R$ the real-valued bump function defined by
	\begin{align*}
		 \phi: \R^n &\longrightarrow \R
		 \\
		 y &\longmapsto \exp\left(-{\norm{y}^2}/{(1-\norm{y}^2)} \right) \indicator{\B(0,1)}(y)
		 .
	\end{align*}
	Then $\Phi_\mathbf{w}$ is $\mathcal{C}^\infty$ smooth, 
	$\lim_{\norm{x} \to \infty} \norm{\Phi_\mathbf{w}(x)} = \infty$, and
	$\Phi_\mathbf{w}$ satisfies $\sup_{x \in \R^n} \norm{x - \Phi_\mathbf{w}(x)} \le \eta$,
	\begin{align*}
        \sup_{x \in \R^n} \normop{I_n - d_x \Phi_\mathbf{w}} \le \frac{5 \eta}{2 \delta}
        \text{ and }
        \sup_{x \in \R^n} \normop{d^2_x \Phi_\mathbf{w}} \le \frac{23 \eta}{\delta^2}
        .
	\end{align*}
\end{lemma}

\begin{proof}[\proofof \Cref{lem:multiple_bump_map_is_nice}]
    Straightforward calculations show that the real-valued map 
    $\phi: \R^n \longrightarrow \R$ is $\mathcal{C}^\infty$ smooth over $\R^n$, 
    equals to $0$ outside  $\B(0,1)$, and satisfies $0 \le \phi \le 1$,
    $\phi(0) = 1$,
	\begin{align*}
    	\sup_{y \in \B(0,1)}
    	\norm{d_y \phi}
    	\le 5/2
    	\text{  and }
    	\sup_{y \in \B(0,1)}
    	\normop{d_y^2 \phi}
    	\le 23
    	.
	\end{align*}
	By composition and linear combination of $\mathcal{C}^\infty$ smooth functions, 
	$\Phi_\mathbf{w}$ is therefore $\mathcal{C}^\infty$ smooth. Also, $\Phi_\mathbf{w}$ 
	coincides with the identity map outside the compact set $\cup_{i=1}^N \B(p_i,\delta)$. Furthermore, for 
	$i \neq j \in \set{1, \dots,N}$, $\B(p_i, \delta) \cap \B(p_j,\delta) = \emptyset$, 
	since $\norm{p_i-p_j} > 2\delta$. Therefore, if $x \in \B(p_i,\delta)$, we have 
	$\Phi_{\mathbf{w}}(x) = x + \eta \phi\left(\frac{x-p_i}{\delta}\right)w_i$.
	This directly gives $\sup_{x \in \R^n} \norm{x - \Phi_\mathbf{w}(x)} \le \eta$, and 
	by chain rule,
	\begin{align*}
	     \sup_{x \in \R^n} \normop{I_n - d_x \Phi_\mathbf{w}}
	     &=
	     \max_{1 \le i \le N}
	     \sup_{x \in \B(p_i,\delta)} 
	     \eta 
	     \normop{d_x \left( \phi\left( \frac{\cdot - p_i}{\delta} \right) w_i\right)}
	     \\
	     &=
	     \max_{1 \le i \le N}
	     \sup_{y \in \B(0,1)}
	     \normop{w_i (d_y \phi)^\top}
	     \frac{\eta}{\delta}
	     \\
	     &=
	     \sup_{y \in \B(0,1)}
	     \norm{d_y \phi}
	     \frac{\eta}{\delta}
	     \\
	     &\le
	     \frac{5\eta}{2\delta}
	     ,
	\end{align*}
	and
	\begin{align*}
	     \sup_{x \in \R^n} \normop{d_x^2 \Phi_\mathbf{w}}
	     &=
	     \max_{1 \le i \le N}
	     \sup_{y \in \B(0,1)}
	     \normop{w_i d_y^2 \phi}
	     \frac{\eta}{\delta^2}
	     \\
	     &\le
	     \frac{23\eta}{\delta^2}
	     ,
	\end{align*}
	which concludes the proof.
\end{proof}

\subsection{Building a Large-Volume Submanifold with Small Euclidean Diameter}
\label{subsec:building-large-volume-manifold}

The proofs of \Cref{thm:SQ_lower_bound_ball_informational,thm:SQ_lower_bound_ball_computational} will involve the construction of submanifolds $M \subseteq \R^n$ with prescribed and possibly large volume $\Haus^d(M)$. Informally, this will enable us to build hypotheses and packings with large cardinality by local variations of it (see \Cref{prop:hypotheses_for_le_cam_local_variations,prop:packing_local_variations}) under nearly minimal assumptions on $f_{\min}$ (which can be seen as an inverse volume, for uniform distributions).
For the reasons mentioned in \Cref{subsubsec:implicit-bounds-on-parameters}, one easily checks that the volume of $M \in \manifoldsball{n}{d}{\rch_{\min}}{R}$ can neither be too small nor too large, when $\rch_{\min}$ and $R$ are fixed (\Cref{prop:volume_bounds_under_reach_constraint}).
Conversely, this section is devoted to prove the \emph{existence} of submanifolds $M \in \manifoldsball{n}{d}{\rch_{\min}}{R}$ that nearly achieve the minimum and maximum possible such volumes given by \Cref{prop:volume_bounds_under_reach_constraint}.

\subsubsection{The Statement}
Namely, the goal of \Cref{subsec:building-large-volume-manifold} is to prove the following result.

\begin{proposition}
\label{prop:manifold_of_prescribed_volume}
	Assume that $\rch_{\min} \le R/36$. 
	Writing $C_d' = 9(2^{2d+1} \sigma_{d - 1})$, let $\mathcal{V} > 0$ be such that
	\[
    	1
    	\le
    	\frac{\mathcal{V}}{C_d' \rch_{\min}^d}
    	\le
    	\max_{1 \le k \le n}
    	\left(\frac{R}{48 \rch_{\min} \sqrt{k}}\right)^k
    	.
	\]
    Then there exists $M_0 \in \manifolds{n}{d}{\rch_{\min}}$ such that 
    $M_0 \subseteq \B(0,R)$ and
	\begin{align*}
		\mathcal{V}/24
		\le
		\Haus^d(M_0)
		\le
		\mathcal{V}
		.
	\end{align*}
\end{proposition}

Informally, in codimension one (i.e. $D = d+1$), the manifold $M_0$ of \Cref{prop:manifold_of_prescribed_volume} can be though of as the boundary of the offset of a Hilbert curve in $\B(0,R)$ of prescribed length. 
This intuition, however, is only limited to codimension one, and requires extra technical developments for general $d < D$.
        
\begin{proof}[\proofof \Cref{prop:manifold_of_prescribed_volume}]
    Consider the discrete grid $G_0$ in $\R^n$ centered at $0 \in \R^n$, with vertices 
    $\left( 24 \rch_{\min} \Z^n \right) \cap \B(0, R / 2)$,
    and composed of hypercubes of side-length $24 \rch_{\min}$. By considering a 
    $k_0$-dimensional sub-grid parallel to the axes, we see that the grid $G_0$
    contains a square grid $G$ with side cardinality 
    $\kappa = \ceil{\frac{R/2}{24\rch_{\min} \sqrt{k_0}}}$, 
    where 
    $k_0$ belongs to $\argmax_{1 \le k \le n} \left(\frac{R}{48 \rch_{\min} \sqrt{k}}\right)^k$.
    Let us write $\ell = \floor{\mathcal{V}/(C_d'\rch_{\min}^d)}$.
    By assumption on $\mathcal{V}$, $\rch_{\min}$ and $R$, we have
    \begin{align*}
        1
        \le
        \ell
        \le
        \frac{\mathcal{V}}{C_d' \rch_{\min}^d}
    	\le
    	\max_{1 \le k \le n}
    	\left(\frac{R}{48 \rch_{\min} \sqrt{k}}\right)^k
    	\le
    	\kappa^{k_0}
    	.
    \end{align*}
    Hence, \Cref{lem:discrete_hypercube_long_curve} asserts that there exists a 
    connected open simple path $L_n(\ell)$ in $G \subseteq G_0$ with length 
    $| L_n(\ell)| = \ell$. 
    Furthermore, \Cref{lem:manifold_from_path} applied with reach parameter $2\rch_{\min}$ provides us with a closed $d$-dimensional submanifold $M_0'$ of class $C^{1,1}$  such that
    $M_0' = M(L_n(\ell)) \subseteq G^{12\rch_{\min}} \subseteq \B(0,R/2)^{12 \rch_{\min}} 
        \subseteq \B(0,2R/3)$ since $\rch_{\min} \le R/36$, reach $\rch_{M_0'} \geq 2 \rch_{\min}$. Furthermore, writing $C_d = 9(2^d \sigma_{d-1})$ for the constant of \Cref{lem:manifold_from_path}, we also have
    \begin{align*}
        \Haus^d(M_0')
        &\le 
        (C_d  (2 \rch_{\min})^d) |L_n(\ell)|
        \\
        &\le
        (C_d (2 \rch_{\min})^d) \frac{\mathcal{V}}{C_d'\rch_{\min}^d}
        \\
        &\leq
        \mathcal{V}/2
        ,
    \end{align*}
    and
    \begin{align*}
        \Haus^d(M_0')
        &\geq
        (C_d  (2\rch_{\min})^d/3)  |L_n(\ell)|
        \\
        &\geq
        (C_d  (2\rch_{\min})^d/3) \frac{\mathcal{V}}{2C_d'\rch_{\min}^d}
        \\
        &=
        \mathcal{V}/12
        ,
    \end{align*}
    where we used that $\floor{t} \geq t/2$ for all $t \geq 1$
    To conclude the proof, we use the density of $\mathcal{C}^{2}$ submanifolds in the space of $\mathcal{C}^{1,1}$ submanifolds to obtain a closed $d$-dimensional submanifold $M_0$ of class $\mathcal{C}^2$ such that $\rch_{M_0} \geq \rch_{M_0'}/2 \geq \rch_{\min}$, $\dHaus(M_0,M_0') \leq \rch_{\min}$ (and hence $M_0 \subset \B(0,2R/3+\rch_{\min}) \subset \B(0,R)$), and $1/2 \leq \Haus^d(M_0)/\Haus^d(M_0') \leq 2$ (and hence $\mathcal{V}/24 \leq \Haus^d(M_0) \leq \mathcal{V}$).
\end{proof}

\subsubsection{Widget Gluing: From Paths on the Discrete Grid to Manifolds}
\label{subsec:widget}
        
\begin{lemma}
\label{lem:widget_construction}
    Given $\rch_{\min} > 0$ and $d\geq 1$, there exist four 
    $d$-dimensional  $\mathcal{C}^{1,1}$-submanifolds with boundary:
    \begin{center}        
    $M_{E},M_{S},M_{TB} \subseteq [-6 \rch_{\min} , 6 \rch_{\min}]^{d + 1}$
    and
    $M_{NB} \subseteq [-6 \rch_{\min} , 6 \rch_{\min}]^{d+2}$, 
    \end{center}
    called 
    respectively \emph{end}, \emph{straight}, \emph{tangent bend} and
    \emph{normal bend} widgets (see \Cref{fig:reach_boundary_comparison}),
    that:
    \begin{itemize}[leftmargin=*]
        \item are smooth:
            $\rch_{M_E},\rch_{M_S},\rch_{M_{TB}}, \rch_{M_{NB}} \geq \rch_{\min}$;
        \item
            have the following topologies:
            \begin{itemize}[leftmargin=*]
                \item $M_E$ is isotopic to a $d$-ball $\B_d(0,1)$,
                \item $M_S$, $M_{TB}$ and $M_{NB}$ are isotopic to a $d$-cylinder $\Sphere^{d-1} \times [0,1]$;
            \end{itemize}
        \item 
            are linkable: writing $s = 6 \rch_{\min}$, we have
            \begin{itemize}[leftmargin=*]
                \item
                For the tip widget $M_E$:
                    \begin{itemize}[leftmargin=*]
                        \item 
                        $
						M_E \cap \left([-s/2 ,s/2]^{d + 1}                        \right)^c
                        =
                        M_E \cap \left([s/2 ,s] \times \R^d \right) 
                        = 
                        [s/2 , s] \times \Sphere^{d - 1}(0,s/3)
                        .$
                        
                    \end{itemize}
            \item 
                For the straight widget $M_S$:
                \begin{itemize}[leftmargin=*]
                	\item
                	$M_S \cap \left([-s/2 ,s/2]^{d + 1}                        \right)^c 
                	= 
                	M_S 
                	\cap 
                	\left( 
                		\left([-s , - s/2] \times \R^d\right) 
	                	\cup
    	            	\left([s/2 ,s] \times \R^d \right)
        	        \right),
        	        $
                    \item
                    $M_S \cap \left([-s , - s/2] \times \R^d\right) =[-s , -s/2] \times \Sphere^{d - 1}(0,s/3)$,
                    \item 
                    $M_S \cap \left([s/2 ,s] \times \R^d \right)= [s/2 , s] \times \Sphere^{d - 1}(0,s/3)
                    .
                    $
                \end{itemize}
                \item 
                For the tangent bend widget $M_{TB}$:
                \begin{itemize}[leftmargin=*]
                	\item
                	$M_{TB} \cap \left([-s/2 ,s/2]^{d + 1}                        \right)^c 
                	= 
                	M_{TB} 
                	\cap 
                	\left( 
                		\left( [-s , - s/2] \times \R^d \right)
	                	\cup
    	            	\left( \R^d \times [-s , - s/2] \right)
        	        \right),
        	        $
                    \item
                    $M_{TB} \cap \left( [-s , - s/2] \times \R^d \right) = [-s , -s/2] \times \Sphere^{d - 1}(0,s/3)$,
                    \item
                    $M_{TB} \cap \left( \R^d \times [-s , - s/2] \right) = \Sphere^{d - 1}(0,s/3) \times [-s , - s/2]
                    .
                    $
                \end{itemize}
                \item 
                For the normal bend widget $M_{NB}$:
                \begin{itemize}[leftmargin=*]
                	\item
                	$M_{NB} \cap \left([-s/2 ,s/2]^{d+2}\right)^c 
                	= 
                	M_{NB} 
                	\cap 
                	\left( 
                		\left( [-s , - s/2] \times \R^{d} \times \set{0} \right)
	                	\cup
    	            	\left( \set{0} \times \R^{d} \times [-s , - s/2] \right)
        	        \right),
        	        $
                \item
                    $M_{NB} \cap \left( [-s , - s/2] \times \R^{d} \times \set{0} \right) =
                        [-s , -s/2] \times \Sphere^{d - 1}(0,s/3) \times \set{0}$,
                \item
                    $
                        M_{NB} \cap \left( \set{0} \times \R^{d} \times [-s , - s/2] \right) = 
                        \set{0} \times \Sphere^{d - 1}(0,s/3) \times [-s , - s/2]
                    .
                    $
                \end{itemize}
            \end{itemize}
            Furthermore, 
            \[
                (C_d/3) \rch_{\min}^d 
                \le
                \Haus^d(M_E),\Haus^d(M_S),\Haus^d(M_{TB}),\Haus^d(M_{NB})
                \le
                C_d \rch_{\min}^d
                ,
            \]
            where $C_d = 9 (2^d \sigma_{d - 1})$ depends only on $d$.
        \end{itemize}
                
\end{lemma}

\begin{figure}[ht]
    \centering
	\subfloat[
		End widget $M_E$.
	]{
	    \includegraphics[width=0.3\textwidth, page = 1]{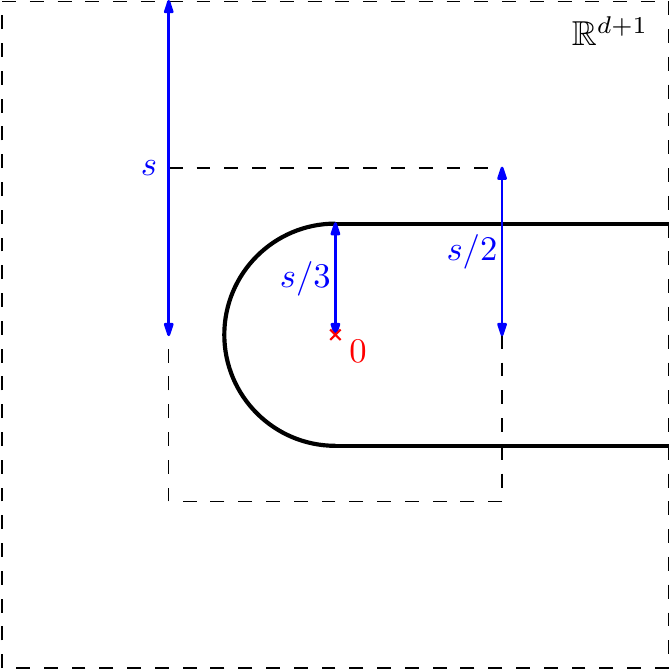}	\label{subfig:sphere_widget_end}
	}
	\subfloat[
		Straight widget $M_S$.
	]{
	    \includegraphics[width=0.3\textwidth, page = 2]{offset_widget_tangent.pdf}		
		\label{subfig:sphere_widget_straight}
	}
	\subfloat[
	    {
	    	Tangent Bend widget $M_{TB}$.	
	    }	
	]{
	    \includegraphics[width=0.3\textwidth, page = 3]{offset_widget_tangent.pdf}	\label{subfig:sphere_widget_tangent_bend}
	}

	\subfloat[
	    {
	        Normal Bend widget $M_{NB}$.	
	    }	
	]{
	    \includegraphics[width=0.5\textwidth]{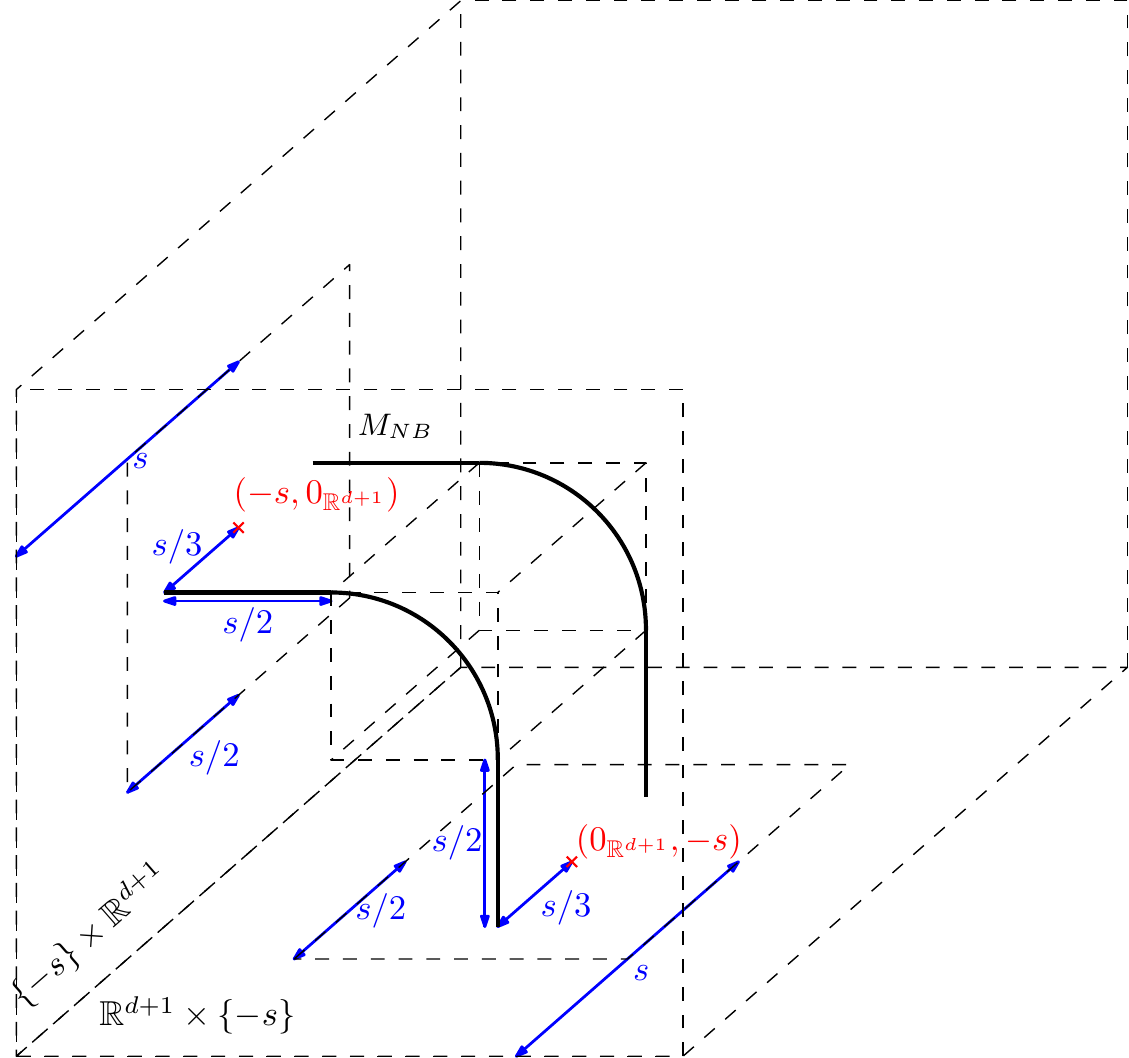}	\label{subfig:sphere_widget_normal_bend}
	}
	\caption{
		The widgets built in \Cref{lem:widget_construction} and used in the proof of \Cref{lem:manifold_from_path}.
	}
	\label{fig:reach_boundary_comparison}
\end{figure}
    
\begin{proof}[\proofof \Cref{lem:widget_construction}]
        First notice that by homogeneity, we can carry out the construction in the unit hypercubes $[-1,1]^{d + 1}$ (respectively $[-1,1]^{d+2}$) and conclude by applying an homothetic transformation.
        Indeed, for all closed set $K \subseteq \R^{n}$ and $\lambda \geq 0$, $\rch_{\lambda K} = \lambda \rch_K$ and $\Haus^d(\lambda K) = \lambda^d \Haus^d(K)$.
        \begin{itemize}[leftmargin=*]

            \item 
            End widget:
            the idea is to glue in a $\mathcal{C}^2$ way a half $d$-sphere with a $d$-cylinder.
            Namely, let us consider
            \begin{align*}
                M_E^{(0)}
                &=
                \left( 
                    \Sphere^{d}(0,1/3) 
                    \cap 
                    \left( [-1,0]\times [-1,1]^d \right) 
                \right)
                \cup
                \left( 
                     [0,1] \times \Sphere^{d - 1}(0,1/3)
                \right)
                .
            \end{align*}
            Elementary calculations yield the intersections 
            $$
            M_E^{(0)} \cap \left([-1/2 ,1/2]^{d + 1}\right)^c = M_E^{(0)} \cap \left([1/2 ,1] \times \R^d \right) = [1/2 , 1] \times \Sphere^{d - 1}(0,1/3)
            .
            $$
            In addition, its medial axis is 
            $\Med(M_E^{(0)}) = [0,1]\times \set{0}^d$, so that 
            $$\rch_{M_E^{(0)}} = \inf_{z \in \Med(M_E^{(0)})} \dd(z,M_E^{(0)}) = 1/3.
            $$
            Finally, $M_E^{(0)}$ is isotopic to the half $d$-sphere 
            $
                \Sphere^{d}(0,1/3) 
                \cap 
                \left( [-1,0]\times [-1,1]^d \right)
            $,
            or equivalently to a $d$-ball.
            
            \item 
            Straight widget: a simple $d$-cylinder satisfies our requirements.
            Similarly as above, the set 
            \begin{align*}
                M_S^{(0)}
                &=
                [-1,1] \times \Sphere^{d - 1}(0,1/3)
            \end{align*}
            clearly is (isotopic to) a $d$-cylinder, has reach $\rch_{M_S^{(0)}} = 1/3$, and all the announced intersection properties with $s=1$.
            
            \item 
            Tangent Bend widget:
            we will glue two orthogonal straight $d$-cylinders via a smoothly rotating $(d - 1)$-sphere.
            More precisely, consider the $d$-cylinders $C_1 = \Sphere^{d - 1}(0,1/3) \times [-1,-1/2]$ and $C_2 = [-1,-1/2] \times \Sphere^{d - 1}(0,1/3)$.
            We will connect smoothly their tips, which are the $(d - 1)$-spheres $S_1 = \Sphere^{d - 1}(0,1/3) \times \set{-1/2} \subseteq C_1$ and $S_2 = \set{-1/2} \times \Sphere^{d - 1}(0,1/3) \subseteq C_2$ of same radius.
            To this aim, take the trajectory of $S_1$ via the affine rotations of center $x_c = (-1/2,0_{\R^{d - 1}},-1/2)$ and linear parts
            \begin{align*}
                R_\theta
                =
                \begin{pmatrix}
                \cos \theta & 0 & \cdots & 0 & -\sin \theta
                \\
                0 & 1 & \cdots & 0 & 0
                \\
                \vdots &&\ddots&& \vdots
                \\
                0 & 0 & \cdots & 1 & 0
                \\
                \sin \theta & 0 & \cdots & 0 & \cos \theta            \end{pmatrix}
                \in \R^{(d + 1)\times(d + 1)}
                ,
            \end{align*}
            when $\theta$ varies in $[0,\pi/2]$. 
            Hence, letting $f_\theta(x) = x_c + R_\theta (x-x_c)$, we have $f_0(S_1) = S_1$, $f_{\pi/2}(S_1) = S_2$. In addition, for all $\theta \in [0,\pi/2]$ and $x \in [-1/2,1/2]^{d} \times \set{-1/2}$, we have $f_\theta(x) \in [-1/2,1/2]^{d + 1}$.
            Hence, letting
            \begin{align*}
                M_{TB}^{(0)}
                =
                C_1 \cup \biggl( \bigcup_{0 \le \theta \le \pi/2} f_\theta (S_1) \biggr) \cup C_2
                ,
            \end{align*}
            we directly get that $M_{TB}^{(0)}$ is isotopic to a $d$-cylinder, and that it satisfies all the announced intersection properties with $s=1$.
            To conclude, by symmetry, the medial axis of this widget writes as
            \begin{align*}
                \Med(M_{TB}^{(0)})
                &=
                \set{0}^{d}\times [-1,-1/2]
                \cup
                \biggl(
                x_c + \bigcup_{t \geq 0}
                (-t,0_{\R^{d-1}},-t)
                \biggr)
                \cup
                [-1,-1/2]\times \set{0}^{d}
                ,
            \end{align*}
            so that straightforward calculations yield $\rch_{M_{TB}^{(0)}} = \min \set{1/3, \dd(x_c,M_{TB}^{(0)}) } = 1/6$.

            \item 
            Normal Bend widget:
            same as for the tangent bend widget, we glue the two orthogonal straight $d$-cylinders $C_1 = \set{0} \times \Sphere^{d-1}(0,1/3) \times [-1,-1/2]$ and $C_2 = [-1,-1/2] \times \Sphere^{d-1}(0,1/3) \times \set{0}$.
            via their respective tips, $S_1 = \set{0} \times \Sphere^{d-1}(0,1/3) \times \set{-1/2} \subseteq C_1$ and $S_2 = \set{-1/2} \times \Sphere^{d-1}(0,1/3) \set{0} \subseteq C_2$.
            To this aim, take trajectory of $S_1$ via the affine rotation of center $x_c = (-1/2,0_{\R^{d}},-1/2)$ and linear parts $R_\theta \in \R^{(d+2)\times(d+2)}$
            for $\theta \in [0,\pi/2]$. 
            As before, letting $f_\theta(x) = x_c + R_\theta (x-x_c)$, we have $f_0(S_1) = S_1$, $f_{\pi/2}(S_1) = S_2$. Also, for all $\theta \in [0,\pi/2]$ and $x \in \set{0} \times [-1/2,1/2]^{d} \times \set{-1/2}$, we have $f_\theta(x) \in [-1/2,1/2]^{d + 1}$.
            Hence, letting
            \begin{align*}
                M_{NB}^{(0)}
                =
                C_1 \cup \biggl( \bigcup_{0 \le \theta \le \pi/2} f_\theta (S_1) \biggr) \cup C_2
                ,
            \end{align*}
            we get the announced results with $s=1$, and in a similar way as above, $\rch_{M_{NB}^{(0)}} =  \min\set{1/3,1/2} = 1/3$.
        \end{itemize}
        Also one easily checks in all the four above cases that
        \[
        \frac{\sigma_{d-1}/3^{d-2}}{3}
        \le
        \Haus^d(M_E^{(0)}),\Haus^d(M_S^{(0)}),\Haus^d(M_{TB}^{(0)}),\Haus^d(M_{NB}^{(0)}) 
        \le 
        \sigma_{d-1}/3^{d-2}
        .
        \]
        Finally, letting
        \[
            C^{-1} 
            = 
            \min
            \set{
                \rch_{M_E^{(0)}}
                ,
                \rch_{M_S^{(0)}}
                ,
                \rch_{M_{TB}^{(0)}}
                ,
                \rch_{M_{NB}^{(0)}}
            }
            =
            1/6
        \]
        and considering the dilations $M_E = (C \rch_{\min}) M_E^{(0)}$, $M_S = (C \rch_{\min}) M_S^{(0)}$, $M_{TB} = (C \rch_{\min}) M_{TB}^{(0)}$ and $M_{NB} = (C \rch_{\min}) M_{NB}^{(0)}$ yields the result by homogeneity, with $C_d = C^d \sigma_{d-1}/3^{d-2} = 9 (2^d \sigma_{d-1})$.
    \end{proof}

\begin{lemma}
    \label{lem:manifold_from_path}
    Let $G$ be a discrete grid in $\R^n$ composed of hypercubes of side-length
    $12 \rch_{\min}$. Then any connected open simple path $L$ in $G$
    (see \Cref{lem:discrete_hypercube_long_curve}) defines a $\mathcal{C}^{1,1}$ $d$-dimensional closed submanifold,
    denoted by $M(L)$, such that:
    \begin{itemize}
        \item 
        $M(L) \subseteq G^{6\rch_{\min}}$;
        \item 
        $M(L) \in \manifolds{n}{d}{\rch_{\min}}$;
        \item
        $C_d/3 \le \dfrac{\Haus^d(M(L))}{|L| \rch_{\min}^d} \le C_d$,
        where $C_d$ is the constant of \Cref{lem:widget_construction};
        \item
        If $L$ and $L'$ are two different such paths in $G$,
        \begin{align*}
            \dHaus(M(L),M(L'))
            > 2 \rch_{\min}
            .
        \end{align*}
     \end{itemize}
\end{lemma}

\begin{remark}
The construction of \Cref{lem:manifold_from_path} shows that, given one discrete path $L$, one could actually define several different manifolds $M(L)$ with the same properties. We will not exploit this fact as the construction is enough for our purpose.
\end{remark}

\begin{proof}[\proofof \Cref{lem:manifold_from_path}]
	For short, we let $s = 6 \rch_{\min}$. Let $L$ be a fixed connected open simple path on $G$. If $|L| = 1$, take $M(L)$ to be a $d$-sphere of radius $2\rch_{\min}$ centered at the only vertex of $L$. 
	Assuming now that $|L|\geq 2$, we will build $M(L)$ iteratively by adding appropriate widgets of \Cref{lem:widget_construction} along the consecutive vertices that $L$ goes through.
	We pick one of the two degree $1$ vertices (endpoints) of $L$ arbitrarily, and denote the consecutive vertices of $L$ as $x_1, \dots,x_{[L|-1}$.
	\begin{enumerate}[label=(\roman*),leftmargin=*]
	    \item 
	    The path $L$ has exactly one edge at $x_0$, called $v^+_0$, which is parallel to the axes of $\R^n$ since $G$ is the square grid.
	    In the cube $x_0 + [- 6\rch_{\min},6\rch_{\min}]^n$, we define $M(L)$ to coincide with the End widget $M_E \times \set{0}^{n-(d + 1)}$, rotated in the $(e_1,v^+_0)$ plane so that $-e_1$ is sent on $v^+_0$.
	    In this first cube, $M(L)$ hence presents a $d$-cylinder, obtained by a rotation of $[-s,-s/2]\times \Sphere^{d-1}(0,s/3) \times \set{0}^{n-(d + 1)}$ around $x_0$, and pointing towards $v^+_0$. Let us call this cylinder $C^+_0$.
	    \item 
	    Assume now that we have visited the consecutive vertices $x_0, \dots,x_{k-1}$ of $L$, for some $k\geq 1$, and that in the cube around $x_{k-1}$, $M(L)$ presents a cylinder $C^+_{k-1}$ in the direction $v^+_{k-1}$.
	    If $x_k$ is not the other endpoint of $L$, there are exactly two edges at $x_k$, represented by the axis-parallel vectors $v^-_k = (x_{k-1}-x_k) = -v^+_{k-1}$ and $v^+_k = (x_{k + 1}-x_k)$.
	    There are three possible cases depending on the turn that $L$ takes at $x_k$:
	    \begin{enumerate}
	        \item
	        If $v^-_k$ and $v^+_k$ are aligned, take $M(L) \cap \left(x_k + [-s,s]^n \right)$ to coincide with the Straight widget $M_S \times \set{0}^{n-(d + 1)}$, rotated in the $\{e_1,v^+_k\}$-plane so that $e_1$ is sent on $v^+_k$.
	        \item
	        If $v^+_k$ belongs to the $(d + 1)$-plane spanned by $C^+_{k-1}$ but $v^-_k$ and $v^+_k$ are not aligned, proceed similarly by rotating the Tangent Bend widget $M_{TB} \times \set{0}^{n-(d + 1)}$ so that $(e_1,e_{d + 1})$ is sent on $(-v^-_{k-1},v^+_k)$.
	        \item
	        Otherwise, if $v^+_k$ does not belong to the $(d + 1)$-plane spanned by $C^+_{k-1}$, then $\{v^+_k, C^+_{k-1}\}$ defines a $(d+2)$-plane. Hence, we proceed similarly by rotating the Normal Bend widget $M_{NB} \times \set{0}^{n-(d+2)}$ so that $(e_1,e_{d+2})$ is sent on $(-v^-_{k-1},v^+_k)$.
	        Note that this case can only occur if $n \geq d+2$.
	    \end{enumerate}
	    
	    \item
	    If we reached the other endpoint of $L$ ($k=|L|-1$), add a rotated End widget oriented in the direction of $C^+_{k-1}$.
	    \end{enumerate}
    Now that the construction of $M(L)$ has been carried out, let us move to its claimed properties.
	\begin{itemize}[leftmargin=*]
	\item
	By construction of the widgets and the fact that all of them are centered at points of the grid $G$, $M(L)$ is included in the offset of $G$ of radius $6\rch_{\min}$.
	\item	
	By induction on the length of the path, it is clear that the union of the straight and bend widgets (without the ends) is isotopic to a cylinder $\Sphere^{d-1}(0,1) \times [0,1]$. As a result, adding the two end widgets at the endpoints of the path yields that $M(L)$ is isotopic to a $d$-dimensional sphere $\Sphere^{d}(0,1)$. It is also clear that $M(L)$ connected, by connectedness of $L$. In particular, $M(L)$ is a compact connected $d$-dimensional submanifold of $\R^n$ without boundary. 
	
	What remains to be proved is that $\rch_{M(L)} \geq \rch_{\min}$.
	To see this, notice that by construction, the widgets connect smoothly through sections of facing straight cylinders $C^\pm = \Sphere^{d-1}(0,s/3) \times [0,\pm s/2] \times \set{0}^{n-(d + 1)}$ (rotated), which are included in the boxes $[-s/2,s/2]^n$ centered a the midpoints of the grid. Apart from these connected ingoing and outgoing cylinders, the widgets are included in boxes $[-s/2,s/2]^n$, which are separated by a distance $s$.
	Hence, if two points $x,y \in M(L)$ are such that $\norm{y-x} \le s/2$, then they must belong to either the same widget or the same connecting cylinder $C^- \cup C^+$.
	As a result, from~\cite[Theorem 4.18]{Federer59} and the fact that $\dd(y-x,T_x M(L)) \le \norm{y-x}$ for all $x \in M(L)$, we get
	\begin{align*}
	    \rch_{M(L)}
	    &=
	    \inf_{x \neq y \in M(L)}
	    \frac{\norm{y-x}^2}{2 \dd(y-x,T_x M(L))}
	    \\
	    &=
	    \min
	    \set{
    	    \inf_{\substack{x, y \in M(L) \\ \norm{y-x} \geq s/2}}
	        \frac{\norm{y-x}^2}{2 \dd(y-x,T_x M(L))}
	        ,
	        \inf_{\substack{x \neq y \in M(L) \\ \norm{y-x} \le s/2}}
	        \frac{\norm{y-x}^2}{2 \dd(y-x,T_x M(L))}
	    }
	    \\
	    &\geq
	    \min
	    \set{
	        s/4
	        ,
            \min \set{
	            \rch_{M_E},\rch_{M_S},\rch_{M_{TB}},\rch_{M_{NB}}
            }
	    }
	    \\
	    &\geq
	    \min\set{6\rch_{\min}/4,\rch_{\min}}
	    \\
	    &=
	    \rch_{\min}
	    ,
	\end{align*}
    which ends proving that $M(L) \in \manifolds{n}{d}{\rch_{\min}}$.
	\item
	As $M(L)$ is the union of $|L|$ of the widgets defined in \Cref{lem:widget_construction}, it follows
	\begin{align*}
		\Haus^d(M(L))
		&\le
		|L|
		\max \set{ \Haus^d(M_E),\Haus^d(M_S),\Haus^d(M_{TB}),\Haus^d(M_{NB}) } 
		\\
		&\le 
		|L| C_d \rch_{\min}^d
		,
	\end{align*}
	and similarly, as the intersection of the consecutive widgets (i.e. $(d-1)$-spheres) is $\Haus^d$-negligible, we have
	\begin{align*}
		\Haus^d(M(L))
		&\geq
		|L|
		\min \set{ \Haus^d(M_E),\Haus^d(M_S),\Haus^d(M_{TB}),\Haus^d(M_{NB}) } 
		\\
		&\geq 
		|L| (C_d/3) \rch_{\min}^d
		.
	\end{align*}
	\item
	Let us now fix two different connected open simple paths $L$ and $L'$ in $G$.
	Since $L \neq L'$, $L$ passes through a vertex, say $x_0 \in \R^n$, where $L'$ doesn't.
	Regardless of the widget used at $x_0$ to build $M(L)$, this widget contains, up to rotation centered at $x_0$, the set $x_0 + \set{-s/2} \times \Sphere^{d - 1}(0,s/3) \times \set{0}^{n-(d + 1)}$. As a result, $\dd(x_0, M(L)) \le \sqrt{(s/2)^2+(s/3)^2}$.
	On the other hand, $M(L')$ does not intersect the cube $x_0 + [-s, s]^n$, so $\dd(x_0, M(L')) \geq s$.	
	Finally, we get
	\begin{align*}
	\dHaus(M(L),M(L'))
	&=
	\sup_{x \in \R^n} \left| \dd(x,M(L')) - \dd(x,M(L))  \right|
	\\
	&\geq
	\left| \dd(x_0,M(L')) - \dd(x_0,M(L))  \right|
	\\
	&\geq
	s - \sqrt{(s/2)^2+(s/3)^2}
	\\
	&=
	6 (1-\sqrt{13}/6) \rch_{\min}
	\\
	&>
	2 \rch_{\min}
	,
	\end{align*}
	which concludes the proof.
	\end{itemize}		
\end{proof}

\subsubsection{Existence of Long Paths on the Grid}
\label{subsec:long-path}

In order to complete the construction of \Cref{prop:manifold_of_prescribed_volume}, we need the existence of paths of prescribed length over the $n$-dimensional discrete grid. Although standard, we include this construction for sake of completeness.
\begin{lemma}
    \label{lem:discrete_hypercube_long_curve}
    Let $\kappa\geq 1$ be an integer and consider the square grid graph
    $G_n$ on $\set{1, \dots, \kappa}^n$. Then for all 
    $\ell \in \set{1, \dots, \kappa^n}$, there exists a connected open simple path
    $L_n(\ell)$ of length $\ell$ in $G_n$. That is, $L_n(\ell)$ is a subgraph of 
    $G_n$ such that:
    \begin{itemize}
        \item $L_n(\ell)$ is connected;
        \item $L_n(\ell)$ has vertex cardinality $\ell$;
        \item if $\ell \geq 2$, $L_n(\ell)$ has maximum degree $2$, and exactly two
            vertices with degree $1$.
    \end{itemize}
\end{lemma}
    
\begin{proof}[\proofof \Cref{lem:discrete_hypercube_long_curve}]
    For $\kappa=1$, $G_n$ consists of a single point, so that the result is trivial.
    We hence assume that $\kappa\geq 2$. Let us first build the paths
    $L_n = L_n(\kappa^n)$ by induction on $n$. For $n = 1$, simply take $L_1$ to be the
    full graph $G_n$. We orientate $L_1$ by enumerating its adjacent vertices in order: 
    $L^\rightarrow_1 [i] = i$ for all $1 \le i \le \kappa$. Given an orientation 
    $L^\rightarrow$ of some path $L$ in $G_n$, we also let
    $L^\leftarrow [i] = L^\rightarrow [|L| - i]$ denote its backwards orientation. Now, 
    assume that we have built $L_n$ for some $n \geq 1$, together with an orientation 
    $L^\rightarrow_n$. To describe $L_{n + 1}$, we list an orientation $L_{n + 1}^\rightarrow$
    of it: an edge of $G_n$ hence belongs to $L_n$ if an only if it joins two consecutive
    vertices in $L_{n + 1}^\rightarrow$. Namely, for $1 \le i \le \kappa^n$, we let
    \begin{align*}
        L_{n + 1}^\rightarrow[i]
        &=
        \left(L_{n}^\rightarrow[i], 1\right)
        \\
        L_{n + 1}^\rightarrow[\kappa^n + i]
        &=
        \left(L_{n}^\leftarrow[i], 2\right)
        \\
        \vdots
        \\
        L_{n + 1}^\rightarrow[(\kappa - 1)\kappa^n + i]
        &=
        \left(L_{n}^\leftrightarrow[i], \kappa\right),
    \end{align*}
    where for the last line, $\leftrightarrow$ stands for $\rightarrow$ if $\kappa$ is odd,
    and $\leftarrow$ otherwise. $L_{n + 1}$ clearly is connected and visits all the vertices
    $\set{1, \dots, \kappa}^n$.
    Its edges all have degree two, except $\left(L_{n}^\rightarrow[1], 1\right)$ and $\left(L_{n}^\leftrightarrow[\kappa^n], \kappa\right)$ which have degree $1$, which concludes the construction of $L_n = L_n(\kappa^n)$.
        To conclude the proof, take $L_n(\ell)$ ($1 \le \ell \le \kappa^n$) to be the first $\ell$ consecutive vertices of $L_n^\rightarrow(\kappa^n)$.
    \end{proof}

\subsection{Informational Lower Bounds: Hypotheses for Le Cam's Lemma}

This section is devoted to prove the two informational lower bounds 
\Cref{thm:SQ_lower_bound_point_informational,thm:SQ_lower_bound_ball_informational}.
We will use the general informational lower bound from \Cref{thm:lecam_SQ} in the models $\distributionspoint{n}{d}{\rch_{\min}}{f_{\min}}{f_{\max}}{L}{0}$ and $\distributionsball{n}{d}{\rch_{\min}}{f_{\min}}{f_{\max}}{L}{R}$ respectively, and parameter of interest $\theta(D) = \supp(D)$ that lies in the metric space formed by the non-empty compact sets of $\R^n$ equipped with the metric $\dist = \dHaus$.

\subsubsection{Construction of the Hypotheses}
    
First, we show how to build hypotheses, i.e probability distributions for Le Cam's Lemma (\Cref{thm:lecam_SQ}). We present a generic construction in the manifold setting by perturbing a base
submanifold $M_0$. Note that the larger the volume $\Haus^d(M_0)$, the stronger 
the result. See also \Cref{prop:packing_local_variations} for a result similar 
in spirit, and used to derive computational lower bounds instead of informational ones.
    
\begin{proposition}
\label{prop:hypotheses_for_le_cam_local_variations}
	For all $M_0 \in \manifolds{n}{d}{2\rch_{\min}}$, $x_0 \in M_0$ and $\tau \le 1$, 
	there exists a manifold $M_1 \in \manifolds{n}{d}{\rch_{\min}}$ such that 
    $x_0 \in M_1$, $\Haus^d(M_0)/2 \le \Haus^d(M_1) \le 2\Haus^d(M_0)$,
    \[
        \frac{\rch_{\min}}{2^{18}}
        \min\set{
            \frac{1}{2^{22}d^2}
            ,
            \left(
            \frac{\Haus^d(M_0) \tau}{\omega_d \rch_{\min}^d}
            \right)^{2/d}
        }
        \le
        \dHaus(M_0,M_1) 
        \le 
        \rch_{\min}/10
        ,
    \]
    and so that the uniform distributions $D_0,D_1$ over $M_0,M_1$ satisfy
    $
        \TV(D_0, D_1)
        \le
        \tau / 2
        .
    $
\end{proposition}
    
\begin{proof}[\proofof \Cref{prop:hypotheses_for_le_cam_local_variations}]
    Let $p_0 \in M_0$ be an arbitrary point such that $\norm{p_0-x_0} \geq \rch_{\min}$. For instance, by taking the geodesic variation
    $p_0 = \gamma_{x_0,v_0}(2\rch_{\min})$, where $v_0 \in T_{x_0} M_0$ is a unit tangent vector, a Taylor expansion of $\gamma_{x_0,v_0}$ and \Cref{lem:geodesic_comparison} yields
    \begin{align*}        
        \norm{p_0-x_0}
        &\geq
        \norm{2\rch_{\min} v_0}
        -
        \norm{\gamma_{x_0,v_0}(\rch_{\min}) - (x_0 + 2\rch_{\min} v_0)}
        \\
        &\geq
        2\rch_{\min} - (2\rch_{\min})^2/(2 \rch_{M_0})
        \\
        &\geq
        \rch_{\min}
        ,
    \end{align*}
    since $\rch_{M_0}\geq 2 \rch_{\min}$.
    Let us denote by $w_0 \in \left(T_{p_0} M_0\right)^\perp$ a unit normal vector of $M_0$ at $p_0$.
    For $\delta,\eta>0$ to be chosen later, let $\Phi_{w_0}$ be the function that maps any $x\in \R^n$ to
    \begin{align*}
        \Phi_{w_0}(x)
        &=
        x
        +
        \eta
            \phi\left(\frac{x-p_0}{\delta}\right)
            w_0
        ,
    \end{align*}
	where $\phi : \R^n \to \R$ is the real bump function $\phi(y) = \exp\left(-{\norm{y}^2}/{(1-\norm{y}^2)} \right) \indicator{\B(0,1)}(y)$ of \Cref{lem:multiple_bump_map_is_nice}.
	We let $M_1 = \Phi_{w_0}(M_0)$ be the image of $M_0$ by $\Phi_{w_0}$. Roughly speaking, $M_0$ and $M_1$ only differ by a bump of width $\delta$ and height $\eta$ in the neighborhood of $p_0$.
	Note by now that $\Phi_{w_0}$ coincides with the identity map outside $\B(p_0,\delta)$ and in particular, $p_0 = \Phi_{w_0}(p_0) \in M_1$ as soon as $\delta \le \rch_{\min}$.
	
	Combining \Cref{prop:diffeomorphism_stability} and \Cref{lem:multiple_bump_map_is_nice}, we get that $M_1 \in \manifolds{n}{d}{\rch_{\min}}$ and $\Haus^d(M_0)/2 \le \Haus^d(M_1) \le 2 \Haus^d(M_0)$ as soon as
	\begin{align*}
	    \frac{5\eta}{2\delta}
	    \le
	   \frac{1}{10d}
	    \text{  and  }
	    \frac{23 \eta}{\delta^2}
	    \le \frac{1}{4\rch_{\min}}
	    .
	\end{align*}
	Under these assumptions, we have in particular that 
	$
	    \dHaus(M_0, M_1) 
	    \le
	    \norm{\Phi_{w_0} - I_n}_\infty 
	    \le 
	    \eta 
	    \le 
	    \rch_{\min}/10
        .
    $
    Also, by construction, $\Phi_{w_0}(p_0) = p_0 + \eta w_0$ belongs to $M_1$, so that
    \begin{align*}
        \dHaus(M_0,M_1) 
        \ge 
        \dd(p_0+\eta w_0,M_0)
        =
        \eta
        ,
    \end{align*}
    since $w_0 \in \left(T_{p_0} M_0 \right)^\perp$~\cite[Theorem 4.8 (12)] {Federer59}.
    Let us now consider the uniform probability distributions $D_0$ and $D_1$ over $M_0$ and $M_1$ respectively. These distributions have respective densities $f_i = \Haus^d(M_i)^{-1} \indicator{M_i}$ ($i \in \set{0,1}$) with respect to the $d$-dimensional Hausdorff measure $\Haus^d$ on $\R^n$.
    Furthermore, $\Phi_{w_0}$ is a global diffeomorphism that coincides with the identity map on $\B(p_0,\delta)^c$.
    As a result, since $\frac{5\eta}{2\delta} \le \frac{1}{10d} \le (2^{1/d} - 1)$,~\cite[Lemma D.2]{Aamari19} yields that for $\delta \le \rch_{\min}/2$,
    \begin{align*}
        \TV(D_0,D_1)
        &\le
        12 D_0\left( \B(0,\delta) \right)
        \\
        &=
        12\Haus^d(M_0 \cap  \B(0,\delta) )/\Haus^d(M_0)
        \\
        &\le
        12(2^d \omega_d \delta^d)/\Haus^d(M_0)
        ,
    \end{align*}
    where we applied the upper bound of \Cref{lem:intrinsic_ball_mass} to get the last inequality, using that $\rch_{M_0} \geq 2\rch_{\min}$.
    
    Finally, setting $\eta = \delta^2/(92\rch_{\min})$ yields a valid choice of parameters for all $\delta \le \rch_{\min}/(2300d)$.
    Hence, we have shown that for all $\delta \le \rch_{\min}/(2^{12}d) \le \rch_{\min}/(2300d)$,
    \begin{align*}
        \dHaus(M_0,M_1) 
        \ge
        \frac{\delta^2}{92\rch_{\min}}
        \text{ and }
        \TV(D_0,D_1) 
        \le
         12(2^d \omega_d \delta^d)/\Haus^d(M_0)
        .
    \end{align*}
    Equivalently, setting 
    $\tau/2 = 12(2^d \omega_d \delta^d)/\Haus^d(M_0)$
    and 
    $\tau_{(0)} := 24 \omega_d (\rch_{\min}/(2^{11}d))^d/\Haus^d(M_0)$,
    we have shown that for all $\tau \le \tau_{(0)}$,
    there exists $M_1 \in \manifolds{n}{d}{\rch_{\min}}$ such that
    \begin{align*}
        \dHaus(M_0,M_1) 
        \ge
        \frac{1}{92\rch_{\min}}
        \left(
        \frac{\Haus^d(M_0) \tau}{24 (2^d \omega_d)}
        \right)^{2/d}
        \text{ and }
        \TV(D_0,D_1) 
        \le
        \tau/2
        .
    \end{align*}
    We conclude the proof for $\tau \le \tau_{(0)}$ by further bounding the term
    \begin{align*}
        \dHaus(M_0,M_1) 
        &\ge
        \frac{1}{92\rch_{\min}}
        \left(
        \frac{\Haus^d(M_0) \tau}{24 (2^d \omega_d)}
        \right)^{2/d}
        \\
        &=
        \frac{\rch_{\min}}{368 \times 24^{2/d}}
        \left(
        \frac{\Haus^d(M_0)\tau}{\omega_d \rch_{\min}^d}
        \right)^{2/d}
        \\
        &\geq
        \frac{\rch_{\min}}{2^{18}}
        \left(
        \frac{\Haus^d(M_0) \tau}{\omega_d \rch_{\min}^d}
        \right)^{2/d}
        .
    \end{align*}
    Otherwise, if $\tau > \tau_{(0)}$, then the above construction applied with $\tau_{(0)}$ yields the existence of some $M_1 \in \manifolds{n}{d}{\rch_{\min}}$ with the same properties, and
    \begin{align*}
        \dHaus(M_0,M_1) 
        &\ge
        \frac{\rch_{\min}}{2^{18}}
        \left(
        \frac{\Haus^d(M_0) \tau_{(0)}}{\omega_d \rch_{\min}^d}
        \right)^{2/d}
        \text{ and }
        \TV(D_0,D_1) \le \tau_{(0)}/2 \le \tau/2
        .
    \end{align*}
    Summing up the two cases above, for all $\tau \le 1$ we have exhibited some $M_1 \in \manifolds{n}{d}{\rch_{\min}}$ with properties as above, $\TV(D_0,D_1) \le \tau/2$ and
    \begin{align*}
        \dHaus(M_0,M_1) 
        &\ge
        \frac{\rch_{\min}}{2^{18}}
        \left(
        \frac{\Haus^d(M_0) \min\set{ \tau, \tau_{(0)}}}{\omega_d \rch_{\min}^d}
        \right)^{2/d}
        \\
        &\geq
        \frac{\rch_{\min}}{2^{18}}
        \min\set{
            \frac{1}{2^{22}d^2}
            ,
            \left(
            \frac{\Haus^d(M_0) \tau}{\omega_d \rch_{\min}^d}
            \right)^{2/d}
        }
        ,
    \end{align*}
    which concludes the proof.
        
    \end{proof}

Applying the technique of \Cref{prop:hypotheses_for_le_cam_local_variations} with manifolds $M_0$ having largest possible volume (typically of order $1/f_{\min}$) in the models $\distributionspoint{n}{d}{\rch_{\min}}{f_{\min}}{f_{\max}}{L}{0}$ and $\distributionsball{n}{d}{\rch_{\min}}{f_{\min}}{f_{\max}}{L}{R}$ yields the following result.
The proof follows the ideas of~\cite[Lemma~5]{Aamari19b}. To our knowledge, the first result of this type dates back to~\cite[Theorem~6]{Genovese12b}.

\begin{lemma}
    \label{lem:hypotheses_for_le_cam}
    \begin{itemize}[leftmargin=*]
   	    \item 
   	    Assume that $f_{\min} \le f_{\max}/4$ and that
   	    \[
    	    2^{d + 1}\sigma_d f_{\min} \rch_{\min}^d
    	    \le
    	    1
    	    .
   	    \]
        Then for all $\tau \le 1$, there exist $D_0,D_1 \in \distributionspoint{n}{d}{\rch_{\min}}{f_{\min}}{f_{\max}}{L}{0}$ with respective supports $M_0$ and $M_1$ such that
        \begin{align*}
            \dHaus(M_0,M_1) 
            \ge
            \frac{\rch_{\min}}{2^{20}}
            \min\set{
                \frac{1}{2^{20}d^2}
                ,
                \left(
                \frac{\tau}{\omega_d f_{\min} \rch_{\min}^d}
                \right)^{2/d}
            }
            \text{ and }
            \TV(D_0,D_1) 
            \le
            \tau/2
            .
        \end{align*}
    	    \item
    	    Assume that $\rch_{\min} \le R/144$ and $f_{\min} \le f_{\max}/96$.
    	    Writing $C_d' = 9(2^{2d+1} \sigma_{d - 1})$, assume that
        	\[
        	    \min_{1 \le k \le n}
            	\left(\frac{192\rch_{\min} \sqrt{k}}{R}\right)^k
    	        \le
        	    2^{d + 1} C_d' f_{\min} \rch_{\min}^d
            	\le
            	1
            	.
        	\]
            Then for all $\tau \le 1$, there exist $D_0,D_1 \in \distributionsball{n}{d}{\rch_{\min}}{f_{\min}}{f_{\max}}{L}{R}$ with respective supports $M_0$ and $M_1$ such that
            \begin{align*}
                \dHaus(M_0,M_1)    
                \geq
                \frac{\rch_{\min}}{2^{30}}
                \min\set{
                \frac{1}{2^{10}d^2}
                ,
                \left(
                \frac{\tau}{\omega_d f_{\min} \rch_{\min}^d}
                \right)^{2/d}
                }
                \text{ and }
                \TV(D_0,D_1) 
                \le
                \tau/2
                .
            \end{align*}
    \end{itemize}
\end{lemma}

\begin{proof}[\proofof \Cref{lem:hypotheses_for_le_cam}]
    For both models, the idea is to first build a manifold $M_0 \in \manifolds{n}{d}{2\rch_{\min}}$ with prescribed volume close to $1/f_{\min}$, and then consider the variations of it given by \Cref{prop:hypotheses_for_le_cam_local_variations}.
    \begin{itemize}[leftmargin=*]
        \item 
        Let $M_0$ be a $d$-dimensional sphere of radius 
        $r_0 = \left( \frac{1}{2\sigma_d f_{\min}} \right)^{1/d}$ in $\R^{d + 1}\times \set{0}^{n-(d + 1)} \subseteq \R^n$ containing $x_0 = 0 \in \R^n$.
        By construction, $\rch_{M_0} = r_0 \geq 2 \rch_{\min}$, so that $M_0 \in \manifolds{n}{d}{2\rch_{\min}}$, and one easily checks that $\Haus^d(M_0) = 1/(2f_{\min})$.
        For all $\tau \le 1$, \Cref{prop:hypotheses_for_le_cam_local_variations} asserts that there exists a manifold $M_1 \in \manifolds{n}{d}{\rch_{\min}}$ such that $x_0 \in M_1$, with volume 
        \[
            1 / f_{\max}
            \le
            1/(4f_{\min}) 
            \le
            \Haus^d(M_0) 
            \le 
            \Haus^d(M_1) 
            \le 
            2 \Haus^d(M_0) 
            \le 
            1/f_{\min}
            ,
        \]
        such that
        \begin{align*}
                \dHaus(M_0,M_1)    
                &\geq
                \frac{\rch_{\min}}{2^{18}}
                \min\set{
                \frac{1}{2^{22}d^2}
                ,
                \left(
                \frac{\tau}{2 \omega_d f_{\min} \rch_{\min}^d}
                \right)^{2/d}
                }
                \\
                &\geq
                \frac{\rch_{\min}}{2^{20}}
                \min\set{
                \frac{1}{2^{20}d^2}
                ,
                \left(
                \frac{\tau}{\omega_d f_{\min} \rch_{\min}^d}
                \right)^{2/d}
                }
                ,
        \end{align*}
        and with respective uniform distributions $D_0$ and $D_1$ over $M_0$ and $M_1$ that satisfy $\TV(D_0,D_1)\le \tau/2$.
        Since the densities of $D_0$ and $D_1$ are constant and equal to $\Haus^d(M_0)^{-1}$ and $\Haus^d(M_1)^{-1}$ respectively, the bounds on the volumes of $M_0$ and $M_1$ show that $D_0$ and $D_1$ belong to $\distributionspoint{n}{d}{\rch_{\min}}{f_{\min}}{f_{\max}}{L=0}{0} \subseteq \distributionspoint{n}{d}{\rch_{\min}}{f_{\min}}{f_{\max}}{L}{0}$, which concludes the proof.
        
        \item
        Let $M_0 \subseteq \R^n$ be a submanifold given by \Cref{prop:manifold_of_prescribed_volume} applied with parameters $\rch_{\min}' = 2\rch_{\min}$, $\mathcal{V} = 1/(2f_{\min})$ and $R' = R/2$. 
        That is, $M_0 \in \manifolds{n}{d}{2 \rch_{\min}}$ is such that $1/(48 f_{\min}) \le \Haus^d(M_0) \le 1/(2 f_{\min})$ and $M_0 \subseteq \B(0,R/2)$.
        For all $\tau \le 1$, \Cref{prop:hypotheses_for_le_cam_local_variations} asserts that there exists a manifold $M_1 \in \manifolds{n}{d}{\rch_{\min}}$ such that $\dHaus(M_0,M_1) \le \rch_{\min}/10$, with volume 
        \[
            1/f_{\max}
            \le
            1/(96f_{\min}) 
            \le
            \Haus^d(M_0)/2
            \le 
            \Haus^d(M_1) 
            \le 
            2 \Haus^d(M_0) 
            \le 
            1/f_{\min}
            ,
        \]
        and
        \begin{align*}
                \dHaus(M_0,M_1)    
                &\geq
                \frac{\rch_{\min}}{2^{18}}
                \min\set{
                \frac{1}{2^{22}d^2}
                ,
                \left(
                \frac{\tau}{48 \omega_d f_{\min} \rch_{\min}^d}
                \right)^{2/d}
                }
                \\
                &\geq
                \frac{\rch_{\min}}{2^{30}}
                \min\set{
                \frac{1}{2^{10}d^2}
                ,
                \left(
                \frac{\tau}{\omega_d f_{\min} \rch_{\min}^d}
                \right)^{2/d}
                }
                ,
        \end{align*}
        and such that the respective uniform distributions $D_0$ and $D_1$ over $M_0$ and $M_1$ satisfy $\TV(D_0,D_1)\le \tau/2$.
    	Because $M_0 \subseteq \B(0,R/2)$ and $\dHaus(M_0,M_1) \le \rch_{\min}/10 \le R/2$, we immediately get that $M_1 \subseteq \B(0,R/2 + R/2) = \B(0,R)$.
    	 As a result, this family clearly provides the existence of the announced $\varepsilon$-packing of
        $\bigl( \manifoldsball{n}{d}{\rch_{\min}}{R}, \dHaus \bigr)$.
        As above, the bounds on the volumes of $M_0$ and $M_1$ show that $D_0,D_1 \in \distributionsball{n}{d}{\rch_{\min}}{f_{\min}}{f_{\max}}{L=0}{R} \subseteq \distributionsball{n}{d}{\rch_{\min}}{f_{\min}}{f_{\max}}{L}{R}$, which concludes the proof.
    \end{itemize}
\end{proof}

\subsubsection{Proof of the Informational Lower Bounds for Manifold Estimation}

With all the intermediate results above, the proofs of \Cref{thm:SQ_lower_bound_point_informational} and \Cref{thm:SQ_lower_bound_ball_informational} follow straightforwardly.

\begin{proof}[\proofof \Cref{thm:SQ_lower_bound_point_informational} and \Cref{thm:SQ_lower_bound_ball_informational}]
    These are direct applications of \Cref{thm:lecam_SQ} for parameter of interest $\theta(D) = \supp(D)$ and distance $\dist = \dHaus$, with the hypotheses $D_0,D_1$ of the models 
    $\distributionspoint{n}{d}{\rch_{\min}}{f_{\min}}{f_{\max}}{L}{0}$
    and
    $\distributionsball{n}{d}{\rch_{\min}}{f_{\min}}{f_{\max}}{L}{R}$
    given by \Cref{lem:hypotheses_for_le_cam}.
\end{proof}

\subsection{Computational Lower Bounds: Packing Number of Manifold Classes}

We now prove the computational lower bounds \Cref{thm:SQ_lower_bound_point_computational,thm:SQ_lower_bound_ball_computational}.
For this, and in order to apply \Cref{theorem:lower_bound_computational_manifold}, we build explicit packings of the manifold classes.
To study the two models and the different regimes of parameters, we exhibit two types of such packings.
The first ones that we describe (\Cref{prop:packing_lower_bound_ambient}) use translations of a fixed manifold $M_0$ in the ambient space, and are called ambient packings (see \Cref{subsubsec-global-ambient-packings}).
The second ones (\Cref{prop:packing_local_variations}) use a local smooth bumping strategy based on a fixed manifold $M_0$, and are called intrinsic packings (see \Cref{subsubsec-local-intrinsic-packings}).
Finally, the proof of the computational lower bounds are presented in \Cref{subsubsec-computational-lower-bounds-proof}.

\subsubsection{Global Ambient Packings}
\label{subsubsec-global-ambient-packings}

To derive the first manifold packing lower bound, we will use translations in $\R^n$ and the following lemma.

    \begin{lemma}
        \label{lem:hausdorff_distance_between_translations}
        La $K$ be a compact subset of $\R^n$. Given $v\in \R^n$, let $K_v = \set{p+v,p \in K}$ be the translation of $K$ by the vector $v$. Then $\dHaus(K,K_v)= \norm{v}$.
    \end{lemma}
    
    \begin{proof}[\proofof \Cref{lem:hausdorff_distance_between_translations}]
        If $v=0$, the result is straightforward, so let us assume that $v \neq 0$.
        Since $K$ is compact, the map $g$ defined for $p \in K$ by $g(p) = \inner{v/\norm{v}}{p}$ attains its maximum at some $p_0 \in K$.
        But by definition of $K_v$,  $p_0+v \in K_v$, so
        \begin{align*}
            \dHaus(K,K_v)
            &\geq
            \dd(p_0+v,K)
            \\
            &=
            \min_{p \in K} \norm{ (p_0+v)-p}            
            \\
            &\geq
            \min_{p \in K} \inner{\frac{v}{\norm{v}}}{ (p_0+v)-p}
            \\
            &=
            \norm{v} + \min_{p \in K} \inner{\frac{v}{\norm{v}}}{ p_0-p}
            \\
            &=
            \norm{v}
            .
        \end{align*}
        On the other hand, for all $p \in K$ we have $p+v \in K_v$, yielding $\dd(p,K_v) \le \norm{v}$, and symmetrically $\dd(p+v,K) \le \norm{v}$. Therefore $\dHaus(K,K_v) \le \norm{v}$, which concludes the proof.
    \end{proof}
    
As a result, packings of sets in $\R^n$ naturally yields packings in the manifold space, by translating a fixed manifold $M_0 \subset \R^n$. With this remark in mind, we get the following ambient packing lower bound.

\begin{proposition}
\label{prop:packing_lower_bound_ambient}
	Assume that $\rch_{\min} \le R/24$.
	Writing $C_d = 9(2^d \sigma_{d - 1})$, let $\mathcal{V}>0$ be such that
	\[
    	1
    	\le
    	\frac{\mathcal{V}}{C_d \rch_{\min}^d}
    	\le
    	\max_{1 \le k \le n}
    	\left(\frac{R}{48 \rch_{\min} \sqrt{k}}\right)^k
    	.
	\]
	Then for all $\varepsilon \le R/2$,
	\begin{align*}
		\log 
		\PK_{\bigl( \manifoldsball{n}{d}{\rch_{\min}}{R} , \dHaus \bigr)}(\varepsilon)
		\geq
		n \log\left( \frac{R}{4\varepsilon} \right)
		,
	\end{align*}
	and such a packing can be chosen so that all its elements $M$ 
    have volume $\mathcal{V}/6 \le \Haus^d(M) \le \mathcal{V}$.
\end{proposition}
    
\begin{proof}[\proofof \Cref{prop:packing_lower_bound_ambient}]
    Let $z_1, \dots,z_N \in \B(0,R/2)$ be a $r$-packing of $\B(0,R/2)$. From 
    \Cref{prop:packing_covering_ball_sphere}, such a packing can be taken so 
    that $N \geq (R/(4r))^n$. Applying \Cref{prop:manifold_of_prescribed_volume}
    with parameters $\rch_{\min}$, $\mathcal{V}$ and $R' = R/2$, we get the 
    existence of some $M_0 \in \manifolds{n}{d}{\rch_{\min}}$ such that 
    $\mathcal{V}/6 \le \Haus^d(M_0) \le \mathcal{V}$ and $M_0 \subseteq \B(0,R/2)$.
    Note that for all $z \in \B(0,R/2)$, the translation 
    $M_z = \set{p + z,p \in M_0}$ belongs to $\manifolds{n}{d}{\rch_{\min}}$, 
    has the same volume as $M_0$, and satisfies 
    $M_z \subseteq \B(0,R/2+\norm{z}) \subseteq \B(0,R)$.
    In addition, \Cref{lem:hausdorff_distance_between_translations} asserts
    that for all $z,z' \in \B(0,R/2)$, $\dHaus(M_z,M_{z'}) = \norm{z-z'}$.
    In particular, for all 
    $i \neq j \in \set{1, \dots,N}$, $\dHaus(M_{z_i},M_{z_j}) = \norm{z_i-z_j} > 2r$.
    As a result, the family $\set{M_{z_i}}_{1 \le i \le N}$ provides us with an
    $r$-packing of $\bigl( \manifoldsball{n}{d}{\rch_{\min}}{R} , \dHaus \bigr)$ 
    with cardinality $N \geq (R/(4r))^n$, and composed of submanifold with volume
    $\mathcal{V}/6 \le \Haus^d(M) \le \mathcal{V}$, which concludes the proof.
\end{proof}  

\subsubsection{Local Intrinsic Packings}
\label{subsubsec-local-intrinsic-packings}

In the same spirit as \Cref{prop:hypotheses_for_le_cam_local_variations} for 
informational lower bounds, the following result allows to build packings of 
manifold classes by small perturbations of a base submanifold $M_0$. Note, again,
that the larger the volume $\Haus^d(M_0)$, the stronger the result.
    
\begin{proposition}
\label{prop:packing_local_variations}
    For all $M_0 \in \manifolds{n}{d}{2\rch_{\min}}$ and $r \le \rch_{\min}/(2^{34} d^2)$,
    there exists a family of submanifolds 
    $\set{M_s}_{1 \le s \le \mathcal{N}} \subseteq \manifolds{n}{d}{\rch_{\min}}$ with 
    cardinality $\mathcal{N}$ such that
    \[
        \log \mathcal{N}
    	\geq
        n \frac{\Haus^d(M_0)}{\omega_d \rch_{\min}^d}
        \left(
    	    \frac{\rch_{\min}}{2^{19}r}
    	\right)^{d/2}
    	,
    \]
    and that satisfies:
    \begin{itemize}[leftmargin=*]
        \item
    	    $M_0$ and $\set{M_s}_{1 \le s \le \mathcal{N}}$ have a point in common: 
    	    $M_0 \cap \bigl( \cap_{1 \le s \le \mathcal{N}} M_s \bigr) \neq \emptyset$.
        \item
    	    For all $s\in \set{1, \dots, \mathcal{N}}$,
            \[
                \dHaus(M_0,M_s) \le 23r
                \text{  and  }
    	        \Haus^d(M_0)/2 \le \Haus^d(M_s) \le 2 \Haus^d(M_0)
    	        .
    	    \]
    	    \item
    	        For all $s \neq s' \in \set{1, \dots, \mathcal{N}}$
    	        $\dHaus(M_s,M_{s'}) > 2r$.
    \end{itemize}
\end{proposition}
    
    \begin{proof}[\proofof \Cref{prop:packing_local_variations}]
    For $\delta \le \rch_{\min}/8$ to be chosen later, let $\set{p_i}_{1\le i \le N}$ be a maximal $\delta$-packing of $M_0$.
    From \Cref{prop:packing_covering_manifold}, this maximal packing has cardinality $N \geq \frac{\Haus^d(M_0)}{\omega_d (4\delta)^d}$.
    
    Let $\eta >0$ be a parameter to be chosen later. 
    Given a family of unit vectors $\mathbf{w} = (w_i)_{1 \le i \le N} \in \left( \R^n \right)^N$ normal at the $p_i$'s, i.e. $w_i \in \left( T_{p_i} M \right)^\perp$ and $\norm{w_i}=1$, we let $\Phi_{\mathbf{w}}$ be the function defined in \Cref{lem:multiple_bump_map_is_nice}, that maps any $x\in \R^n$ to
    \begin{align*}
        \Phi_{\mathbf{w}}(x)
        &=
        x
        +
        \eta
        \left(
            \sum_{i = 1}^N
            \phi\left(\frac{x-p_i}{\delta}\right)
            w_i
        \right)
        ,
    \end{align*}
	where $\phi : \R^n \to \R$ is the real bump function 
	$\phi(y) = \exp\left(-{\norm{y}^2}/{(1-\norm{y}^2)} \right) \indicator{\B(0,1)}(y)$ 
	of \Cref{lem:multiple_bump_map_is_nice}.
	We let $M_\mathbf{w} = \Phi_\mathbf{w}(M_0)$ be the image of $M_0$ by $\Phi_\mathbf{w}$. The set $M_\mathbf{w} \subseteq \R^n$ hence coincides with $M_0$, except in the $\delta$-neighborhoods of the $p_i$'s, where it has a bump of size $\eta$ towards direction $w_i$.
	Note by now that up to rotations of its coordinates, the vector $\mathbf{w} = (w_i)_{1 \le i \le N}$ belongs to $\Sphere^{n-d}(0,1)^N$.
	Combining \Cref{prop:diffeomorphism_stability} and \Cref{lem:multiple_bump_map_is_nice}, we see that $M_\mathbf{w} \in \manifolds{n}{d}{\rch_{\min}}$ and $\Haus^d(M_0)/2 \le \Haus^d(M_{\mathbf{w}}) \le 2 \Haus^d(M_0)$ as soon as
	\begin{align*}
	    \frac{5\eta}{2\delta}
	    \le
	    \frac{1}{10d}
	    \text{  and  }
	    \frac{23 \eta}{\delta^2}
	    \le \frac{1}{4\rch_{\min}}
	    .
	\end{align*}
	In the rest of the proof, we will work with these two inequalities holding true. 
	In particular, because $\norm{\Phi_{\mathbf{w}} - I_n}_\infty \le \eta$, we immediately get that $\dHaus(M_0,M_{\mathbf{w}}) \le \eta$.
	We note also that all the $\Phi_\mathbf{w}$'s coincide with the identity map on (say) $M_0 \cap \partial \B(p_1,\delta)$, so that $M_0 \cap \bigl( \cap_{\mathbf{w}} M_\mathbf{w} \bigr)$ contains $M_0 \cap \partial \B(x_1,\delta)$ and is hence non-empty.

	We now take two different families of unit normal vectors $\mathbf{w}$ and $\mathbf{w}'$ (i.e. $w_i,w_i' \in \left( T_{p_i} M_0 \right)^\perp$ and $\norm{w_i} = \norm{w_i'} = 1$ for $1 \le i \le N$), and we will show that their associated submanifolds $M_\mathbf{w}$ and $M_{\mathbf{w}'}$ are far away in Hausdorff distance as soon as $\max_{1\le i \le N} \norm{w_i - w_i'}$ is large enough.	
	To this aim, we first see that by construction, $\Phi_\mathbf{w}(p_i) = p_i + \eta w_i \in \Phi_\mathbf{w}(M_0) = M_\mathbf{w}$ for all $i \in \set{1, \dots,N}$. In particular,
	\begin{align*}
	    \dHaus(M_{\mathbf{w}},M_{\mathbf{w}'})
	    &\geq
	    \max_{1 \le i \le N}
	    \dd(p_i + \eta w_i, M_{\mathbf{w}'})
	    .
    \end{align*}
    Let us fix a free parameter $\lambda_i \in [0,1]$ to be chosen later. As $\norm{\Phi_\mathbf{w'}-I_n}_\infty \le \eta$, we can write for all $i \in \set{1, \dots,N}$ that
    \begin{align*}
        \dd(p_i + \eta w_i, M_{\mathbf{w}'})
        &=
        \dd\left(
            p_i + \eta w_i
            , 
            \Phi_{\mathbf{w}'}(M_{0})
        \right)
        \\
        &=
        \dd\left(
            p_i + \eta w_i
            , 
            \Phi_{\mathbf{w}'}(M_{0} \setminus \B(p_i,\lambda_i \delta))
        \right)
        \wedge
        \dd\bigl(
            p_i + \eta w_i
            , 
            \Phi_{\mathbf{w}'}(M_{0} \cap \B(p_i,\lambda_i \delta))
        \bigr)        
        \\
        &\geq
        (\lambda_i \delta - \eta)
        \wedge
        \dd\bigl(
            p_i + \eta w_i
            , 
            \Phi_{\mathbf{w}'}(M_{0} \cap \B(p_i,\lambda_i \delta))
        \bigr)
        .
    \end{align*}
	Further investigating the term
	$
	\dd\left(
            p_i + \eta w_i
            , 
            \Phi_{\mathbf{w}'}(M_{0} \cap \B(p_i,\lambda_i \delta))
        \right)
    $,
	we see that for all $x \in M_0 \cap \B(p_i,\lambda_i \delta) \subseteq \B(p_i,\delta)$,
	$\Phi_{\mathbf{w}'} (x) = x + \eta \phi\left(\frac{x-p_i}{\delta}\right)w_i$.
	But from~\cite[Theorem 4.18]{Federer59}, $\rch_{M_0} \geq 2 \rch_{\min}$ ensures that any $x \in M_0 \cap \B(p_i,\lambda_i \delta)$ can be written as $x = p_i + v + u$, where $v \in T_{p_i} M_0$ with $\norm{v} \le \lambda_i\delta$, and $u \in \bigl( T_{p_i} M_0 \bigr)^\perp$ with $\norm{u} \le (\lambda_i\delta)^2/(4\rch_{\min})$.
	As a result, we have
	\begin{align*}
        \dd&\left(
            p_i + \eta w_i
            , 
            \Phi_{\mathbf{w}'}(M_{0} \cap \B(p_i,\lambda_i \delta))
        \right)
        \\
        &\geq
        \min_{
        \substack{
            v \in  T_{p_i} M_0,\norm{v} \le \lambda_i\delta
            \\
            u \in ( T_{p_i} M_0 )^\perp, \norm{u} \le (\lambda_i\delta)^2/(4\rch_{\min})
        }
        }
        \norm{
            v + u + \eta \left( \phi\left( \frac{v+u}{\delta} \right) w'_i - w_i \right)
        }
        .
	\end{align*}
	But in the above minimum, $v$ is orthogonal to $u,w_i$ and $w_i'$, so
	\begin{align*}
        \norm{
            v + u + \eta \left( \phi\left( \frac{v+u}{\delta} \right) w'_i - w_i \right)
        }
        \geq
        \norm{
            u + \eta \left( \phi\left( \frac{v+u}{\delta} \right) w'_i - w_i \right)
        }.
	\end{align*}
	Additionally, $\phi\left( \frac{v+u}{\delta} \right)$ ranges in (a subset of) $[0,1]$ since $0 \le \phi \le 1$. 
	In particular,
	\begin{align*}
        \dd\left(
            p_i + \eta w_i
            , 
            \Phi_{\mathbf{w}'}(M_{0} \cap \B(p_i,\lambda_i \delta))
        \right)
        &\geq
        \min_{
        \substack{
            u \in ( T_{p_i} M_0 )^\perp, \norm{u} \le (\lambda_i\delta)^2/(4\rch_{\min})
            \\
            0 \le t \le 1
        }
        }
        \norm{
            u + \eta \left( t w'_i - w_i \right)
        }
        \\
        &\geq
        \min_{0 \le t \le 1}
        \eta \norm{t w'_i - w_i}
        -
        \frac{(\lambda_i\delta)^2}{4\rch_{\min}}
        \\
        &=
        \eta \norm{\left(0 \vee \inner{w_i}{w'_i} \right) w'_i - w_i}
        -
        \frac{(\lambda_i\delta)^2}{4\rch_{\min}}
        \\
        &\geq
        \eta \frac{\norm{w'_i - w_i}}{2}
        -
        \frac{(\lambda_i\delta)^2}{4\rch_{\min}}
        ,
	\end{align*}
	where the second line follows from triangle inequality, and the last two from elementary calculations.
	Putting everything together, we have shown that for all $\lambda_1, \dots, \lambda_N \in [0,1]$,
	\begin{align*}
	    \dHaus(M_{\mathbf{w}},M_{\mathbf{w}'})
	    &\geq
	    \max_{1 \le i \le N}
        \set{
        (\lambda_i \delta - \eta)
	    \wedge
	    \left(
            \eta \frac{\norm{w'_i - w_i}}{2}
            -
            \frac{(\lambda_i\delta)^2}{4\rch_{\min}}
        \right)
        }        
        .
	\end{align*}
	One easily checks that under the above assumptions on the parameters, 
	\[
	\lambda_i 
	:= 
	\frac{\sqrt{\sqrt{2} \rch_{\min} \norm{w'_i -w_i} \eta}}{\delta}
	\]
	provides valid choices of $\lambda_i \in [0,1]$. Plugging these values in the previous bound yields
	\begin{align*}
	    \dHaus(M_{\mathbf{w}},M_{\mathbf{w}'})
	    &\geq
	    \max_{1 \le i \le N}
        \set{
        \left( \sqrt{\sqrt{2} \rch_{\min} \norm{w'_i -w_i} \eta}- \eta \right)
	    \wedge
	    \left(
            \eta \frac{\norm{w'_i - w_i}}{8}
        \right)
        }        
        ,
	\end{align*}
	so that if we further assume that $\norm{w'_i -w_i} \geq 4\sqrt{2}\eta/\rch_{\min}$, we obtain
	\begin{align*}
	    \dHaus(M_{\mathbf{w}},M_{\mathbf{w}'})
	    &\geq
	    \max_{1 \le i \le N}
        \set{
        \eta
	    \wedge
	    \left(
            \eta \frac{\norm{w'_i - w_i}}{8}
        \right)
        }
	    \\
	    &=
	    \frac{\eta}{8}
	    \max_{1 \le i \le N}
        \norm{w'_i - w_i}
        ,
	\end{align*}
	where the last line follows from $\norm{w_i-w'_i} \le \norm{w_i} + \norm{w'_i} \le  2$.
	
	Setting $\eta = \delta^2/(92\rch_{\min})$, which is a value that satisfies all the requirements above as soon as $\delta \le \rch_{\min}/(2300 d)$, we have built a family of submanifolds $\set{M_\mathbf{w} }_{\mathbf{w}}$ of $\manifolds{n}{d}{\rch_{\min}}$ indexed by $\mathbf{w} \in \Sphere^{n-d}(0,1)^N$, such that $\Haus^d(M_0)/2 \le \Haus^d(M_\mathbf{w}) \le 2 \Haus^d(M_0)$, and which are guaranteed to satisfy
	\begin{align*}
	    \dHaus(M_{\mathbf{w}},M_{\mathbf{w}'})
	    &>
	     \frac{\delta^2}{8 (92 \rch_{\min})}
	     \times
	     \frac{1}{4}
        \geq
        2
        \left(
        \frac{\delta^2}{2082 \rch_{\min}}
        \right)
	    ,
    \end{align*}
    provided that $\max_{1 \le i \le N} \norm{w'_i - w_i} > 1/4 = 2/8$.
    As a result, if we consider $(1/8)$-packings of the unit spheres $\Sphere_{(T_{p_i} M_0)^\perp}(0,1) = \Sphere^{n-d}(0,1)$ for $i \in \set{1, \dots,N}$, then for all $\delta \le \rch_{\min}/(2300 d)$, it naturally defines a $\left(\frac{\delta^2}{2082 \rch_{\min}}\right)$-packing of $\manifolds{n}{d}{\rch_{\min}}$ with cardinality $\mathcal{N}$ a least
	\begin{align*}
	 \mathcal{N}
	 &\geq
	 \PK_{\Sphere^{n-d}(0,1)}(1/8)^N 
	 \geq
	 \PK_{\Sphere^{n-d}(0,1)}(1/8)^{\frac{\Haus^d(M_0)}{\omega_d (4\delta)^d}}
	 ,
	\end{align*}
	and which consists of elements $M_\mathbf{w}$ such that $\Haus^d(M_0)/2 \le \Haus^d(M_\mathbf{w}) \le 2\Haus^d(M_0)$ and $\dHaus(M_0,M_\mathbf{w}) \le \eta = \delta^2/(92\rch_{\min})$.
	In particular, by setting $r = \frac{\delta^2}{2082 \rch_{\min}}$, then for all $0 < r \le \rch_{\min}/(2^{34}d^2)$, we have exhibited a $r$-packing of $\manifolds{n}{d}{\rch_{\min}}$ of cardinality $\mathcal{N}$ with
    \begin{align*}
    		\log \mathcal{N}
    		&\geq
    		\frac{\Haus^d(M_0)}{\omega_d (4 \sqrt{2082 \rch_{\min} r})^d}
    		 \log \PK_{\Sphere^{n-d}(0,1)}(1/8)
            ,
            \end{align*}
    composed of submanifolds having volume as above, and $\dHaus(M_0,M_\mathbf{w}) \le 2082 r/92 \le 23r$.
	From \Cref{prop:packing_covering_ball_sphere},
	$\log \PK_{\Sphere^{n-d}(0,1)}(1/8) \geq (n-d) \log 2$.
	Finally, by considering the cases $d \le n/2$ and $d \geq n/2$, one easily checks that $(n-d) \geq n/(2d)$.
	In all, we obtain the announced bound
    	\begin{align*}
    		\log 
    		\mathcal{N}
    		&\geq
    		\frac{\Haus^d(M_0)}{\omega_d (4 \sqrt{2082 \rch_{\min} r})^d}
    		\frac{n \log 2}{2d}
    		\\
    		&\geq
    		n
    		\frac{\Haus^d(M_0)}{\omega_d \rch_{\min}^d}
    		\left(
    		\frac{\rch_{\min}}{2^{19} r}
    		\right)^{d/2}
    		,
    	\end{align*}
    which yields the announced result.
    \end{proof}

Applying the technique of \Cref{prop:packing_local_variations} with manifolds $M_0$ having a large prescribed volume $\manifoldspoint{n}{d}{\rch_{\min}}{0}$ and $\manifoldsball{n}{d}{\rch_{\min}}{R}$ respectively yields the following result.

    \begin{proposition}
    	\label{prop:packing_lower_bound_local}
    	Let $\mathcal{V} > 0$ and $\varepsilon \le \rch_{\min}/(2^{34} d^2)$.
    	\begin{itemize}[leftmargin=*]
    	    \item 
    	    Assume that
    	    \[
    	    1
    	    \le
    	    \frac{\mathcal{V}}{2^{d + 1}\sigma_d \rch_{\min}^d}
    	    .
    	    \]
    	    Then,
        	\begin{align*}
        		\log 
        		\PK_{\bigl( \manifoldspoint{n}{d}{\rch_{\min}}{0}, \dHaus \bigr)}(\varepsilon)
    	    	\geq
        		n
        		\frac{\mathcal{V}}{\omega_d \rch_{\min}^d}
    	    	\left(
        		\frac{\rch_{\min}}{2^{21} \varepsilon}
        		\right)^{d/2}
        		.
    	    \end{align*}
    Furthermore, this packing can be chosen so that all its elements $M$ satisfy 
    	\[
    	\mathcal{V}/4
    	\le
    	\Haus^d(M)
    	\le
    	\mathcal{V}
    	.
    	\]

    	    \item
    	    Assume that $\rch_{\min} \le R/144$. Writing $C_d' = 9(2^{2d+1} \sigma_{d - 1})$, assume that
        	\[
            	1
    	        \le
        	    \frac{\mathcal{V}}{2^{d + 1} C_d' \rch_{\min}^d}
            	\le
        	    \max_{1 \le k \le n}
            	\left(\frac{R}{192\rch_{\min} \sqrt{k}}\right)^k
            	.
        	\]
        	Then,
        	\begin{align*}
    	    	\log 
        		\PK_{\bigl( \manifoldsball{n}{d}{\rch_{\min}}{R} , \dHaus \bigr)}(\varepsilon)
        		\geq
        		n
    	    	\frac{\mathcal{V}}{\omega_d \rch_{\min}^d}
        		\left(
        		\frac{\rch_{\min}}{2^{31} \varepsilon}
        		\right)^{d/2}
        		.
        	\end{align*}
    	Furthermore, this packing can be chosen so that all its elements $M$ satisfy 
    	\[
    	\mathcal{V}/96
    	\le
    	\Haus^d(M)
    	\le
    	\mathcal{V}
    	.
    	\]
    	\end{itemize}

    \end{proposition}
    
    \begin{proof}[\proofof \Cref{prop:packing_lower_bound_local}]
    For both models, the idea is to first build a manifold $M_0 \in \manifolds{n}{d}{2\rch_{\min}}$ with prescribed volume close to $\mathcal{V}$, and then consider the variations of it given by \Cref{prop:packing_local_variations}.
    \begin{itemize}[leftmargin=*]
        \item 
        Let $M_0$ be the centered $d$-dimensional sphere of radius 
        $r_0 = \left( \frac{\mathcal{V}/2}{\sigma_d} \right)^{1/d}$ in $\R^{d + 1}\times \set{0}^{n-(d + 1)} \subseteq \R^n$.
        By construction, $\rch_{M_0} = r_0 \geq 2 \rch_{\min}$, so that $M_0 \in \manifolds{n}{d}{2\rch_{\min}}$. Furthermore, one easily checks that $\Haus^d(M_0) = \mathcal{V}/2$.
        From \Cref{prop:packing_local_variations}, there exists a family of submanifolds $\set{M_s}_{1 \le s \le \mathcal{N}} \subseteq \manifolds{n}{d}{\rch_{\min}}$ with cardinality $\mathcal{N}$ such that
    	\begin{align*}
    	        \log \mathcal{N}
    	        &\geq
        		n\frac{\mathcal{V}/2}{\omega_d \rch_{\min}^d}
        		\left(
    	    	\frac{\rch_{\min}}{2^{19}\varepsilon}
    	    	\right)^{d/2}
    	    	\\
    	    	&\geq
        		n\frac{\mathcal{V}}{\omega_d \rch_{\min}^d}
        		\left(
    	    	\frac{\rch_{\min}}{2^{21}\varepsilon}
    	    	\right)^{d/2}
    	    	,
    	\end{align*}
    	that all share a point $x_0 \in \cap_{1 \le s \le \mathcal{N}} M_s$, and such that 
        $\dHaus(M_s,M_{s'}) > 2\varepsilon$ for all $s \neq s' \in \set{1, \dots, \mathcal{N}}$,
    	with volumes
    	$\mathcal{V}/4 = \Haus^d(M_0)/2 \le \Haus^d(M_s) \le 2 \Haus^d(M_0) = \mathcal{V}$.
        As a result, the family given by the translations $M'_s = M_s - x_0$ clearly provides the existence of the announced $\varepsilon$-packing of
        $\bigl( \manifoldspoint{n}{d}{\rch_{\min}}{0}, \dHaus \bigr)$.

        \item
        Let $M_0 \subseteq \R^n$ be a submanifold given by \Cref{prop:manifold_of_prescribed_volume} applied with parameters $\rch_{\min}' = 2\rch_{\min}$, $\mathcal{V}' = \mathcal{V}/2$ and $R' = R/2$. That is, $M_0 \in \manifolds{n}{d}{2 \rch_{\min}}$ is such that $\mathcal{V}/48 \le \Haus^d(M_0) \le \mathcal{V}/2$ and $M_0 \subseteq \B(0,R/2)$.
        From \Cref{prop:packing_local_variations}, there exists a family of submanifolds $\set{M_s}_{1 \le s \le \mathcal{N}} \subseteq \manifolds{n}{d}{\rch_{\min}}$ with cardinality $\mathcal{N}$ such that
    	\begin{align*}
    	        \log \mathcal{N}
    	        &\geq
        		n\frac{\mathcal{V}/48}{\omega_d \rch_{\min}^d}
        		\left(
    	    	\frac{\rch_{\min}}{2^{19}\varepsilon}
    	    	\right)^{d/2}
    	    	\\
    	    	&\geq
        		n\frac{\mathcal{V}}{\omega_d \rch_{\min}^d}
        		\left(
    	    	\frac{\rch_{\min}}{2^{31}\varepsilon}
    	    	\right)^{d/2}
    	    	,
    	\end{align*}
    	with 
    	$\dHaus(M_0,M_s) \le 23\varepsilon$
        and $\dHaus(M_s,M_{s'}) > 2\varepsilon$ for all $s \neq s' \in \set{1, \dots, \mathcal{N}}$,
    	and volumes
    	$\mathcal{V}/96 \le \Haus^d(M_0)/2 \le \Haus^d(M_s) \le 2 \Haus^d(M_0) \le \mathcal{V}$.
    	Because $M_0 \subseteq \B(0,R/2)$ and $\dHaus(M_0,M_s) \le 23\varepsilon$ for all $s \in \set{1, \dots, \mathcal{N}}$, we immediately get that $M_s \subseteq \B(0,R/2 + 23 \varepsilon) \subseteq \B(0,R)$.
    	 As a result, this family clearly provides the existence of the announced $\varepsilon$-packing of
        $\bigl( \manifoldsball{n}{d}{\rch_{\min}}{R}, \dHaus \bigr)$.
    \end{itemize}
    
\end{proof}
    
\subsubsection{Proof of the Computational Lower Bounds for Manifold Estimation}
\label{subsubsec-computational-lower-bounds-proof}

We are now in position to prove the computational lower bounds presented in this work.
First, we turn to the infeasibiliy result of manifold estimation using statistical queries in the unbounded model $\distributions{n}{d}{\rch_{\min}}{f_{\min}}{f_{\max}}{L}$ (\Cref{prop:SQ_lower_bound_infinity_unbounded}).

\begin{proof}[\proofof \Cref{prop:SQ_lower_bound_infinity_unbounded}]
        Since $\sigma_d f_{\min} \rch_{\min}^d \le 1$, the uniform probability distribution $D_0$ over the centered unit $d$-sphere $M_0 \subseteq \R^{d + 1}\times \set{0}^{n-(d + 1)}$ of radius $\rch_{\min}$ belongs to
        $\distributions{n}{d}{\rch_{\min}}{f_{\min}}{f_{\max}}{L}$.
        Given a unit vector $v \in \R^n$, the invariance of the model by translation yields that the uniform distributions $D_k$ over $M_k = \set{p+(3k\varepsilon)v, p \in M_0}$, for $k \in \Z$, also belong to $\distributions{n}{d}{\rch_{\min}}{f_{\min}}{f_{\max}}{L}$.
        But for all $k \neq k' \in \Z$, $\dHaus(M_k,M_{k'}) = 3|k-k'|\varepsilon > 2\varepsilon$. 
        Hence, writing 
        \[  
            \mathcal{M} = \set{\supp(D), D \in \distributions{n}{d}{\rch_{\min}}{f_{\min}}{f_{\max}}{L}},
        \]
        we see that the family $\set{M_k}_{k \in \Z}$ forms an infinite $\varepsilon$-packing of $(\mathcal{M},\dHaus)$.
        From \Cref{theorem:lower_bound_computational_manifold}, we get that the statistical query complexity of manifold estimation over the model $\distributions{n}{d}{\rch_{\min}}{f_{\min}}{f_{\max}}{L}$ with precision $\varepsilon$ is infinite, which concludes the proof.
\end{proof}

We finally come to the proofs of the computational lower bounds over the fixed point model $\distributionspoint{n}{d}{\rch_{\min}}{f_{\min}}{f_{\max}}{L}{0}$ (\Cref{thm:SQ_lower_bound_point_computational}) and the bounding ball model $\distributionsball{n}{d}{\rch_{\min}}{f_{\min}}{f_{\max}}{L}{R}$ (\Cref{thm:SQ_lower_bound_ball_computational}).

\begin{proof}[\proofof \Cref{thm:SQ_lower_bound_point_computational} and \Cref{thm:SQ_lower_bound_ball_computational}]
    For both results, the idea is to exhibit large enough $\varepsilon$-packings of $\mathcal{M} = \set{\supp(D),D \in \model}$, and apply \Cref{theorem:lower_bound_computational_manifold}. In each case, the assumptions on the parameters $f_{\min},f_{\max}$, $\rch_{\min}$ and $d$ ensure that the uniform distributions over the manifolds given by the packings of \Cref{prop:packing_lower_bound_local} (and \Cref{prop:packing_lower_bound_ambient} for \Cref{thm:SQ_lower_bound_ball_computational}) applied with $\mathcal{V} = 1/f_{\min}$ belong to the model, and hence that $\mathcal{M}$ contain these packings.
    \begin{itemize}[leftmargin=*]
        \item 
        To prove \Cref{thm:SQ_lower_bound_point_computational}, let us write 
        \[
            \mathcal{M}_0
            := 
            \set{\supp(D), D \in \distributionspoint{n}{d}{\rch_{\min}}{f_{\min}}{f_{\max}}{L}{0}} 
            \subseteq 
            \manifoldspoint{n}{d}{\rch_{\min}}{0}
            .
        \]
        From \Cref{theorem:lower_bound_computational_manifold}, any randomized SQ algorithm estimating $M = \supp(D)$ over the model $\distributionspoint{n}{d}{\rch_{\min}}{f_{\min}}{f_{\max}}{L}{0}$ with precision $\varepsilon$ and with probability of success at least $1 - \alpha$ must make at least 
        \[
            q
            \geq 
            \frac{\log \bigl( (1-\alpha) \PK_{(\mathcal{M}_0, \dHaus)}(\varepsilon) \bigr)}{\log (1+\floor{1/\tau})}
        \]
        queries to $\STAT(\tau)$.
        Furthermore, let $\set{M_i}_{1 \le i \le \mathcal{N}}$ be an $\varepsilon$-packing of $\manifoldspoint{n}{d}{\rch_{\min}}{0}$ given by \Cref{prop:packing_lower_bound_local}, that we apply with volume $\mathcal{V} = 1/f_{\min}$.
        Recall that these manifolds are guaranteed to have volumes $1/(4f_{\min}) \le \Haus^d(M_i) \le 1/f_{\min}$.
        From the assumptions on the parameters of the model, we get that the uniform distributions $\set{D_i := \indicator{M_i} \Haus^d /\Haus^d(M_i)}_{1 \le i \le \mathcal{N}}$ over the $M_i$'s all belong to $\distributionspoint{n}{d}{\rch_{\min}}{f_{\min}}{f_{\max}}{L}{0}$.
        In particular, the family $\set{M_i}_{1 \le i \le \mathcal{N}}$ is also an $\varepsilon$-packing of $\mathcal{M}_0$, and therefore
    \begin{align*}
        \log \bigl(\PK_{(\mathcal{M}_0, \dHaus)}(\varepsilon) \bigr)
        &\geq
        \log \mathcal{N}
        \geq
                n
        		\frac{1}{\omega_d f_{\min} \rch_{\min}^d}
    	    	\left(
        		\frac{\rch_{\min}}{2^{21} \varepsilon}
        		\right)^{d/2}
        		,
    \end{align*}
    which yields the announced result.
    
        \item 
        Similarly, to prove \Cref{thm:SQ_lower_bound_ball_computational}, write 
        \[
            \mathcal{M}_R
            := 
            \set{\supp(D), D \in \distributionsball{n}{d}{\rch_{\min}}{f_{\min}}{f_{\max}}{L}{R}} 
            \subseteq 
            \manifoldsball{n}{d}{\rch_{\min}}{R}
            ,
        \]
        and apply \Cref{theorem:lower_bound_computational_manifold} to get
        \[
            q
            \geq 
            \frac{\log \bigl( (1-\alpha) \PK_{(\mathcal{M}_R, \dHaus)}(\varepsilon) \bigr)}{\log (1+\floor{1/\tau})}
            .
        \]
        The assumptions on the parameters ensure that the packings exhibited in \Cref{prop:packing_lower_bound_ambient} and \Cref{prop:packing_lower_bound_local} applied with volume $\mathcal{V} = 1/f_{\min}$ are included in $\mathcal{M}_R$, so that
        \begin{align*}
            \log \bigl(\PK_{(\mathcal{M}_R, \dHaus)}(\varepsilon) \bigr)
            &\geq
                n
        		\max\set{
        		    \log\left(\frac{R}{4\varepsilon} \right)
            		,
            		\frac{1}{\omega_d f_{\min} \rch_{\min}^d}
    	        	\left(
            		\frac{\rch_{\min}}{2^{31} \varepsilon}
            		\right)^{d/2}
        		}
            ,
        \end{align*}
        which concludes the proof.
    \end{itemize}
\end{proof}

%% file: parts/notation.tex
\nomenclature[A,01]{$n$}{Ambient dimension}
\nomenclature[A,02]{$d$}{Intrinsic dimension}
\nomenclature[A,03]{$D$}{Unknown distribution}
\nomenclature[A,04]{$\norm{\cdot}$}{Euclidean Norm}
\nomenclature[A,05]{$\Gr{n}{d}$}{Grassmannian of $d$-dimensional subspaces of $\R^n$}
\nomenclature[A,06]{$\Haus^d$}{$d$-dimensional Hausdorff measure}

\nomenclature[B,01]{$\supp(D)$}{Support of $D$}
\nomenclature[B,02]{$\B(x,r)$}{Closed Euclidean $r$-ball around $x \in \R^n$}
\nomenclature[B,03]{$\omega_d$}{Volume of unit $d$-ball}
\nomenclature[B,04]{$\sigma_{d-1}$}{Surface area of unit $(d-1)$-dimensional sphere}
\nomenclature[B,05]{$M$}{Submanifold}
\nomenclature[B,06]{$\dd_M$}{Geodesic distance on $M$}
\nomenclature[B,07]{$\B_M(p,r)$}{Closed geodesic $r$-ball around $p \in M$}
\nomenclature[B,08]{$T_p M$}{Tangent space of $M$ at $p \in M$}
\nomenclature[B,09]{$\angle$}{Principal angle between linear subspaces}
\nomenclature[B,10]{$\Med$}{Medial axis \hfill (\Cref{subsec:geom-stat-model})}
\nomenclature[B,11]{$\pi_M$}{Projection map onto $M$ \hfill (\Cref{subsec:geom-stat-model})}
\nomenclature[B,12]{$\rch_M$}{Reach of $M$ \hfill (\Cref{def:reach})}

\nomenclature[B,13]{$\dd(\cdot,M)$}{Distance-to-$M$ function}
\nomenclature[B,14]{$\dHaus$}{Hausdorff distance \hfill (\Cref{def:hausdorff_distance})}
\nomenclature[B,15]{$M^r$}{$r$-Offset of $M$ \hfill (\Cref{eq:offset})}
\nomenclature[B,16]{$\CV_K(r)$}{$r$-covering number of $K$ \hfill (\Cref{def:packing_covering})}
\nomenclature[B,17]{$\PK_K(r)$}{$r$-packing number of $K$ \hfill (\Cref{def:packing_covering})}

\nomenclature[D,01]{$\tau$}{Tolerance}
\nomenclature[D,02]{$\STAT(\tau)$}{Class of oracles with tolerance $\tau$ \hfill (\Cref{def:deterministic-SQ})}
\nomenclature[D,03]{$q$}{Number of queries}
\nomenclature[D,04]{$\varepsilon$}{Precision}
\nomenclature[D,05]{$\alpha$}{Probability of failure \hfill (\Cref{def:randomized-SQ})}
\nomenclature[D,06]{$r$}{Query to $\STAT(\tau)$}
\nomenclature[D,07]{$\answer$}{Answer to a query}
\nomenclature[D,08]{$\oracle$}{SQ oracle}
\nomenclature[D,09]{$\Algorithm{A}$}{SQ algorithm}
\nomenclature[D,10]{$\RandmizedAlgorithm{A}$}{Randomized SQ algorithm}

\nomenclature[E,01]{$\mu_i(\cdot)$}{$i$th singular value in decreasing order}
\nomenclature[E,02]{$\normF{\cdot}$}{Frobenius norm}
\nomenclature[E,03]{$\inner{\cdot}{\cdot}$}{Inner product}
\nomenclature[E,04]{$\normop{\cdot}$}{Operator norm}
\nomenclature[E,05]{$\normNuc{\cdot}$}{Nuclear norm}

\nomenclature[C,01]{$\manifolds{n}{d}{\rch_{\min}}$}{\hspace{26ex}Geometric model \hfill (\Cref{def:manifold_model})}

\nomenclature[C,02]{$\distributions{n}{d}{\rch_{\min}}{f_{\min}}{f_{\max}}{L}$}{\hspace{11ex}Statistical model \hfill (\Cref{def:manifold_model_distribution})}

\nomenclature[C,03]{$\manifoldspoint{n}{d}{\rch_{\min}}{0}$}{\hspace{19.4ex}Fixed point geometric model \hfill (\Cref{def:models_bounded})}
\nomenclature[C,04]{$\distributionspoint{n}{d}{\rch_{\min}}{f_{\min}}{f_{\max}}{L}{0}$}
{\hspace{4.4ex}Fixed point stat. model
\hfill (\Cref{def:models_bounded})}
\nomenclature[C,05]{$\manifoldsball{n}{d}{\rch_{\min}}{R}$}{\hspace{15ex}Bounding ball geometric model
\hfill (\Cref{def:models_bounded})}
\nomenclature[C,06]{$\distributionsball{n}{d}{\rch_{\min}}{f_{\min}}{f_{\max}}{L}{R}$}
{Bounding ball stat. model
\hfill (\Cref{def:models_bounded})}

\nomenclature[C,07]{$\rch_{\min}$}{Lower bound on the reach}
\nomenclature[C,08]{$f$}{Density with respect to the volume measure}
\nomenclature[C,09]{$f_{\min}$}{Lower bound on $f$}
\nomenclature[C,10]{$f_{\max}$}{Upper bound on $f$}
\nomenclature[C,11]{$L$}{Upper bound on the Lipschitz constant of $f$}
\nomenclature[C,12]{$R$}{Radius of enclosing ball}             

\nomenclature[F,01]{$\TV$}{Total variation distance \hfill (\Cref{def:tv})}
\nomenclature[F,02]{$\model$}{Generic model}
\nomenclature[F,03]{$(\metric, \dist)$}{Generic metric space}
\nomenclature[F,04]{$\theta : \model \to \metric$}{Generic parameter of interest}